\documentclass{psapm-l}

\usepackage{amssymb,amsmath,amscd}
\usepackage[mathbf,mathcal]{euler}
\usepackage{eucal}
\usepackage{dsfont}
\usepackage{epsfig}
\usepackage{graphicx}
\usepackage{mdwlist}
\usepackage{hyperref}

\usepackage{xy}
\xyoption{all}
\usepackage{floatflt}

\newtheorem{theorem}{Theorem}[section]
\newtheorem{lemma}[theorem]{Lemma}
\newtheorem{proposition}[theorem]{Proposition}
\newtheorem{corollary}[theorem]{Corollary}

\theoremstyle{definition}

\newtheorem{example}[theorem]{Example}

\theoremstyle{remark}
\newtheorem{remark}[theorem]{Remark}

\numberwithin{equation}{section}

\newcommand{\style}[1]{{\sc{#1}}\index{{#1}}}    

\def\sfopname#1{\operatorname{\mathsf{#1}}}

\newcommand{\sweetline}{
\begin{center}
\nointerlineskip\vspace{-0.04in}
        $\bullet$\hfill\rule{0.77\linewidth}{1.0pt}\hfill$\bullet$
\par\nointerlineskip\vspace{-0.01in}
\end{center}
}

\def\Flip#1{\raisebox{1ex}{\rotatebox{180}{#1}}}
\def\bigcheck#1{\rlap{\smash{\raisebox{2.25ex}{\Flip{$\hat{~}\,$}}}}#1}

\newcommand{\ie}{{\em i.e.}}
\newcommand{\eg}{{\em e.g.}}
\newcommand{\cf}{{\em cf.~}}

\newcommand{\real}{{\mathbb R}}

\newcommand{\nats}{{\mathbb N}}

\newcommand{\zed}{{\mathbb Z}}

\newcommand{\A}{{\mathcal A}}           
\newcommand{\abs}[1]{\left\vert{#1}\right\vert}

\newcommand{\AGr}{{\sf AGr}}           
\newcommand{\B}{{\mathcal B}}           %
\newcommand{\Bessel}{{\bf B}}       
\newcommand{\bfx}{\mbox{\boldmath $x$}}

\newcommand{\C}{{\mathcal C}}

\newcommand{\CF}{{\sfopname{CF}}}
\newcommand{\CFS}{{\sfopname{CFS}}}
\newcommand{\Chamber}{{\mathcal C}}

\newcommand{\Complex}{{\mathcal C}} 

\newcommand{\crit}{{\sfopname{Cr}}}        
\newcommand{\Crit}{{\sfopname{Cr}}}        

\newcommand{\dchifloor}{{\lfloor d\chi\rfloor}}
\newcommand{\dchiceil}{{\lceil d\chi\rceil}}

\newcommand{\Def}{{\sfopname{Def}}}            
\newcommand{\del}{\partial}

\newcommand{\diam}{{\sfopname{diam}}}
\renewcommand{\dim}{{\sfopname{dim}}}              
\newcommand{\Dual}{{\bf D}}
\newcommand{\Dualizing}{{\mathbb{D}}}

\newcommand{\Field}{{\mathbb F}}        

\newcommand{\Fourier}{{\bf F}}       
\newcommand{\FourierSato}{{\bf _\mu\! F}}       

\newcommand{\Grading}{\bullet}       
\newcommand{\Graph}{{\mathcal G}}       
\newcommand{\Gr}{{\sfopname{Gr}}}           

\newcommand{\Hom}{{\sfopname{Hom}}}

\newcommand{\id}{{\sfopname{Id}}}              
\newcommand{\Id}{{\sfopname{Id}}}              
\newcommand{\im}{{\sfopname{im}}}              

\newcommand{\Index}{{\mathcal I}}       
\newcommand{\inv}{^{-1}}
\newcommand{\isom}{\cong}               
\newcommand{\Kernel}{{K}}   
\renewcommand{\ker}{{\sfopname{ker\ }}}  

\newcommand{\Link}{{\Lambda}} 
\newcommand{\LMD}{{\sfopname{LMD}}}       

\newcommand{\Nodes}{{\mathcal N}}       
\newcommand{\norm}[1]{\left\|{#1}\right\|}
\newcommand{\Omin}{{\mathcal O}}
\newcommand{\one}{{\mathbf{1}}}
\newcommand{\Obs}{{\mathcal O}}         

\newcommand{\PD}{{\sfopname{PD}}}

\newcommand{\Presheaf}{{\mathcal P}}

\newcommand{\proj}{{\pi}}
\newcommand{\Radon}{{\bf R}}       

\newcommand{\Ring}{{\Bbbk}}
\newcommand{\RP}{{\mathbb{P}}}
\newcommand{\sect}{{\gamma}}  

\newcommand{\Sense}{{\mathcal S}}

\newcommand{\Sheaf}{{\mathcal F}}
\newcommand{\Snake}{{\delta}}

\newcommand{\Sphere}{{\mathbb S}}   

\newcommand{\supp}{{\sfopname{supp}}}      
\newcommand{\Support}{{\mathcal S}}     

\newcommand{\totvar}{{\sfopname{totvar}}}      
\newcommand{\Transform}{{\mathcal T}}         

\newcommand{\V}{{\mathcal V}}           
\newcommand{\vol}{{\sfopname{vol}}}      
\newcommand{\Wavelet}{{\mathcal W}}       

\begin{document}

\title{Euler Calculus with Applications to Signals and Sensing}

\author{Justin Curry}
\address{Department of Mathematics, University of Pennsylvania}
\email{jucurry@math.upenn.edu}

\author{Robert Ghrist}
\address{Departments of Mathematics and Electrical/Systems Engineering, University of Pennsylvania}
\email{ghrist@math.upenn.edu}
\thanks{All authors supported by DARPA DSO - HR0011-07-1-0002 and ONR N000140810668.}

\author{Michael Robinson}
\address{Department of Mathematics, University of Pennsylvania}
\email{mrobin@math.upenn.edu}

\subjclass[2000]{Primary ; Secondary }

\keywords{Euler characteristic, sheaves, cohomology, integral geometry, signal processing.}

\begin{abstract}
This article surveys the Euler calculus --- an integral calculus based on Euler characteristic --- and its applications to data, sensing, networks, and imaging.
\end{abstract}

\maketitle



\section{Introduction}
\label{sec:intro}

This work surveys the theory and applications of \style{Euler calculus}, an integral calculus built with the Euler characteristic, $\chi$, as a measure. While the theory engages an ethereal swath of topology (complexes, sheaves, and cohomology), the applications (to signal processing, data aggregation, and sensing) are concrete. These notes are meant to be read by both pure and applied mathematicians.

In the mid-1970s, MacPherson \cite{MacPherson} and Kashiwara \cite{Kashiwara} independently published seminal works on constructible sheaves. Their respective motivations appeared quite different. MacPherson was interested in  answering a conjecture of Deligne and Grothendieck on the theory of Chern classes for  complex algebraic varieties with singularities. Kashiwara had been following up on work from his 1970 thesis on the algebraic study of partial differential equations via D-modules. In each setting --- singularities, solutions, and obstructions ---  were understood using sheaf theory. Both MacPherson and Kashiwara made use of constructible sheaves and functions to provide algebraic characterizations of the local nature of singularities. Both provided index-theoretic formulae and developed a calculus relying on Euler characteristic. If it was in sheaf theory that Euler calculus was born, it had an earlier conception in geometry, going back at least to work of Blaschke \cite{Blaschke} and perhaps before, though neither he nor those who followed (Hadwiger, Groemer, Santalo, Federer, Rota, etc.) developed the full calculus that arose from sheaves.

The language of Euler calculus was slow to form and be appreciated. The short survey paper of Viro \cite{Viro} cited MacPherson and mentioned simple applications of the Euler integral to algebraic geometry. The short survey paper of Schapira \cite{Schapira:op} relied more upon Kashiwara \cite{Kashiwara,Kashiwara2} and was full of interesting formulations and applications --- it was explicitly motivated by the work of Guibas, Ramshaw, and Stolfi \cite{GRS} in computational geometry. In addition, the followup paper of Schapira \cite{Schapira:tom} contained a prescient application of Euler integral transforms to problems of tomography and reconstruction of images from the Euler characteristics of slicing data. In the decade following  these works, the language of Euler integration was used infrequently, hiding mostly in works on real-algebraic geometry and constructible sheaves and paralleled in the combinatorial geometry literature \cite{Groemer,Rota,Morelli,KR,Chen,Schanuel}.

There has been a recent renaissance of appreciation for Euler calculus. Some of this activity comes from the role of Euler characteristic as an elementary type of \style{motivic measure} in motivic integration \cite{CL,GZ}. Applications to algebraic geometry seem to be the primary impetus for interest in the subject \cite{Viro,McCP,Kiritchenko,GZ}. Parallel applications to integral geometry also have recently emerged. This survey comes out of the recent applications \cite{BG:enum,BG:PNAS} to problems in sensing, networks, signal processing, and data aggregation. These applications, presaged by Schapira \cite{Schapira:tom}, are poised to impact a number of problems of contemporary relevance.

It is somewhat remarkable that Schapira's deep insights saw little-to-no followup. One explanation is that in this, as in many other things, Schapira is ahead of his time. However, the language in which his results were couched --- sheaf theory --- was and is beyond the grasp of nearly all researchers in the application areas to which his results were directed. In the intervening decades, algebraic-topological methods have become more numerous and palatable to scientists and engineers, and the basics of homology and cohomology are now not so foreign outside of Mathematics departments. The same cannot yet be said of sheaf theory. It is with this in mind that this article exposits the Euler calculus from both explicit/applied and implicit/sheaf-theoretic perspectives.

The article begins with a concrete presentation of the Euler integral (\S\ref{sec:euler}-\ref{sec:integral}); continues with a gentle if brief introduction to the topology undergirding the Euler characteristic (\S\ref{sec:hom}-\ref{sec:seq}); then advances to the sheaf-theoretic view (\S\ref{sec:presheaves}-\ref{sec:dictionary}). With this full span of concepts established, this article turns to the many applications of the Euler calculus to engineering systems and data aggregation (\S\ref{sec:target}-\ref{sec:fubini}). We emphasize issues connected with implementation of the Euler calculus, including numerical approximation (\S\ref{sec:numerical}-\ref{sec:eucharis}), the use and inversion of integral transforms (\S\ref{sec:conv}-\ref{sec:wavelet}), and the extension to a real-valued theory (\S\ref{sec:Rval}-\ref{sec:defker}). This last development, motivated by the need to build an honest numerical analysis for the Euler calculus, flows back to the abstract sea from which Euler calculus was born, by yielding fundamental connections to Morse theory. The article concludes (\S\ref{sec:open}) with a collection of open problems and directions for further research.

The all-encompassing title of this work is a misnomer: our applications of the Euler calculus focus primarily on problems inspired by engineering systems and data. This is by no means the sole --- or even most important --- application of the Euler calculus. Other survey articles on applications of the Euler calculus (\eg, \cite{GZ}) detail applications that have no overlap with those of this paper at all: it is a broad subject. Several classical results in topology/geometry (the Gauss-Bonnet and Riemann-Hurwitz theorems) are both simplified and illuminated by an application of the Euler calculus. More applications are either implicit or emerging in the literature:
\begin{enumerate}
\item A careful reading of work by R. Adler and others on Gaussian random fields \cite{Adler,AT}, in which one wants to compute the expected Euler characteristic of an excursion set of a random smooth distribution over a domain, reveals the generous use of Euler calculus, without the language. Recent preprints \cite{Adler+,BoBo} incorporate this language.
\item The exciting work on persistent homology and associated barcodes for data \cite{CZ,Carlsson,G:barcodes} has a recent connection to Euler calculus. O. Bobrowski and M. Strom Borman \cite{BoBo} define the Euler characteristic of a barcode and relate this quantity to Euler integrals.
\item The work of S. Gal on configuration spaces \cite{Gal} and its applications to robotics \cite{Farberbook} resonates with similar constructions in the algebraic geometry literature \cite{GZ} and is clearly awaiting an Euler calculus reinterpretation.
\item Computational complexity of Euler characteristic computation for semialgebraic sets has been considered by Basu \cite{Basu1}, with recent work of \cite{Basu2} identifying Euler characteristic as an important obstruction in complexity theory related to the classical Toda theorem. This hints at the use of Euler integrals in computational complexity of constructible functions.
\end{enumerate}
It is to be hoped that other problems in Applied Mathematics and Statistics are equally amenable to simplification by means of this elegant and efficacious theory.

\specialsection*{\bf The Combinatorial Formulation}
\sweetline

\section{Euler characteristic}
\label{sec:euler}

The Euler characteristic is a generalization of counting. Given a finite discrete set $X$, the Euler characteristic is its cardinality $\chi(X)=\abs{X}$. If one connects two points of $X$ together by means of an edge (in a cellular/simplicial structure), the resulting space has one fewer component and the Euler characteristic is decremented by one. Continuing inductively, the Euler characteristic counts vertices with weight $+1$ and edges with weight $-1$. This intuition of counting connected components works at first; however, the addition of an edge producing a cycle does not change the count of connected components. To fill in such a cycle with a $2$-cell would return to the setting of counting connected components again, suggesting that $2$-cells be weighted with $+1$. This intuition of counting with $\pm 1$ weights inspires the following combinatorial definition of $\chi$.

Given a space $X$ and a decomposition of $X$ into a finite number of open cells $X=\coprod_\alpha\sigma_\alpha$, where each $k$-cell $\sigma_\alpha$ is homeomorphic to $\real^k$, the \style{Euler characteristic} of $X$ is defined as
\begin{equation}
    \chi(X) = \sum_\alpha (-1)^{\dim\ \sigma_\alpha} .
\end{equation}
For an appropriate class of {\em ``tame''} spaces (see \S\ref{sec:tame}), this quantity is well-defined and independent of the cellular decomposition of $X$. This combinatorial Euler characteristic is a homeomorphism invariant, but, as defined, is not a homotopy invariant for non-compact spaces, as, \eg, it distinguishes $\chi((0,1))=-1$ from $\chi([0,1])=1$. Among compact finite cell complexes, it is a homotopy invariant, as will be shown in \S\ref{sec:homot}.

\begin{example}
\begin{enumerate}
\item Euler characteristic completely determines the homotopy type of a compact connected graph.
\item Euler characteristic is also a sharp invariant among closed orientable 2-manifolds: $\chi=2-2g$, where $g$ equals the genus.
\item Any compact convex subset of $\real^n$ has $\chi=1$. Removing $k$ disjoint convex open sets from such a convex superset results in a compact space with Euler characteristic $1-k(-1)^n$.
\item The $n$-dimensional sphere $\Sphere^n$ has $\chi(\Sphere^n)=1+(-1)^n$.
\end{enumerate}
\end{example}

\section{Tame topology}
\label{sec:tame}

Euler characteristic requires some degree of finiteness to be well-defined. This finiteness, when enlarged to unions, intersections, and mappings of spaces, demands a behavior that is best described as {\em tameness}. Different mathematical communities have adopted different schemes for imposing tameness on subsets of Euclidean space. Computer scientists often focus on piecewise linear (PL) spaces, describable in terms of affine spaces and matrix inequalities. Combinatorial geometers sometimes use a generalization called \style{polyconvex} sets \cite{KR,Morelli}. Algebraic geometers tend to prefer \style{semialgebraic} sets --- subsets expressible in terms of a finite number of polynomial inequalities. Analysts prefer the use of analytic functions, leading to a class of sets called \style{subanalytic} \cite{Shiota}. Logicians have recently created an axiomatic reduction of these classes of sets in the form of an \style{o-minimal structure}, a term derived from {\em order minimal}, in turn derived from model theory. The text of Van den Dries \cite{vdD} is a beautifully clear reference.

For our purposes, an o-minimal structure $\Omin=\{\Omin_n\}$ (over $\real$) denotes a sequence of Boolean algebras $\Omin_n$ of subsets of $\real^n$ (families of sets closed under the operations of intersection and complement) which satisfies simple axioms:
\begin{enumerate*}
\item $\Omin$ is closed under cross products;
\item $\Omin$ is closed under axis-aligned projections $\real^{n}\to\real^{n-1}$;
\item $\Omin$ contains all algebraic sets (zero-sets of polynomials);
\item $\Omin_1$ consists of all finite unions of points and open intervals.
\end{enumerate*}
Elements of $\Omin$ are called \style{tame} or, more properly, \style{definable} sets. Canonical examples of o-minimal systems include semialgebraic sets and subanalytic sets. Note the last axiom: the finiteness imposed there is the crucial piece that drives the theory. (The open intervals need not be bounded, however.)

Given a fixed o-minimal structure, one can work with tame sets with relative ease: \eg, the union and intersection of two tame sets are again tame. A (not necessarily continuous) function between tame spaces is tame (or definable) if its graph (in the product of domain and range) is a tame set. A definable homeomorphism is a tame bijection between tame sets. To repeat: {\em definable homeomorphisms are not necessarily continuous.} Such a convention makes the following theorem concise:

\begin{theorem}[Triangulation Theorem \cite{vdD}]
\label{thm:Triangulation}
Any tame set is definably homeomorphic to a subcollection of open simplices in (the geometric realization of) a finite Euclidean simplicial complex.
\end{theorem}

The analogue of the Triangulation Theorem for tame mappings is equally important:

\begin{theorem}[Hardt Theorem \cite{vdD}]
\label{thm:Hardt}
Given a tame mapping $F\colon X\to Y$, $Y$ has a definable partition into cells $Y_\alpha$ such that $F\inv(Y_\alpha)$ is definably homeomorphic to $U_\alpha\times Y_\alpha$ for $U_\alpha$ definable, and $F$ restricted to this inverse image acts as projection to $Y_\alpha$.
\end{theorem}

The Triangulation Theorem implies that tame sets always have a well-defined Euler characteristic, as well as a well-defined dimension (the max of the dimensions of the simplices in a triangulation). The surprise is that these two quantities are not only topological invariants with respect to definable homeomorphism; they are complete invariants.

\begin{theorem}[Invariance Theorem \cite{vdD}]
\label{thm:EulerDim}
Two tame sets in an o-minimal structure are definably homeomorphic if and only if they have the same dimension and Euler characteristic.
\end{theorem}

This result reinforces the idea of a definable homeomorphism as a \style{scissors equivalence}. One is permitted to cut and rearrange a space (or, even, a mapping) with abandon. Recalling the utility of such scissors work in computing areas of planar sets, the reader will not be surprised to learn of a deep relationship between tame sets, the Euler characteristic, and integration.

\section{The Euler integral}
\label{sec:integral}

We now have all the tools at our disposal for an integral calculus based on $\chi$. The {\em measurable sets} in this theory are the tame sets in a fixed o-minimal structure. From Theorem \ref{thm:Triangulation} it follows that each such set has a well-defined Euler characteristic. The Euler characteristic, like a measure is \style{additive}:

\begin{lemma}
\label{lem:MVchi}
For $A$ and $B$ tame,
\begin{equation}
\label{eq:MVchi}
    \chi(A\cup B) = \chi(A) + \chi(B) - \chi(A\cap B).
\end{equation}
\end{lemma}
\begin{proof}
The result follows from (1) the Triangulation Theorem, (2) the well-definedness of $\chi$, and (3) counting cells.
\end{proof}

Additivity suggests converting $\chi$ into a (signed, integer-valued) measure $d\chi$ against which a indicator function is integrated in the obvious manner:
\begin{equation}
\label{eq:intchar}
    \int_X \one_A\,d\chi := \chi(A) ,
\end{equation}
The additivity of $\chi$ is, crucially, finite; the limiting process used in (standard) measure theory is certainly inapplicable. The natural set of {\em measurable functions} in this theory are definable functions with finite, discrete image. For the sake of convenience and clear presentation, we use the integers $\zed$ as range and invoke a compactness assumption. In what follows, $X$ will be assumed a tame space in a fixed o-minimal structure.

An integer-valued function $h\colon X\to\zed$ is \style{constructible} if all of its level sets $h\inv(n)\subset X$ are tame. Denote by $\CF(X)$ the set of bounded compactly supported constructible functions on $X$. The use of compact support is not strictly needed and is done for convenience. We may sometimes bend this criterion without warning: {\em caveat lector}. The integrable functions of Euler calculus on $X$ are, precisely, $\CF(X)$.

The \style{Euler integral} is defined to be the homomorphism $\int_X\colon\CF(X)\to\zed$ given by:
\begin{equation}
\label{eq:level}
    \int_X h\,d\chi = \sum_{s=-\infty}^{\infty}s\ \chi(\{h=s\}) .
\end{equation}
Each level set is tame and thus has a well-defined Euler characteristic. Alternately, one may use triangulation to write $h\in\CF(X)$ as $h=\sum_\alpha c_\alpha\one_{\sigma_\alpha}$, where $c_\alpha\in\zed$ and $\{\sigma_\alpha\}$ is a decomposition of $X$ into a disjoint union of open cells, yielding
\begin{equation}
\label{eq:simplices}
    \int_X h\,d\chi
    = \sum_\alpha c_\alpha\chi(\sigma_\alpha)
    = \sum_\alpha c_\alpha(-1)^{\dim\ \sigma_\alpha}
\end{equation}
That this sum is invariant under the decomposition into definable cells is a consequence of corresponding properties of the Euler characteristic.

Equation (\ref{eq:level}) is deceptive in explicit computations, since the level sets are rarely compact and the Euler characteristics must be computed with care. The following reformulation --- a manifestation of the Fundamental Theorem of Integral Calculus --- is more easily implemented in practice.

\begin{proposition}
\label{prop:computation}
For any $h\in CF(X)$,
\begin{equation}
\label{eq:eulerexcursion}
    \int_Xh\,d\chi = \sum_{s=0}^\infty \chi\{h>s\} - \chi\{h<-s\}
\end{equation}
\end{proposition}

\begin{proof}
Rewrite $h$ as
\begin{eqnarray*}
h
&=& \sum_{s=-\infty}^\infty s\one_{\{h=s\}}
\\
&=& \sum_{s=0}^\infty
    s(\one_{\{h\geq s\}}-\one_{\{h>s\}})
+ \sum_{s=0}^{-\infty}
    s(\one_{\{h\leq s\}}-\one_{\{h<s\}})
\\
&=& \sum_{s=0}^\infty \one_{\{h>s\}} - \one_{\{h<-s\}} ,
\end{eqnarray*}
where the last equality comes from telescoping sums.
\end{proof}

\begin{example}
In the example illustrated in Figure \ref{fig:simplexample}, the integral with respect to Euler characteristic is equal to $6$, since the integrand can be expressed as the sum of indicator functions over six (contractible) closed sets with unit Euler characteristics, though such a decomposition is not assumed given.
\begin{figure}[hbt]
\begin{center}
\includegraphics[angle=0,width=3.0in]{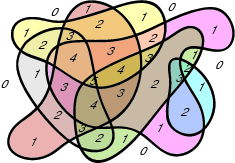}
\caption{A simple example of an integrand $h\in\CF(\real^2)$ whose Euler integral equals six: it is the sum of six characteristic functions over contractible discs in the plane.}
\label{fig:simplexample}
\end{center}
\end{figure}
\end{example}


Euler characteristic is like a measure in another aspect: it is multiplicative under cross products:
\begin{lemma}
\label{lem:chiproduct}
For $X$ and $Y$ definable sets,
\begin{equation}
    \chi(X\times Y) = \chi(X)\chi(Y).
\end{equation}
\end{lemma}
\begin{proof}
The product $X\times Y$ has a definable cell structure using products of cells from $X$ and $Y$. For cells $\sigma\subset X$ and $\tau\subset Y$, the lemma holds via the exponent rule since $\dim(\sigma\times\tau)=\dim\ \sigma + \dim\ \tau$. The rest follows from additivity.
\end{proof}

The assertion that $\int\,d\chi$ is an honest integral is supported by this fact and its corollary: the Euler integration theory admits a Fubini theorem. Calculus students know that $\int f(x,y)dA$ is computable via the double integral $\iint f(x,y)dx\,dy$. This familiar result is the image of a deeper truth about integrations and projections.

\begin{theorem}[Fubini Theorem]
\label{thm:fubini}
Let $F\colon X\to Y$ be a tame mapping between tame spaces. Then for all $h\in CF(X)$,
\begin{equation}
\label{eq:fubini}
    \int_X h\,d\chi = \int_Y \left( \int_{F^{-1}(y)}h(x)\,d\chi(x) \right) d\chi(y) .
\end{equation}
\end{theorem}
\begin{proof}
If $X$ is homeomorphic to $U\times Y$ and $F$ is projection to the second factor, the result follows from Lemma \ref{lem:chiproduct}. The Hardt Theorem
and additivity of the integral complete the proof.
\end{proof}

\specialsection*{\bf The (Co)Homological Formulation}
\sweetline
\vspace{0.1in}

As defined, both $\chi$ and the integral $\int\,d\chi$ are explicit, combinatorial, and concrete. Much of the depth and applicability of the theory derives from the pairing of these features with the algebraic-topological formulation. We provide a brief introduction to these methods, referring the interested reader to the better and more in-depth treatment in, \eg, \cite{Hatcher} for more details.

\section{Homology}
\label{sec:hom}

The reason for the topological invariance of ostensibly combinatorial $\chi$ lies in a particular \style{categorification} ---  an enrichment of the combinatorial Euler characteristic with (first) linear and (then) homological algebra, yielding an algebraic-topological means of counting and canceling features in a topological space. The simplest such lifting is via cellular homology.


Consider a finite cell complex, $X$: a space built from standard compact cells (simplices, discs, cubes, or other simple components) assembled by means of attaching maps along cell faces.\footnote{The reader for whom cell complexes are unfamiliar should consult, \eg, \cite{Hatcher,Koslov}. The reader for whom CW complexes are familiar should replace {\em cell complex} with {\em CW complex}.} The mechanics of counting used to define the Euler characteristic of $X$,
\[
\xymatrix{
\cdots  & \{\mbox{k-cells}\}  & \{\mbox{(k-1)-cells}\} & \cdots & \{\mbox{1-cells}\} & \{\mbox{0-cells}\}
},
\]
may be lifted to a sequence of vector spaces over a field $\Field$,
\[
\xymatrix{
\cdots  \quad & C_k \quad  & C_{k-1} & \cdots  & C_1 & \quad C_0
},
\]
where each $C_k$ is the $\Field$-vector space with basis the $k$-cells of $X$. This collection of vector spaces is then enriched to a sequence of linear transformations that detail how the cells are connected together:
\begin{equation}
\xymatrix{
\cdots \ar[r] & C_k \ar[r]^{\del_k} & C_{k-1} \ar[r]^{\del_{k-1}} &
\cdots \ar[r]^{\del_2} & C_1 \ar[r]^{\del_1} & C_0 \ar[r]^{\del_0} & 0
}.
\end{equation}
Here, the linear transformation $\del_k$ sends a $k$-cell of $X$ (a basis element of $C_k$) to the abstract sum of its oriented $(k-1)$-dimensional faces (a sum of basis elements in $C_{k-1}$). For $\Field=\Field_2$ the field of integers modulo 2, this is a simple sum of the faces, each with coefficient $1$. For other fields $\Field$, one must assign an orientation to all basis cells and compute the image of $\del$ with coefficients $\pm 1$, depending on orientation. The reader for whom homology is unfamiliar may want to work with $\Field_2$ coefficients at first, for which $-1=+1$: \cf the treatment in \cite{EdelsHarer}. Coefficients in $\real$ are motivated by, \eg, currents in electrical networks. The most generally informative coefficients are $\zed$, prompting the use of free abelian chain groups $C_k$ and homomorphisms $\del_k$. For simplicity of exposition, we will use fields and linear-algebraic constructs where possible.

The boundary of a boundary is null: $\del_{k-1}\circ\del_{k}=0$ for all $k$. As such, for all $k$, $\im\ \del_{k+1}$ is a subspace of $\ker  \del_k$. The \style{homology} of $\Complex$, $H_\Grading(\Complex)$, is a sequence of $\Field$-vector spaces built from the following features of $\del$. A \style{cycle} of $\Complex$ is a chain with empty boundary, \ie, an element of $\ker  \del$. Homology is an equivalence relation on cycles of $\Complex$. Two cycles in $Z_k=\ker  \del_k$ are said to be \style{homologous} if they differ by something in $B_k=\im\ \del_{k+1}$. The homology of $X$ is the sequence of quotients $H_k(X)$, for $k\in\nats$, given by:
\begin{equation}
\label{eq:homology}
    H_k(X) = Z_k/B_k = \left.\ker  \del_k\right/\im\ \del_{k+1} = \left.{\mbox{\rm cycles}} \right/{\mbox{\rm boundaries}} .
\end{equation}
To repeat: a \style{homology class} $[\alpha]\in H_k(X)$ is an equivalence class of cycles, two cycles being declared homologous if their difference is a boundary. Homology $H_k(X)$ inherits the sequential structure or \style{grading} ($k=0,1,2,\ldots$) of $\Complex$ and will be denoted $H_\Grading(X)$ when no particular grading is intended.

It takes some effort to get an intuition for homology, and several perspectives and examples are useful to this end. Homology is:
\begin{enumerate}
\item {\bf Multifarious:} The homology of a space $X$ can be defined in numerous ways, each of which counts some feature and cancels according to a boundary-like accounting. Homology theories which count cells [cellular], patches in a covering [\v{C}ech], critical points of a smooth map $f\colon X\to\real$ [Morse], maps of cells into $X$ [singular], and more, are, under the right `tameness' assumptions, isomorphic.
\item {\bf Functorial:} Homology applies not only to spaces, but to mappings between spaces. For $f\colon X\to Y$, there is an induced homomorphism (or linear transformation) $H(f):H_\Grading(X)\to H_\Grading(Y)$ that respects grading, identities, and composition of mappings. This $H(f)$ indicates how $f$ transforms cycles of $X$ over to cycles of $Y$.
\item {\bf Invariant:} The homology of $X$ is invariant not only under changes in cell decompositions, but also under homotopy equivalences: for $f\colon X\stackrel{\simeq}{\longrightarrow}Y$, $H(f)$ is an isomorphism.
\item {\bf Excisive:} Given $A\subset X$ a subcomplex, there is a \style{relative homology} $H_\Grading(X,A)$ given by taking the quotients $C_k(X,A)=C_k(X)/C_k(A)$ and using the induced boundary maps $\del_k$. This relative homology has the effect of collapsing the subcomplex $A$ to an abstract point: $H_\Grading(X,A)\cong H_\Grading(X/A,\{A\})$.
\end{enumerate}

\section{Cohomology}
\label{sec:cohom}

An algebraic mirror image of homology will prove salient to defining the Euler characteristic on tame but non-compact spaces. A \style{cochain complex} is a sequence $\Complex=(C^\Grading,d)$ of $\Field$-vector spaces $C^k$ (or free abelian groups) and linear transformations (homomorphisms) $d^k:C^k\to C^{k+1}$ with the property that $d^{k+1}\circ d^k=0$ for all $k$. The arrows are reversed:
\[
\xymatrix{
0 \ar[r] & C^0 \ar[r]^{d^0} & C^{1} \ar[r]^{d^1} &
\cdots \ar[r]^{d^{k-1}} & C^{k} \ar[r]^{d^k} & C^{k+1} \ar[r]^{d+1} & \cdots
}.
\]
The \style{cohomology} of a cochain complex is,
\begin{equation}
    H^k(\Complex) = \left.\ker  d^k\right/\im\ d^{k-1} .
\end{equation}
Cohomology classes are equivalence classes of \style{cocycles} in $\ker  d$. Two cocylces are \style{cohomologous} if they differ by a \style{coboundary} in $\im\ d$.

The simplest means of constructing cochain complexes is to dualize a chain complex $(C_\Grading,\del)$. Given such a complex (with vector spaces over $\Field$), define $C^k=C_k^\vee$, the dual space of linear functionals $C_k\to\Field$ (or, in the group setting, the group $\Hom$ of homomorphisms to the coefficient group). The coboundary $d$ is the adjoint of the boundary $\del$, so that
\[
    d\circ d=\del^\vee\circ\del^\vee=(\del\circ\del)^\vee=0^\vee=0.
\]
The coboundary operator $d$ can be presented explicitly: $(df)(\tau) = f(\del\tau)$. For $\tau$ a $k$-cell, $d$ implicates the \style{cofaces} --- those $(k+1)$-cells having $\tau$ as a face.

All of the constructs of homology --- exact sequences, functoriality, excision --- pass naturally to cohomology by means of dualization. One further construct is necessary for use in Euler calculus. Given a cochain complex $\Complex^\Grading$ on a space $X$, consider the subcomplex $\Complex^\Grading_c$ of cochains which are compactly supported. The coboundary map restricts to $d:C_c^k\to C^{k+1}_c$ with $d^2=0$, yielding a well-defined \style{cohomology with compact supports}, $H^\Grading_c(X)$. This cohomology satisfies the following:
\begin{enumerate*}
\item   $H^k_c(\real^n)=0$ for all $k$ except $k=n$, in which case it is of rank $1$.
\item   $H^\Grading_c$ is not a homotopy invariant, but is a proper homotopy (and hence a homeomorphism) invariant.
\item   $H^\Grading_c(X)\isom H^{\Grading}(X)$ for $X$ compact.
\end{enumerate*}

\begin{example}
\label{ex:Poincare}
Homology and cohomology are closely related by a classical duality that springs from Poincar\'e's original conceptions of homology. For a simple example, consider a compact surface with a polyhedral cell structure, and let $\Complex$ be the cellular chain complex with $\Field_2$ coefficients. There is a dual polyhedral cell structure, yielding a chain complex $\overline{\Complex}$, where the dual cell structure places a vertex in the center of each original $2$-cell, has $1$-cells transverse to each original $1$-cell, and, necessarily, has as its $2$-cells neighborhoods of the original vertices. Each dual $2$-cell is an $n$-gon, where $n$ is the degree of the original $0$-cell. Note that these cell decompositions are truly dual and have the effect of reversing the dimensions of cells: $k$-cells generating $C_k$ are in bijective correspondence with $(2-k)$-cells (on a surface) generating a modified cellular chain group $\overline{C}_{2-k}$. The dual complex $\overline{\Complex}^\bullet$ consisting of $\overline{C}^k=\overline{C}_k^\vee$ and $\overline{d}=\overline{\del}^\vee$ entwines with $\Complex_\bullet$ in a diagram:
\[
\xymatrix{
0 \ar[r] & C_{2}\ar[d]^{\cong} \ar[r]^{\del} & C_1 \ar[d]^{\cong}\ar[r]^{\del} & C_{0} \ar[d]^{\cong}\ar[r] & 0
\\
0 \ar[r] & \overline C^0 \ar[r]^{\overline d} & \overline C^1 \ar[r]^{\overline d} & \overline C^2 \ar[r] & 0
}
\]
The vertical maps are isomorphisms and, crucially, the diagram is commutative. The equivalence of singular and cellular (co)homology implies that, for a compact surface with $\Field_2$ coefficients, $H_k\cong H^{2-k}$. Such a result fails for a non-compact surface (\cf, $\real^2$), unless one switches to $H^\Grading_c$. With this, and using a similar proof as in the 2-d case, one obtains:

\begin{figure}[hbt]
\begin{center}
\includegraphics[angle=0,width=2in]{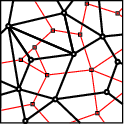}
\caption{A polyhedral cell structure on a surface and its dual mirror homological-cohomological duality.}
\label{fig:poincaredual}
\end{center}
\end{figure}
%
\begin{theorem}[Poincar\'{e} duality]
For $M$ an $n$-manifold, there is a natural isomorphism $\PD:H_c^k(M;\Field_2)\to H_{n-k}(M;\Field_2)$.
\end{theorem}
It follows that for $M$ a compact $n$-manifold, $H_k(M;\Field_2)\cong H_{n-k}(M;\Field_2)$. The coefficients may be modified at the expense of worrying about orientability of the manifold and the torsional elements: see, \eg, \cite{Hatcher} for details. Poincar\'e duality does not hold in general for non-manifolds, but generalizations abound. In \S\ref{sec:shops} and \S\ref{sec:dictionary}, a very broad and powerful extension will be used in defining Euler integrals.
\end{example}

\section{Homotopy invariance of $\chi$}
\label{sec:homot}


Invariance of Euler characteristic can be pursued from many angles; \eg, invariance under a refinement of cell structure can be ascertained through simple combinatorics. To get the full invariance under homotopy equivalence requires some basic homological algebra. This comes from lifting the notion of Euler characteristic and homology from a cell complex to an arbitrary (finite) chain complex $\Complex=(C_{\Grading},\del)$, a sequence of free abelian groups $C_k$ and homomorphisms $\del_k$ satisfying $\del_{k-1}\del_k=0$ for all $k$. For such a sequence, the homology $H_\Grading(\Complex)$ is, as before, $\ker\ \del / \im\ \del$ and the Euler characteristic is given via:
\begin{equation}
\label{eq:ChainEuler}
    \chi(\Complex) = \sum_{k}(-1)^k \dim\ C_k .
\end{equation}
Note that this is independent of the maps $\del_k$ and thus is sensible for any sequence of vector spaces. The reason for the alternating sum is to take advantage of cancelations that permit the following.

\begin{lemma}
\label{lem:ChainEuler}
The Euler characteristic of a chain complex and its homology are identical, when both are defined.
\end{lemma}
\begin{proof}
Since $H_k=Z_k/B_k$, one has
\[
    \dim\ Z_k = \dim\ H_k + \dim\ B_k .
\]
Since $\del^2=0$, one has $C_k/Z_k\cong B_{k-1}$, so that
\[
    \dim\ C_k = \dim\ Z_k + \dim\ B_{k-1} ,
\]
from which it follows that
\[
    \dim\ C_k = \dim\ H_k + \dim\ B_k + \dim\ B_{k-1}.
\]
Multiply this equation by $(-1)^k$; the sum over $k$ telescopes.
\end{proof}

Applying this to the chain complex for cellular homology and invoking the homotopy invariance of homology yields:
\begin{corollary}
\label{cor:EulerChar}
For $X$ a finite compact cell complex,
\begin{equation}
\label{eq:homeuler}
    \chi(X) = \sum_k(-1)^k\dim\ H_k(X;\real).
\end{equation}
\end{corollary}

\begin{corollary}
\label{cor:Eulerinvt}
$\chi$ is a homotopy invariant among finite compact cell complexes.
\end{corollary}

This is helpful is computing integrals with respect to Euler characteristic: one need merely count holes as opposed to counting cells.

\begin{example}
The homological formulation of $\chi$ not only gives invariance --- it also provides a potentially simple means of computing Euler integrals. Consider the integrand displayed in Figure \ref{fig:eulerintex}. Constructing the appropriate triangulation and computing Euler characteristic may be involved. By combining Equation (\ref{eq:eulerexcursion}) with the homological definition, it is an easy matter (in this example at least) to compute the Euler integral.
\begin{figure}[hbt]
\begin{center}
\includegraphics[angle=0,width=5.0in]{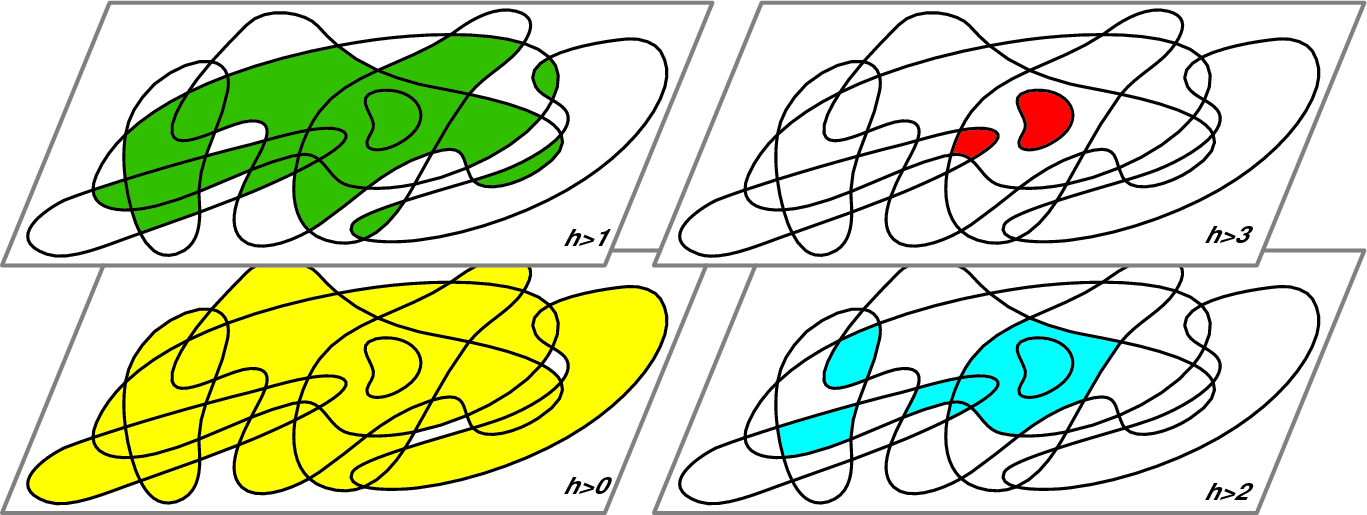}
\caption{A simple example of Euler integration. Instead of triangulating, computing homology ranks of upper excursion sets is a simple method to determine the Euler integral.}
\label{fig:eulerintex}
\end{center}
\end{figure}
\end{example}

\section{Sequences}
\label{sec:seq}

There is more to homology than simply counting cells or holes. The algebraic constructs built to support homology mirror the topological spaces implicated \cite{GelfandManin}. If a chain complex $\Complex=(C_{\Grading},\del)$ is the analogue of a topological space or cell complex, then the analogue of a continuous map is a \style{chain map} --- a map $\varphi_\Grading:C_\Grading\to C_\Grading'$ between chain complexes that is a homomorphism on chain groups respecting the grading and commuting with the boundary maps. This is best expressed in the form of a \style{commutative diagram}:
\begin{equation}
\label{eq:chainmap}
\displaystyle
\xymatrix{
\cdots \ar[r] & C_{n+1} \ar[r]^{\del}\ar[d]^{\varphi_\Grading} & C_n \ar[r]^{\del}\ar[d]^{\varphi_\Grading} & C_{n-1} \ar[r]^{\del}\ar[d]^{\varphi_\Grading} & \cdots
\\
\cdots \ar[r] & C'_{n+1} \ar[r]^{\del'} & C'_n \ar[r]^{\del'} & C'_{n-1} \ar[r]^{\del'} & \cdots
}
\end{equation}
Commutativity means that homomorphisms are path-independent in the diagram: $\varphi_\Grading\circ\del=\del'\circ\varphi_\Grading$. Since chain maps send neighbors to neighbors, the appropriate generalization of a homeomorphism to chain complexes is therefore an invertible chain map --- one which is an isomorphism for all $C_k\to C'_k$. Clearly, a homeomorphism $f\colon X\to Y$ induces a chain map $f_\Grading\colon\Complex_X\to\Complex_Y$ from a cellular chain complex on $X$ to the complex of the induced cell structure on $Y$: in this case, $f_\Grading$ is an isomorphism and, clearly, $H_{\Grading}(X)\cong H_{\Grading}(Y)$. Furthermore, if $f\simeq g\colon X\to Y$ are homotopic maps, then the induced homomorphisms are also isomorphisms: $H(f)\cong H(g)$.

Of equal importance is the algebraic analogue of a nullhomologous space.  A chain complex $\Complex=(C_{\Grading},\del)$ is \style{exact} when its homology vanishes: $\ker  \del_n=\im\ \del_{n+1}$ for all $n$. The most important examples of exact sequences are those relating homologies of various spaces and subspaces. The critical technical tool for the generation of such uses the technique of weaving an exact thread through a loom of chain complexes.

\begin{theorem}[Snake Lemma]
\label{thm:snake}
Any short exact sequence of chain complexes
\[
\displaystyle
\xymatrix{
0 \ar[r] & \A_{\Grading} \ar[r]^{i_\Grading} & \B_{\Grading} \ar[r]^{j_\Grading} & \C_{\Grading} \ar[r] & 0
} ,
\]
induces the \style{long exact sequence}:
\begin{equation}
\label{eq:LES}
\displaystyle
\xymatrix{
\ar[r] & H_n(\A_{\Grading}) \ar[r]^{H(i)} & H_n(\B_{\Grading}) \ar[r]^{H(j)} & H_n(\C_{\Grading}) \ar[r]^{\Snake} & H_{n-1}(\A_{\Grading}) \ar[r]^(.7){H(i)} &
} ,
\end{equation}
where $\Snake$ is the induced \style{connecting map}. Moreover, the long exact sequence is functorial: a commutative diagram of short exact sequences and chain maps
\[
\displaystyle
\xymatrix{
0 \ar[r] & \A_{\Grading} \ar[r] \ar[d]^{f_\Grading}& \B_{\Grading} \ar[r] \ar[d]^{g_\Grading} & \C_{\Grading} \ar[r] \ar[d]^{h_\Grading} & 0
\\
0 \ar[r] & \tilde\A_{\Grading} \ar[r] & \tilde\B_{\Grading} \ar[r] & \tilde\C_{\Grading} \ar[r] & 0
}
\]
induces a commutative diagram of long exact sequences
\begin{equation}
\label{eq:naturality}
\displaystyle
\xymatrix{
\ar[r] & H_n(\A_{\Grading}) \ar[r] \ar[d]^{H(f)}& H_n(\B_{\Grading}) \ar[r] \ar[d]^{H(g)} & H_n(\C_{\Grading}) \ar[r]^{\Snake} \ar[d]^{H(h)} & H_{n-1}(\A_{\Grading}) \ar[r] \ar[d]^{H(f)} &
\\
\ar[r] & H_n(\tilde\A_{\Grading}) \ar[r] & H_n(\tilde\B_{\Grading}) \ar[r] & H_n(\tilde\C_{\Grading}) \ar[r]^{\Snake} & H_{n-1}(\tilde\A_{\Grading}) \ar[r] &
}
\end{equation}
\end{theorem}

The exact sequence of chain complexes means that there is a short exact sequence for each dimension, and these short exact sequences fit into a commutative diagram with respect to the boundary operators. For details on the definition of $\Snake$, see \cite{Hatcher,GelfandManin,Koslov} or any standard reference.

\begin{example}[LES of pair]
Given $A\subset X$ a subcomplex, the following short sequence is exact:
\[
\displaystyle
\xymatrix{
0 \ar[r] & C_{\Grading}(A) \ar[r]^{i_\Grading} & C_{\Grading}(X) \ar[r]^{j_\Grading} & C_{\Grading}(X,A) \ar[r] & 0
} ,
\]
where $i:A\hookrightarrow X$ is an inclusion and $j:(X,\varnothing)\hookrightarrow(X,A)$ is an inclusion of pairs. This yields the \style{long exact sequence of the pair} $(X,A)$:
\begin{equation}
\label{eq:LESpair}
\displaystyle
\xymatrix{
\ar[r] & H_n(A) \ar[r]^{H(i)} & H_n(X) \ar[r]^(.4){H(j)} & H_n(X,A) \ar[r]^{\Snake} & H_{n-1}(A)\ar[r] & },
\end{equation}

The connecting map $\Snake$ takes a relative homology class $[\alpha]\in H_n(X,A)$ to the homology class $[\del\alpha]\in H_{n-1}(A)$.
\end{example}

\begin{example}[Mayer-Vietoris]
Another important sequence is derived from a decomposition of $X$ into subcomplexes $A$ and $B$. Consider the short exact sequence
\begin{equation}
\label{eq:SESMV}
\xymatrix{
0 \ar[r] & C_{\Grading}(A\cap B) \ar[r]^(.4){\phi_\Grading} & C_{\Grading}(A)\oplus C_{\Grading}(B) \ar[r]^(.6){\psi_\Grading} & C_{\Grading}(A+B) \ar[r] & 0
,
}
\end{equation}
with chain maps $\phi_\Grading:c\mapsto(c,-c)$, and $\psi_\Grading:(a,b)\mapsto a+b$. The term on the right, $C_{\Grading}(A+B)$, consists of those chains which can be expressed as a sum of chains on $A$ and chains on $B$. In cellular homology with $A, B$ subcomplexes, $C_{\Grading}(A+B)\cong C_{\Grading}(X)$, resulting in the \style{Mayer-Vietoris sequence}:
\begin{equation}
\label{eq:LESMV}
\displaystyle
\xymatrix{
\ar[r] & H_n(A\cap B) \ar[r]^(.35){H(\phi)} & H_n(A)\oplus H_n(B) \ar[r]^(.7){H(\psi)} & H_n(X) \ar[r]^(.35){\Snake} & H_{n-1}(A\cap B) \ar[r] &
}
\end{equation}
\end{example}

These exact sequences allow one to re-derive the combinatorial properties of the Euler characteristic. The following come from careful use of Lemmas \ref{lem:ChainEuler}, exactness, and the two exact sequences above:

\begin{proposition}
\label{prop:EulerChar}
For $A,B\subset X$ compact and tame,
\begin{eqnarray}
    \chi(A\cup B) &=& \chi(A) + \chi(B) - \chi(A\cap B) \\
    \chi(X - A) &=& \chi(X) - \chi(A)
\end{eqnarray}
\end{proposition}

The combinatorial Euler characteristic $\chi$ is equal to $\chi_c$, the Euler characteristic defined via cohomology with compact supports:
\begin{lemma}
For $X$ tame  and locally compact,
\begin{equation}
   \chi(X) := \sum_{\sigma}(-1)^{\dim\ \sigma} = \chi_c := \sum_{k=0}^\infty(-1)^k\dim\ H^k_c(X;\real)
\end{equation}
\end{lemma}
\begin{proof}
If $U$ is an open subset of a locally compact $X$ we have the following long exact sequence in cohomology:
\begin{equation}
\displaystyle
\xymatrix{
\ar[r] & H^n_c(U) \ar[r] & H^n_c(X) \ar[r] & H^n_c(X-U) \ar[r] &
}
,
\end{equation}
whence, from Lemma \ref{lem:ChainEuler}, $\chi_c(X)=\chi_c(U)+\chi_c(X-U)$. Invoking the Triangulation Theorem, fix a triangulation of $X$ with $U$ a cell of maximal dimension $k$ in $X$. This is necessarily open and so applying the above result allows us to peel off a sum of $(-1)^k$. Repeat this for all (finitely many!) $k$-dimensional cells, and note that $X-(U_1\cup\cdots\cup U_{n(k)})$ is a tame set of dimension $k-1$. Repeating the argument inductively with respect to dimension concludes the proof.
\end{proof}

This is the impetus for using a cohomological version of the Euler characteristic --- it allows one to manage non-compact spaces with a cell structure. Cohomological formulations quickly become complicated if local compactness is not assumed (\eg, a compact triangle with one open face removed) \cite{Beke}. A better method still is via categorifying a constructible set to a richer data structure: a constructible \style{sheaf}.

\specialsection*{\bf The Sheaf-Theoretic Formulation}
\sweetline
\vspace{0.1in}

Tame sets possess triangulations of such virtue as to permit a combinatorial Euler characteristic. This simple counting procedure foreshadows the  richer (co)homological formulation from algebraic topology, which, in turn, returns the principle of independence of triangulation. At yet higher elevations, this principle transmutes into invariance of the Euler integral by relating it to an intrinsic quantity --- the Euler characteristic of a constructible sheaf. Far from being useless abstraction, this lifting provides both tools and insights for the applications to come. We proceed, therefore, via a brief introduction to sheaves.

It is of general interest to aggregate local observations into globally coherent ones. The connection between local and global properties has been captured in many famous mathematical theorems attached to equally famous names: Gauss-Bonnet, Poincar\'e-Hopf, Riemann-Roch, Atiyah-Singer and so on. In many of these results local quantities are integrated to produce interesting invariants. These invariants, traditionally algebraic in nature, can, for instance, provide obstructions to extending local constructions globally.

Sheaf theory provides a systematic means of describing and deriving many local-to-global results. This intricate machinery, often coupled with and cast in the language of categories, has garnered a reputation for being abstruse and abstract. The authors are of the opinion that sheaves are extremely useful as data management tools and will gradually shed their intimidating appearance as incarnations of sheaves relevant to Applied Mathematics are found. One such incarnation is given by Euler integration and constructible functions. Although Euler calculus can be appreciated separately from sheaves, the connection between the two, expressed most clearly in the work of Schapira \cite{Schapira:op}, serves to both deepen the foundations of Euler calculus and elucidate sheaves.

\section{Presheaves}
\label{sec:presheaves}

The most basic assignment of local data is captured in a structure called a \style{presheaf}. A presheaf is a map $\Presheaf$ from open sets of $X$ to (say) abelian groups in a manner that respects {\em restriction to subsets}. Specifically, $\Presheaf$ is a contravariant functor from open sets under inclusion to a category (of, in this paper, abelian groups). For a less abstract definition, the following will suffice: given $V\subset U$ open in $X$, there is an induced homomorphism $\Presheaf(V)\stackrel{r}{\leftarrow}\Presheaf(U)$ that respects composition. Namely, the map $\Presheaf(U)\stackrel{r}{\leftarrow}\Presheaf(U)$ is the identity transformation, and for $W\subset V\subset U$ one has

\[
    \xymatrix{
        W\,\,\ar@{^{(}->}[dr]\ar@{^{(}->}[rr] & & \,\,U & &
        \Presheaf(W) & & \Presheaf(U)\ar[ll]_{r}\ar[dl]^{r}\\
        & \,\,V\,\,\ar@{^{(}->}[ur] & & & & \Presheaf(V)\ar[ul]^{r} &
    } .
\]

Note that, as in the case of cohomology and other \style{contravariant} theories, the algebraic maps reverse the direction of the topological maps. One calls $\Presheaf(U)$ the group of \style{sections} of $\Presheaf$ over $U$; an element $\sect\in\Presheaf(U)$ is called a section. The presheaf condition implies that one can restrict sections uniquely. Although we define presheaves valued in abelian groups, one is free to use any algebraic data one might like to use: vector spaces, rings, modules, algebras, etc. Even more general data may be assigned: sets, spaces, spectra and even categories.\footnote{Presheaves of categories are often called \style{prestacks} and have functors as restriction maps.}

The use of open sets for the assignees of data has many advantages over assigning data directly to points. However, pointwise data assignment can be recovered by a limiting process. For $x\in X$ and $\{U_i\}$ a nested sequence of open neighborhoods converging to $x$ (that is, any neighborhood of $x$ contains $U_i$ for all $i$ sufficiently large), one has a sequence of groups of sections and homomorphisms
\[
    \cdots
    \stackrel{r}{\longrightarrow}\Presheaf(U_{i-1})
    \stackrel{r}{\longrightarrow}\Presheaf(U_i)
    \stackrel{r}{\longrightarrow}\Presheaf(U_{i+1})
    \stackrel{r}{\longrightarrow}\cdots
\]
The \style{stalk} of $\Presheaf$ at $x$, $\Presheaf_x$, is the group of equivalence classes of sequences $(\sect_i)_i$ with $\sect_{i+1}=r(\sect_i)$ for $\Presheaf(U_i)\stackrel{r}{\rightarrow}\Presheaf(U_{i+1})$ and all $i$; and where two such sequences are equivalent if they eventually agree: $[(\sect_i)]=[(\sect_i')]$ if and only if $\sect_i=\sect_i'$ for $i>N$, for some $N$. By using a more implicit definition --- the stalk is the colimit over $\Presheaf(U_x)$ for all open sets $U_x$ containing $x$ --- it follows that the stalk is a group and that $\Presheaf_x$ is independent of the system of neighborhoods chosen to limit to $x$.

\section{Sheaves}
\label{sec:sheaves}

A \style{sheaf} is a presheaf $\Sheaf$ that respects {\em gluings} as well as restrictions. Consider two open subsets $U, V\subset X$.
A presheaf is a sheaf if and only if for all $U$ and $V$ open in $X$, and sections $\sect_U$, $\sect_V$ of $U$ and $V$ which agree on the overlap $U\cap V$, there exists a {\em unique} section of $U\cup V$ which agrees with $\sect_U$ and $\sect_V$ on the components. More generally, we require that this property hold for a family of open sets $\{U_i\}$ for $i$ in some potentially large indexing set.

The canonical example of a non-sheaf presheaf is the presheaf that assigns to an open set $U$ the group of constant functions $f:U\to\zed$, with restriction maps defined to be the identity. To see this is not a sheaf, consider two disjoint open sets $U$ and $V$ and the functions $f_U(x)=5$ for $x\in U$ and $f_V(x)=7$ for $x\in V$. Since the sets have empty intersection they vacuously agree there, despite the absence of a constant function $f:U\cup V\to\zed$ such that $f|_V=f_V$ and $f|_U=f_U$. One can mitigate the problem by instead assigning to any open set $U$ the group of {\em locally constant functions}: this defines a sheaf, sometimes called $\widetilde{\zed}$. To an open subset it assigns continuous functions $f:U\to\zed$, which, by the discrete topology on $\zed$, are constant on connected components. Consequently the group of everywhere defined sections $\widetilde{\zed}(X)$ are exactly functions that are constant on connected components of $X$. The rank of this free abelian group calculates the number of connected components and is the most basic topological invariant of a space. This sheaf has even  more information about $X$ and has cohomological data exactly equal to the familiar cohomology of $X$ with $\zed$ coefficients. This indicates another interpretation of sheaves --- a sheaf is a generalized coefficient system for computing cohomology.

Sheaves, like simplicial complexes, are defined abstractly but have a geometric realization reminiscent of covering spaces. To any sheaf $\Sheaf$ is associated a topological space (also denoted $\Sheaf$) and a projection $\proj\colon\Sheaf\to X$. This \style{\'etale space} is the disjoint union of all stalks $\Sheaf_x$, $x\in X$, outfitted with a topology so that:
\begin{enumerate*}
    \item Fibers $\proj\inv(x)$ are discrete spaces with the structure of an abelian group.
    \item Algebraic operations in the fibers (addition/multiplication) yield continuous maps of the space $\Sheaf$.
    \item The projection $\proj$ is a local homeomorphism: for each small neighborhood $U\subset X$, $\proj\inv(U)$ is a disjoint union of homeomorphic copies of $U$ in $\Sheaf$.
\end{enumerate*}

A \style{section} of $\Sheaf$ over an open set $U\subset X$ is a map $\sect:U\to\Sheaf$ with $\proj\circ\sect = \id$. Denote by $\Gamma(U,\Sheaf)$ the group of sections over $U$, defined in the obvious manner (via pointwise operations). The \style{global sections} of $\Sheaf$ are denoted $\Gamma(X,\Sheaf)$. The utility of the \'etale space perspective is that, as a space, the sheaf possesses global features whose detection and computation reveal important structure.

Note that the \'etale space of a sheaf provides an alternate path to defining sheaves. More precisely, let us temporarily term the pair $(\Sheaf,\proj)$ consisting of an \'etale space with its projection map an \style{\'etale sheaf}. To such an \'etale sheaf we could associate a presheaf $\Gamma(\Sheaf)$ that assigns to an open subset $U$ in $X$ the group of sections of the \'etale sheaf $\Gamma(U,\Sheaf)$. It is a theorem (see \cite{Rotman} Thm 5.68) that $\Gamma(\Sheaf)$ is actually a sheaf. Conversely, as noted above, starting with an abstract sheaf $\Sheaf$ and associating to it an \'etale sheaf will produce an isomorphic sheaf after applying $\Gamma(-)$, \ie, $\Sheaf(U)\cong\Gamma(U,\Sheaf)$ for all $U$. Consequently one may use these two perspectives interchangeably. This also explains why in the literature one may see the terms {\em sheaf of sections} and {\em sheaf of groups} to refer to the same sheaf. The former term emphasizes a model of sheaves as \'etale sheaves whereas the latter term emphasizes the abstract assignment model we first described.

\section{Sheaf operations}
\label{sec:shops}

In all the examples of sheaves considered above, certain common features resolve into canonical constructions. These are the beginnings of sheaf theory --- a means of working with data over spaces in a platform-independent manner. For example, every presheaf can be turned into a sheaf through a process called \style{sheafification} \cite{Bredon}. Conversely, every sheaf gives rise to a presheaf simply by forgetting the extra structure of sheaf, and these operations are, to a reasonable degree, inverses. Other operations in sheaf theory involve pushing forward and pulling back sheaves based on maps of the base spaces, defining the cohomology of a sheaf, and various constructions related to duality and the distinction between compactly and non-compactly supported sections. Instead of detailing these operations in full (as in, \eg, \cite{Godemont,Bredon}), a few brief highlights with applications are given.

\vspace{0.1in}
\noindent
{\bf Morphisms:} A \style{morphism} $F\colon\Sheaf\to\Sheaf'$ of (pre)sheaves over $X$ is a collection of maps $F(U)\colon\Sheaf(U)\to\Sheaf'(U)$ that are compatible with the restrictions internal to both $\Sheaf$ and $\Sheaf'$. One can pass to the stalk at $x$ and get a well-defined map $F_x\colon\Sheaf_x\to\Sheaf'_x$. This allows one to address the notions of subsheaves, quotient sheaves, and exact sequences of sheaves on a stalk-by-stalk basis: the machinery is built to integrate based on stalk data. This becomes particularly relevant when defining cohomology for sheaves (\S\ref{sec:shcohom}), but is also requisite for understanding most other sheaf operations.

\vspace{0.1in}
\noindent
{\bf Direct image:} Assume a continuous map $F\colon X\to Y$ of spaces. For $\Sheaf_X$ a sheaf over $X$, the \style{direct image} or \style{pushforward} of $F$ is the sheaf $F_*\Sheaf_X$ on $Y$ defined by $F_*\Sheaf_X(V) := \Sheaf_X(F\inv(V))$, for $V\subset Y$ open. A nice connection between the direct image and the group of global sections should be noted: If $p\colon X\to\star$ is the constant map then the pushforward of $\Sheaf$ from $X$ to $\star$ returns a sheaf which is a single group $p_*\Sheaf_X=\Gamma(X,\Sheaf)$. This example will be fundamental in understanding the theoretical connection between sheaf theory and the Euler calculus when $X$ is compact. To handle the situation of non-compact $X$ we use instead the direct image with compact supports
\[
       F_!\Sheaf_X(V) := \{s\in\Sheaf_X(F\inv(V)) \, ; \, F|_{\supp(s)} \, \mathrm{proper}\}
\]
Note that $\supp(s)$ is just the set of points $x$ such that $s_x\neq 0$ --- the image of $s$ in the stalk at $x$. When $F$ is already a proper continuous map we have $F_!=F_*$. If we again take $p$ the constant map $p_!\Sheaf=\Gamma_c(X,\Sheaf)$ is the group of global sections with compact support.

\vspace{0.1in}
\noindent
{\bf Inverse image:} With $F\colon X\to Y$ as before, one can pull back a sheaf from the codomain to the domain. The \style{inverse image} or \style{pullback} of a sheaf $\Sheaf_Y$ on $Y$ is the sheaf on $X$ defined (using the \'etale interpretation) as
\[
    F^*\Sheaf_Y := \{ (x,y) \in X\times\Sheaf_Y \, : \, F(x)=\proj_Y(y) \} ,
\]
where $\proj$ is the canonical projection to $Y$. It is more awkward to define $F^*$ in terms of presheaves, since one cannot guarantee that the forward image $F(U)$ of an open set $U\subset X$ is open in $Y$. This can be remedied by defining
\[
   F^*\Sheaf_Y(U) := \varinjlim_{F(U)\subset V} \Sheaf(V)
\]
and this will be a presheaf. One can then sheafify this to define a pullback of sheaves without the \'etale perspective. One should note that the inverse image of $\Sheaf$ by $F$ is sometimes written $F\inv\Sheaf$ as in \cite{KS}: our notation is closer to that of  \cite{Iversen}.

\vspace{0.1in}
\noindent
{\bf Duality:} The canonical operations in sheaf theory are entwined. The morphisms between a direct and inverse image of a
sheaf are related by a categorical \style{adjunction} --- a form of duality. Let $\Sheaf$ and $\mathcal{G}$ be arbitrary sheaves on $X$ and $Y$ respectively and consider $F\colon X\to Y$ a continuous map; then
\begin{equation}
        \Hom(F^*\mathcal{G},\Sheaf)\cong\Hom(\mathcal{G},F_*\Sheaf),
\end{equation}
where $\Hom$ denotes the group of sheaf morphisms. The motivation for calling this a form of duality comes from a related pair of operations: direct and inverse image with compact supports. As we will see later an adjunction between these two operations provides a vast generalization of Poincar\'e duality called \style{Verdier duality} \cite{Schurmann,Dimca}.

\section{Sheaf Cohomology}
\label{sec:shcohom}

There are two approaches to defining the cohomology of sheaves, each with particular advantages. One, defined using \style{\v Cech cohomology}, is computationally the most tractable and will be familiar to topologists. The other, which uses \style{injective resolutions}, is very useful for proving theorems and will be familiar to the reader comfortable with homological algebra.

\vspace{0.1in}
\noindent
{\bf \v{C}ech cohomology:}
Fix $X$ a base space, $\{U_i\}$ an open cover of $X$, and a sheaf $\Sheaf$ over $X$. Now choose $k+1$ sets from $\{U_i\}$, say $U_{j_0},\ldots,U_{j_k}$. Denote their intersection by $U_{j_0,\ldots,j_k}$, and let $\sect_{j_0,\ldots,j_k}$ be an element of $\Sheaf(U_{j_0,\ldots,j_k})$. An assignment that takes every ordered tuple of $k+1$ sets from $\{U_i\}$ and specifies a section over the intersection is called a \style{$k$-cochain}. The group of all such assignments,
\[
	C^k(\{U_i\},\Sheaf)=\prod_{j_0<\ldots<j_k}\Sheaf(U_{j_0,\ldots,j_k}) ,
\]
is the $k^{th}$ cochain group of $\Sheaf$. For each $k\geq 0$ there is such a group and successive groups are connected via a differential described as follows: to define $d\alpha$ given $\alpha\in C^{k-1}(\{U_i\},\Sheaf)$, one must specify a value on every $k+1$-fold intersection. For each nonempty intersection of $k+1$ sets, one forgets each factor, say $U_{j_i}$, in the intersection, yielding a $k$-fold intersection with value $\alpha(U_{j_0,\ldots,\hat{j}_i,\ldots,j_k})$. Repeating this for $0\leq i\leq k$, one obtains $k+1$ sections that can be combined after restriction:
\[
	d\alpha(U_{j_0,\ldots,j_k})=\sum_{i=0}^k (-1)^i\alpha(U_{j_0,\ldots,\hat{j}_i,\ldots,j_k})|_{U_{j_0,\ldots,j_k}}.
\]
Notice that the restriction is necessary, as the $\alpha(U_{j_0,\ldots,\hat{j}_i,\ldots,j_k})$ live in different abelian groups and so an alternating sum would not be well-defined, but the ability to restrict to a common abelian group is part of the definition of a sheaf. The situation is perhaps more easily visualized diagrammatically.

\[
	\xymatrix{
	\Sheaf(U_{j_1,\ldots,j_k}) \ar[rd]_r & \cdots\,\, \Sheaf(U_{j_0,\ldots,\hat{j}_i,\ldots,j_k}) \,\,\cdots \ar[d] & \Sheaf(U_{j_0,\ldots,j_{k-1}}) \ar[ld]^r \\
	& \Sheaf(U_{j_0,\ldots,j_k}) &
	}
\]

It is left to the reader to verify that $d^2=0$ and conclude that the cohomology $H^\Grading(C^{\Grading}(\{U_i\},\Sheaf),d)$ is well-defined. For suitably nice covers $\{U_i\}$ this cohomology computes the cohomology $H^\Grading(X;\Sheaf)$ of the sheaf $\Sheaf$. For example, to compute the cohomology of the constant sheaf $\widetilde{\zed}$ one may simply employ a \style{good cover}, \ie, one whose intersections are contractible\footnote{For non-good covers, \v{C}ech cohomology computes the $E_2$ page in a spectral sequence that converges to the full sheaf cohomology.}, and \v{C}ech cohomology will compute $H^\Grading(X;\zed)$.

\vspace{0.1in}
\noindent
{\bf Injective resolutions:}
The \v Cech approach has the advantage of being computable, but, like many constructions, one would like to know if the output of said computation is independent of choices: the choice of cover, the choice of ordering, and so on. In order to put sheaf cohomology on intrinsic ground one appeals to injective resolutions. The basic operation one uses when dealing with sheaf cohomology is the global section functor:
\[
	\Gamma(X,-):Sh(X)\to Ab
\]
This is the map that takes a sheaf $\Sheaf$ on $X$ and sends it to the abelian group $\Sheaf(X)$. The fact that this map is functorial comes from the observation that a map of sheaves includes maps on the level of global sections. More is true: if $F\colon\Sheaf\to\mathcal{G}$ is a map of sheaves that is injective on stalks, then $F_X\colon\Sheaf(X)\to\mathcal{G}(X)$ is injective (\cite{Iversen} II.2.2). The same statement is not true for surjections.

A canonical example of this asymmetry is given by the exponential sequence
\[
	\xymatrix{0\to \widetilde{\zed} \to \mathcal{O}_X \to \mathcal{O}^*_X \to 0}
\]
where $X=\mathbb{C}-{0}$ is the punctured complex plane, $\mathcal{O}_X$ is the sheaf of holomorphic functions, and $\mathcal{O}^*_X$ nonvanishing holomorphic functions. The first map simply includes continuous integer valued functions into holomorphic ones. The second map takes a function $f\colon X\to \mathbb{C}$ and sends it to $g=e^{2\pi i f}$, which is manifestly nonzero. It is surjective on stalks because locally a non-zero function has a well defined logarithm, but this is not true globally. In particular, $g=z$ is a global section of $\mathcal{O}^*_X$ not hit by the exponential map.

This sequence is exact despite the fact that it is {\em not} an exact sequence of groups for every open set $U$. Saying that a sequence of sheaves is exact is a local statement. One could take the following theorem as a definition of exactness:

\begin{theorem}[\cite{Iversen} II.2.6]
A sequence of sheaves and sheaf morphisms
	\[
	    \cdots
	    \stackrel{}{\longrightarrow}\Sheaf^{i-1}
	    \stackrel{}{\longrightarrow}\Sheaf^i
	    \stackrel{}{\longrightarrow}\Sheaf^{i+1}
	    \stackrel{}{\longrightarrow}\cdots
	\]
is exact if and only if the corresponding sequence of groups and group homomorphisms is exact for all $x\in X$
	\[
	    \cdots
	    \stackrel{}{\longrightarrow}\Sheaf^{i-1}_x
	    \stackrel{}{\longrightarrow}\Sheaf^i_x
	    \stackrel{}{\longrightarrow}\Sheaf^{i+1}_x
	    \stackrel{}{\longrightarrow}\cdots
	\]
\end{theorem}

The key feature of taking global sections is that it preserves exactness at the left endpoint. Namely, if
\[
    0
    \stackrel{}{\longrightarrow}\Sheaf
    \stackrel{}{\longrightarrow}\mathcal{I}^0
    \stackrel{}{\longrightarrow}\mathcal{I}^1
    \stackrel{}{\longrightarrow}\cdots
\]
is an exact sequence of sheaves then
\[
    0
    \stackrel{}{\longrightarrow}\Sheaf(X)
    \stackrel{}{\longrightarrow}\mathcal{I}^0(X)
    \stackrel{}{\longrightarrow}\mathcal{I}^1(X)
    \stackrel{}{\longrightarrow}\cdots
\]
is no longer an exact sequence of abelian groups, but $\Sheaf(X)$ still injects into $\mathcal{I}^0(X)$. Since $\mathcal{I}^{\Grading}(X)$ is no longer exact it has potentially interesting cohomology. When each of the $\mathcal{I}^i$ are injective (see \cite{Iversen} for a good introduction) the cohomology of this complex is taken as the {\em definition} of sheaf cohomology of $\Sheaf$:
\[
	H^i(X,\Sheaf) := H^i(\mathcal{I}^{\Grading}(X))
\]
Injective sheaves can be quite mysterious. One may take on faith that an injective sheaf has the property that $r\colon\mathcal{I}(X)\to\mathcal{I}(U)$ is surjective for every open set $U$. Sheaves with this property are called \style{flabby}. In a flabby sheaf, one can always extend a section by zero to obtain a global section.

Since cohomology can also be calculated locally this data can be again assembled into a sheaf - called the $i^{th}$ \style{cohomology sheaf}. This is the sheaf $\mathcal{H}^i\Sheaf$ associated to the presheaf
\[
	\mathcal{H}^i\Sheaf: U \mapsto H^i(U,\Sheaf)=H^i(\mathcal{I}^{\Grading}(U)).
\]

\vspace{0.1in}
\noindent
{\bf Higher direct images:}
For a conceptual cartoon one may compare an injective resolution of a sheaf with a Taylor expansion of a function. Each term in the resolution has nicer algebraic (analytic) properties than the sheaf (function), and with infinitely many terms the resolution can replace the
sheaf (function). The idea that a sheaf can be replaced by its injective resolution was instrumental to the development of the now standard technology of \style{derived categories} and functors. The sweeping generalization which Grothendieck introduced \cite{tohoku} allowed this construction of sheaf cohomology to be imitated with other operations --- direct image, restriction and so on --- to produce the \style{higher direct image} and other generalizations.

The higher direct image or right-derived direct image is the one most important to the sheaf-theoretic definition of the Euler integral. Take $F\colon X\to Y$ continuous and $\Sheaf$ a sheaf on $X$. Then the higher direct image is a sequence of sheaves $R^iF_*\Sheaf$ (one for each $i$) associated to the presheaves on $Y$
\[
	R^iF_*\Sheaf: V\mapsto H^i(F^{-1}(V),\Sheaf)
\]
One can define the higher direct image more intrinsically by replacing $\Sheaf$ with its injective resolution $\mathcal{I}^{\Grading}$ and applying $F_*$ level-wise to $\mathcal{I}^{\Grading}$ then
\[
	RF_*\Sheaf:=F_*\mathcal{I}^{\Grading} \qquad \mathrm{and} \qquad
	R^iF_*\Sheaf:=\mathcal{H}^i(F_*\mathcal{I}^{\Grading}).
\]
This $RF_*\Sheaf$ is sometimes called the total right-derived pushforward and should be thought of as a complex of sheaves such that when you take the $i$th cohomology sheaf you get the $i$th right-derived pushforward $R^iF_*\Sheaf$. In the special case of a constant map $p\colon X\to\star$,
\[
       R^ip_*\Sheaf=H^i(X,\Sheaf) ,
\]
the $i^{th}$ cohomology group of $\Sheaf$ on $X$. Similarly $RF_!$ is what you would expect --- replace $\Sheaf$ with an
injective resolution and apply $F_!$ level-wise to that. With $F=p\colon X\to\star$, we have
\[
       R^ip_!\Sheaf=H^i_c(X,\Sheaf) ,
\]
the $i^{th}$ cohomology group with compact supports.

Since the combinatorial Euler characteristic only agrees with the intrinsic cohomological Euler characteristic when compact supports are used, an analogous construction is needed in the sheaf world. This construction is a form of duality.

\vspace{0.1in}
\noindent
{\bf Verdier Duality:}
One of the appealing features of sheaf theory is the access to the higher operations of the derived category. In this section we ask that our sheaves have some extra structure: for each open set $U$, $\Sheaf(U)$ is a $\Ring$-module for a fixed commutative ring $\Ring$ and the restriction maps respect this structure --- they are $\Ring$-linear. In the language of Iversen \cite{Iversen}, these are $\Ring$-sheaves.

For a simple example, consider an arbitrary $\Ring$-module $D$. This can be made into a sheaf $\widetilde{D}$ on $X$ by taking the inverse image $p^*D$ where $p\colon X\to\star$ is the constant map. Alternatively, this is the sheafification of the constant presheaf that assigns $D$ to every open set.

\begin{lemma}
\[
       \Hom(\widetilde{D},\Sheaf)\cong\Hom(D,\Gamma(X,\Sheaf))
\]
\end{lemma}
\begin{proof}
This is immediate from the first adjunction recorded
       \[
               \Hom(p^*D,\Sheaf)\cong\Hom(D,p_*\Sheaf)
       \]
\end{proof}

As a special case of this consider when $D=\Ring$. The isomorphism becomes much more useful to understand
\[
       \Hom(\tilde{\Ring},\Sheaf)\cong\Hom(\Ring,\Gamma(X,\Sheaf))\cong\Gamma(X,\Sheaf)
\]
where the last isomorphism comes from noting that any homomorphism is determined uniquely by where it sends $1\in \Ring$. This is one instance of how adjunctions reveal information about sections.

As hinted earlier, adjunctions between other functors provide useful insight into sheaves and their sections.

\begin{theorem}[Global Verdier Duality]
Suppose $F\colon X\to Y$ is a continuous map of locally compact spaces, and $\Sheaf$ and $\mathcal{G}$ are sheaves on $X$ and $Y$ then there exists an operation $F^!$ such that
\[
       \Hom(RF_!\Sheaf,\mathcal{G})\cong\Hom(\Sheaf,F^!\mathcal{G})
\]
where the left and right hand sides live in the derived category of sheaves on $Y$ and $X$ respectively.
\end{theorem}

\begin{example}[Classical duality]
Consider $F=p\colon X\to\star$ the constant map, $\Sheaf=\tilde{\Ring}$, $\mathcal{G}=\Ring$. Verdier duality then says
\[
       \Hom(Rp_!\tilde{\Ring},\Ring)\cong\Hom(\tilde{\Ring},p^!\Ring).
\]
The second term is isomorphic to $\Gamma(X,p^!\Ring)$ as noted above. The first term is understood as usual: we replace $\tilde{\Ring}$ with its injective resolution $\mathcal{I}^{\Grading}$ and by applying $Rp_!$ we are just taking compact supports. Our reinterpreted isomorphism then becomes
\[
       \Hom(\Gamma_c(X,\mathcal{I}^{\Grading}),\Ring)\cong \Gamma(X,p^!\Ring)
\]
This may still seem discouraging. The $\Hom$-term is still mysterious. To clarify, assume that $\Ring$ is a field and so everything in sight is a vector space and complexes thereof. Consequently $\Hom$ just indicates linear maps to the ground field $\Ring$, i.e. linear functionals. So when the complex of vector spaces,
\[
\xymatrix{
    \ar[r] &\Gamma_c(X,{I}^{k-1})\ar[r] & \Gamma_c(X,{I}^k)\ar[r] & \Gamma_c(X,{I}^{k+1}) \ar[r] &
    }
\]
is dualized we must take the transposes of the connecting linear maps, thereby reversing them
\[
\xymatrix{ & \ar[l] \Gamma_c(X,{I}^{k-1})^\vee & \ar[l]
\Gamma_c(X,{I}^k)^\vee & \ar[l]  \Gamma_c(X,{I}^{k+1})^\vee & \ar[l].}
\]
Notice that the grading in the dual complex decreases by one, as in homology, but we nevertheless want to take the cohomology of both sides. Fortunately, there is a general notational device that allows one to think of homology as cohomology, merely by allowing grading in negative degrees. We therefore place the dual of the $k^{th}$ vector space in degree $-k$, the dual of $(k-1)^{th}$ vector space in degree $1-k$ and so on. Notice that the dual map now augments $-k$ to $1-k$, so this agrees with cohomological conventions.\footnote{The interested reader is encouraged to consult \cite{GelfandManin} or any other good text on homological algebra for notational conventions.}

In the case where $X$ is an oriented $n$-manifold $p^!\Ring=\tilde{\Ring}[n]$ the constant sheaf concentrated in degree $-n$. If we replace $\widetilde{\Ring}[n]$ with its injective resolution, $\mathcal{I}^{\Grading}[n]$, the isomorphism above is then interpreted as a chain complex isomorphism relating
\[
        \Gamma_c(X,\mathcal{I}^{\bullet})^{\vee}
    \quad {\rm and } \quad
        \Gamma(X,\mathcal{I}[n]^{-\bullet})
        =
        \Gamma(X,\mathcal{I}^{n-\bullet}),
\]
which, after taking cohomology, yields,
\[
       H^k_c(X,\tilde{\Ring})^\vee=H^{n-k}(X,\tilde{\Ring})
\]
the classical Poincar\'e duality.
\end{example}

The operation $p^!$ is curious in that it associates to any complex of vector spaces a sheaf on $X$. This is called the \style{dualizing complex} $\omega_X:=p^!\Ring$ and it relies only on the topological properties of $X$. The dualizing complex allows one to associate to any sheaf of real vector spaces $\Sheaf$ a dual sheaf $\Dualizing\Sheaf$ that comes from treating sheaf morphisms from $\Sheaf$ to $\omega_X$ as a sheaf --- it is the sheaf that assigns to an open subsets $U$ the group of sheaf morphisms from $\Sheaf|_U$ to $\omega_X|_U$.

The above example can be generalized as follows:
\begin{theorem}[\cite{Dimca} Thm 3.3.10]
\begin{equation}
       H_c^k(X,\Sheaf)^\vee\cong H^{-k}(X,\Dualizing\Sheaf).
\end{equation}
and thus
\begin{equation}
       \chi_c(X,\Sheaf)=\chi(X,\Dualizing\Sheaf).
\end{equation}
\end{theorem}
This will be useful below because it implies that the switch between compactly supported and regular cohomology is a switch between a sheaf and its dual.

\section{Constructible Functions and Sheaves: A Dictionary}
\label{sec:dictionary}

Persistent readers shall now be rewarded for their efforts. Euler calculus lifts to sheaf theory, whence its operations emanate. Recall that a constructible function is an integer-valued function $f\colon X\to\zed$ such that the level sets form a tame partition of $X$. The definition of a \style{constructible sheaf} is similar: $X$ has a tame partition $\{X_{\alpha}\}$ such that $\Sheaf|_{X_{\alpha}}$ is locally constant. A sheaf $\Sheaf$ on a topological space $X$ is locally constant if there is an open cover $\{U_i\}$ of $X$ such that $\Sheaf|_{U_i}$ is isomorphic to the constant sheaf. By {\em constructible sheaf}, we will mean sheaves valued in real vector spaces of finite dimension: more general data values are possible.

The connection between constructible functions and sheaves is given by Euler characteristic, but the connection is subtle. Suppose one is given a constructible sheaf $\Sheaf$, then one could define a constructible function $h$ as follows
\[
	h(x):=\chi(\Sheaf)(x)=\dim (\Sheaf_x).
\]
For example, suppose that $A\subset X$ is a tame set. Then consider the constant sheaf $\widetilde{\real}_A$ supported on $A$, \ie, it is the constant sheaf with stalk $\real$ for points on $A$ and zero elsewhere. In this case
\[
	\chi(\widetilde{\real}_A)(x)=\mathbf{1}_A(x)
\]
is the indicator function on $A$.

This definition will have trouble producing constructible functions that take values in the negative integers. To fix this we consider a bounded complex of constructible sheaves with zero maps connecting each sheaf:
\[
    \Sheaf^{\Grading}\quad =\quad \cdots
    \stackrel{0}{\longrightarrow}\Sheaf^{i-1}
    \stackrel{0}{\longrightarrow}\Sheaf^i
    \stackrel{0}{\longrightarrow}\Sheaf^{i+1}
    \stackrel{0}{\longrightarrow}\cdots
\]
A wider class of constructible functions can be achieved simply by setting
\[
	h(x):=\chi(\Sheaf^{\Grading})(x)=\sum_i (-1)^i\dim (\Sheaf^i_x).
\]

In the case where the maps in $\Sheaf^{\Grading}$ are not identically zero one must instead use the \style{local Euler-Poincar\'e index}
\[
	\chi(\Sheaf^{\Grading})(x)=\sum_i (-1)^i\dim(\mathcal{H}^i\Sheaf^{\Grading}_x)
\]
which generalizes the zero-map case because there $\mathcal{H}^i\Sheaf^{\Grading}=\Sheaf^i$.

To see that every constructible function is obtained via this lifted perspective, start with $h\in\CF(X)$ and consider the partition of $X$ via level sets $h\inv(n)$, $n\in\zed$. Consider the following two-term complex of constructible sheaves $\widetilde{\real}_h$ (we omit the grading for simplicity):
\begin{equation}\label{sheaf-fn}
	0
	\stackrel{0}{\longrightarrow}\bigoplus_{n<0}\widetilde{\real}_{X_n}^{-n}
    \stackrel{0}{\longrightarrow}\bigoplus_{n>0}\widetilde{\real}_{X_n}^{n}
	\stackrel{0}{\longrightarrow}
	0 .
\end{equation}
The second term occupies the zeroth slot in the complex (such an assignment is necessary for the alternating sum to make sense) and $\widetilde{\real}^n$ indicates the constant sheaf of $\real^n$; thus $\chi(\Sheaf)=h(x)$.

One can now define Euler integration via sheaf theory: let $h\in\CF(X)$, for $X$ compact, and $\Sheaf$ be a complex of sheaves on $X$ such that $h(x)=\chi(\Sheaf)(x)$. Let $p\colon X\to\star$ be the constant map. The \style{Euler integral} of $h$ is
\begin{equation}
		\int_X h\,d\chi:=\chi(Rp_*\Sheaf)
\end{equation}
When $X$ is not compact then we use the right-derived pushforward with compact supports:\footnote{The reader may note with pleasure that here, as in calculus class, integration over a noncompact domain can be ill-defined if one is not careful.}
\begin{equation}
               \int_X h\,d\chi:=\chi(Rp_!\Sheaf)
\end{equation}

The higher push forward of $\Sheaf$ to a point takes an injective resolution of $\Sheaf\to\mathcal{I}^{\Grading}$ and applies the direct image level-wise to $\mathcal{I}^{\Grading}$. From the definition of direct image in \S\ref{sec:shops},
\[
	p_*\mathcal{I}=\mathcal{I}(X) ,
\]
is the group (or vector space) of global sections. Thus we have a complex of vector spaces
\[
	\mathcal{I}^0(X)
	\stackrel{}{\longrightarrow} \mathcal{I}^1(X)
	\stackrel{}{\longrightarrow} \mathcal{I}^2(X)
	\stackrel{}{\longrightarrow} \cdots
\]
If one calculates the local Euler-Poincar\'e index over the one and only point $\star$, one gets
\begin{eqnarray*}
	\chi(Rp_*\Sheaf)(\star) &=& \sum_i (-1)^i\dim(\mathcal{H}^i\mathcal{I}^{\Grading}(X)) \\
	&=& \sum_i (-1)^i\dim(H^i\mathcal{I}^{\Grading}(X)) \\
	&=& \sum_i (-1)^i\dim(H^i(X,\Sheaf)) \\
	&=:&\chi(X,\Sheaf)
\end{eqnarray*}
in the compact case. In the general case the Euler integral is actually computing $\chi_c(X,\Sheaf)$. The moral of this section is this:
{\em the Euler integral of a constructible function is the Euler characteristic of its associated constructible sheaf}.
	
The right hand side of the definition may now seem less mysterious, but the connection to the original definition of the Euler integral may still be elusive. The connection is given by o-minimal geometry and properties of constant sheaves like $\widetilde{\real}_A$. As noted for $\widetilde{\zed}$, the cohomology of the constant sheaf coincides with normal cohomology of the space. Specifically
\[
	H^i(X,\widetilde{\real}_A)=H^i(A;\real)
\]
and consequently
\[
	\chi(X,\widetilde{\real}_A)
    = \sum_k(-1)^k\dim H^k(A;\real) .
\]
Note that this is {\it not} the same combinatorial Euler characteristic $\chi$ used in Euler integration, because compact supports are not used for $H^\Grading$ above.

Euler characteristic and its properties are adaptable to any situation with complexes involved. In particular for two sheaves (or complexes of sheaves) $\Sheaf$, $\mathcal{G}$ over $X$
\begin{eqnarray*}
	\chi(X,\Sheaf\oplus\mathcal{G})&=&\chi(X,\Sheaf)+\chi(X,\mathcal{G}) \\
	\chi(X,\Sheaf\otimes\mathcal{G})&=&\chi(X,\Sheaf)\cdot\chi(X,\mathcal{G})
\end{eqnarray*}
The same equations hold for $\chi_c$, the compactly-supported cohomological Euler characteristic of \S\ref{sec:cohom}. As a particular instance of the above one has
\[
	\chi(X,\widetilde{\real}^n_A)=n\chi(X,\widetilde{\real}_A)=n\chi(A)
\]
To see this result and more we prove the following lemma:
\begin{lemma}[\cite{Dimca}, 2.5.4 pg. 49]
For $\Sheaf$ a locally constant sheaf on $X$ valued in $\real$-vector spaces,
\[
	\chi(X,\Sheaf)=n\chi(X)
    \quad {\rm and } \quad
    \chi_c(X,\Sheaf)=n\chi_c(X)
\]
where $n$ is the (locally constant) stalk dimension.
\end{lemma}
\begin{proof}
By hypothesis there is a cover of $\{U_i\}$ of $X$ so that $\Sheaf|_{U_i}\cong \widetilde{\real}^n_{U_i}$. Let us first assume that the cover $\{U_i\}$ consists of a single open set $U$. In this case
\[	
	\Sheaf\cong\widetilde{\real}^n\cong\bigoplus^n_{i=1}\widetilde{\real}.
\]
Consequently an injective resolution of $\widetilde{\real}^n$ can be constructed as follows: Take a resolution of $\widetilde{\real}$, say $\mathcal{I}^{\Grading}$. Now define an injective sheaf $\mathcal{J}^i=\oplus_{k=1}^n\mathcal{I}^i.$ The map from $\mathcal{J}^i\to\mathcal{J}^{i+1}$ is the n-fold direct sum (block matrix) of the map $\mathcal{I}^i\to\mathcal{I}^{i+1}$. This shows that for each $i$
\[	
	H^i(X,\widetilde{\real}^n)=\oplus H^i(X,\widetilde{\real})
\]
and thus
\[
	\chi(X,\widetilde{\real}^n)=n\chi(X,\widetilde{\real}).
\]
Now assume that $X=U\cup V$ where $\Sheaf|_U$ and $\Sheaf|_V$ are isomorphic to the constant sheaf. To handle this case, we look to the \style{Mayer-Vietoris sequence for sheaves} for inspiration. For $U,V$ open there is a long exact sequence in sheaf cohomology:
\begin{equation}
\label{eq:LESMVS}
\displaystyle
\xymatrix{
\ar[r] & H^n(U\cup V,\Sheaf) \ar[r] & H^n(U,\Sheaf)\oplus H^n(V,\Sheaf) \ar[r] & H^n(U\cap V,\Sheaf) \ar[r] &.
}
\end{equation}
The proof of this comes from the fact that if we replace $\Sheaf$ with its injective (flabby) resolution $\mathcal{J}^{\Grading}$, then the following complex of vector spaces (or groups) is exact:
\[
 \xymatrix{
0 \ar[r] & \Gamma(U\cup V,\mathcal{J}^{\Grading}) \ar[r] & \Gamma(U,\mathcal{J}^{\Grading})\oplus \Gamma(V,\mathcal{J}^{\Grading}) \ar[r] & \Gamma(U\cap V,\mathcal{J}^{\Grading}) \ar[r] & 0.
}
\]
As a consequence of this short exact sequence of complexes we get for free that
\[
	\chi(U,\Sheaf)+\chi(V,\Sheaf)=\chi(U\cup V,\Sheaf)+\chi(U\cap V,\Sheaf).
\]
Since we have assumed that $\Sheaf$ is isomorphic to $\widetilde{\real}^n$ on the open sets $U,V, U\cap V$ the Euler characteristic formula above yields for $U\cup V=X$
\[
	\chi(X,\Sheaf)=n\chi(U)+n\chi(V)-n\chi(U\cap V)=n\chi(X).
\]
Applying this argument inductively to a countable cover yields the result. To repeat the proof for $\chi_c$, note that Mayer-Vietoris for compact supports is also valid, except for the reversal of arrows, as per the general covariant behavior of compact supports.
\end{proof}

It follows that, for a constructible --- that is, piecewise-constant --- sheaf, the Euler characteristic of the sheaf is simple to compute on constant patches. One more argument is necessary to show that the Euler characteristic of a constructible sheaf is an Euler integral.

\begin{theorem}[\cite{Dimca}, Theorem 4.1.22; \cite{Schurmann}, Lemma 2.0.2]
For $\Sheaf$ a constructible sheaf and $h(x)=\chi(\Sheaf)(x)$, the Euler characteristic of $\Sheaf$ equals the integral of $h$ with respect to $d\chi$:
\begin{equation}
    \chi_c(X,\Sheaf)=\int_X h\,d\chi.
\end{equation}
\end{theorem}
\begin{proof}
Let $X_i$ be the partition of $X$ into tame sets so that $\Sheaf|_{X_i}$ is locally constant. Since each $X_i$ is tame, consider a triangulation of each $X_i$ into cells $S_{ij}$. Let $U$ be the union of the (disjoint) top dimensional cells, which is necessarily open. In analogy with the compactly supported cohomology we have the following long exact sequence:
\begin{equation}
\displaystyle
\xymatrix{
\ar[r] & H^n_c(U,\Sheaf) \ar[r] & H^n_c(X,\Sheaf) \ar[r] & H^n_c(X - U,\Sheaf) \ar[r] & }
\end{equation}
and so
\[
	\chi_c(X,\Sheaf)=\chi_c(U,\Sheaf)+\chi_c(X-U,\Sheaf)=\sum_i n_i\chi_c(U_i)+\chi_c(X-U,\Sheaf)
\]
Here the $U_i$ are the open cells belonging to each of the $X_i$ (there may be none), and $\Sheaf$ has potentially different rank on each of $X_i$. This is the inductive step. The theorem holds if $X$ is zero dimensional. Induction allows us to excise the top dimensional cells and pass from dimension $n$ to $n-1$. Thus,
\[
	\chi_c(X,\Sheaf)=\sum_{i,j}n_{ij}\chi_c(S_{ij})=\sum_i n_i\chi_c(X_i)=\int h\,d\chi.
\]
\end{proof}

Now duality makes another nice appearance. Recall that to any sheaf we can associate a dual sheaf $\Dualizing\Sheaf$ such that $\chi_c(X,\Sheaf)=\chi(X,\Dualizing\Sheaf)$, so the Euler integral is also computing the Euler characteristic of the dual sheaf. If one considers the associated dual function
\[
       \Dualizing h(x)=\chi(\Dualizing\Sheaf)(x)
\]
and integrates this instead, one obtains:
\[
       \int_X \Dualizing h\,d\chi=\chi_c(X,\Dualizing\Sheaf)=\chi(X,\Dualizing\Dualizing\Sheaf)
\]
As duality is involutive --- $\Dualizing\Dualizing\Sheaf\cong\Sheaf$ --- the Euler characteristic of a sheaf is computed via duality.

This lift of constructible functions to constructible sheaves fits nicely within the story (if not the rigorous definition) of categorification that inspired our use of homology in \S\ref{sec:hom}. Here one lifts from the group of constructible functions to the category of constructible sheaves and then ``decategorifies'' by taking Euler characteristic.

The sheaf perspective does more than simply give a definition of the Euler integral. If one pays special attention to the definition of $F_!$ and the limiting procedure of the stalk, one gets the \style{base change} theorem \cite{Schurmann}:
\[
    (RF_!\Sheaf)_y \cong R\Gamma_c(F^{-1}(y),\Sheaf)
\]
The integral over the fiber computes the compactly supported sheaf cohomology there, simply by applying $\chi$ to both sides. The correspondence between constructible sheaves and functions is totally functorial: one may operating on sheaves first and then decategorify, or decategorify first and then operate on functions.

\specialsection*{\bf Applications to Sensor \& Network Data Aggregation}
\sweetline
\vspace{0.1in}

Although sheaf theory provides a systematic approach to Euler calculus, the real utility of sheaves lies in a conceptual broadening of the notion of a function. Instead of assigning a single number to a point, a vector, group or other type of datum can be substituted. Moreover, this assignment is local in nature --- one is recording observations valid only in a neighborhood of the observer. Sheaf cohomology answers the question whether these local observations can be aggregated into globally coherent ones and provides the obstruction to consistency.

The core insight then is the following: there is a rigorous framework for organizing observations valued in arbitrarily complicated data types, that is local for whatever topology desired, and a global calculus for inference-making. This global calculus {\it is} sheaf theory. Euler calculus is just a small part of this larger theory that exploits the following: constructible sheaves can be decategorified into constructible functions via local Euler-Poincar\'e index, so general data types are packaged into integer data. Integer data is computationally easier to manipulate, but the interpretation of these quantities in the context of an applied problem only makes sense on the sheaf level. We thus turn from the theory of Euler calculus to its recent applications in management of data in engineering systems and sensor networks, beginning with an example having no sensible distinction between sheaf and function --- counting anonymous targets locally. 

\section{Target enumeration}
\label{sec:target}

Consider a finite collection of \style{targets}, represented as discrete points in a domain $W$. There is a large collection of sensors, each of which observes some subset of $W$ and counts the number of targets therein. The sensors will be assumed to be distributed so densely as to be approximated by a topological space $X$ (typically a manifold, in the continuum limit). Note that various factors important in sensing (geometry of the domain, obstacles, time-dependence, moving targets, etc.) may be accounted for by enlarging the domains $X$ and/or $W$. For example, the case of moving distinct targets may require setting $W$ to be a configuration space of points.

There are many modes and means of sensing: infrared, acoustic, optical, magnetometric, and more are common. To best abstract the idea of sensing away from the engineering details, it is proper to give a topological definition of sensing. In a particular system of sensors in $X$ and targets in $W$, let the \style{sensing relation} be the relation $\Sense\subset W\times X$ where $(w,x)\in\Sense$ iff a sensor at $x\in X$ detects a target at $w\in W$. The horizontal and vertical fibers (inverse images of the projections of $\Sense$ to $X$ and $W$ respectively) have simple interpretations. The vertical fibers --- \style{target supports} --- are those sets of sensors which detect a given target in $W$. The horizontal fibers --- \style{sensor supports} --- are those targets observable to a given sensor in $X$.

Assume that the sensors are additive but anonymous: each sensor at $x\in X$ counts the number of targets in $W$ detectable and returns a local count $h(x)$, but the identities of the sensed targets are unknown. This counting function $h\colon X\to\zed$ is, under the assumption of tameness, constructible. Given $h$ and some minimal information about the sensing relation $\Sense$, what is the total number of targets? This is in essence a problem of computing a global section (number) from a collection of local sections (target counts): it is no surprise that the tools of sheaf theory are applicable.

\begin{theorem}[\cite{BG:enum}]
\label{thm:enumerate}
If $h\in\CF(X)$ is a counting function of target supports $U_\alpha$ with $\chi(U_\alpha)=N\neq 0$ for all $\alpha$, then $\#\alpha=\frac{1}{N}\int_Xh\,d\chi$.
\end{theorem}
The proof is trivial given that $\int\,d\chi$ is an integration operator with measure $d\chi$.
\begin{equation}
    \int_X h\,d\chi
    =
    \int_X \left(\sum_\alpha \one_{U_\alpha}\right) \,d\chi
    =
    \sum_\alpha \int_X \one_{U_\alpha}\,d\chi
    =
    \sum_\alpha \chi(U_\alpha)
    =
    N\,\#\alpha
.
\end{equation}

This is really a problem in aggregation of redundant data, since many nearby sensors with the same reading are detecting the same targets; in the absence of target identification (an expensive signal processing task, in practice), it seems very difficult. Notice that the restriction $N\neq 0$ is nontrivial. If $h\in\CF(\real^2)$ is a finite sum of characteristic functions over annuli, it is not merely inconvenient that $\int_{\real^2} h\,d\chi=0$, it is a fundamental obstruction to disambiguating sets. Some sums of annuli may be expressed as a union of different numbers of embedded annuli.

\begin{figure}[hbt]
\begin{center}
\includegraphics[angle=0,width=2.0in]{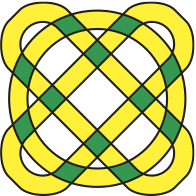}
\caption{An example for which the integral with respect to Euler characteristic vanishes. Note that the integrand can be represented as a sum of indicator functions of embedded annuli, the number of which is indeterminate.}
\label{fig:annuli}
\end{center}
\end{figure}

\section{Enumeration and Fubini}
\label{sec:fubini}

The Fubini theorem is very significant in applications of the Euler calculus, both classical (\eg, the Riemann-Hurwitz Theorem \cite{Viro,GZ}) and contemporary.

\subsection{Enumerating vehicles}
One application of the Fubini Theorem is to time-dependent targets. Consider a collection of vehicles, each of which moves along a smooth curve $\gamma_i:[0,T]\to\real^2$ in a plane filled with sensors that count the passage of a vehicle and increment an internal counter. Specifically, assume that each vehicle possesses a {\em footprint} --- a support $U_i(t)\subset\real^2$ which is a compact contractible neighborhood of $\gamma_i(t)$ for each $i$, varying tamely in $t$. At the moment when a sensor $x\in\real^2$ detects when a vehicle comes within proximity range --- when $x$ crosses into $U_i(t)$ --- that sensor increments its internal counter. Over the time interval $[0,T]$, the sensor field records a counting function $h\in\CF(\real^2)$, where $h(x)$ is the number of times $x$ has entered a support. As before, the sensors do not identify vehicles; nor, in this case, do they record times, directions of approach, or any ancillary data. It is helpful to think of the \style{trace} of a vehicle: the union over $t$ of $U_i(t)$ in $\real^2$. Note the possibility that the trace can be a non-contractible set. On an overlap, however, the counting function reads a value higher than $1$. This is the key to resolving the true vehicle count.

\begin{proposition}[\cite{BG:enum}]
\label{prop:t-dep-euler}
Assuming the above, the number of vehicles is equal to $\int_{\real^2}h\,d\chi$.
\end{proposition}
\begin{proof}
Consider the projection map $p\colon\real^2\times[0,T]\to\real^2$. Each target traces out a compact tube in $\real^2\times[0,T]$ given by the union of slices $(U_i(t),t)$ for $t\in[0,T]$. Each such tube has $\chi=1$. The integral over $\real^2\times[0,T]$ of the sum of the characteristic functions over all $N$ tubes is, by Theorem \ref{thm:enumerate}, $N$, the number of targets. By the Fubini Theorem, this also equals $\int_{\real^2}h\,d\chi$, where $h(x)$ is the value of the integral of the aforementioned sum of tubes over the fiber $p\inv(x)$. Since $p\inv(x)$ is $\{x\}\times[0,T]$, the integral over $p\inv(x)$ records the number of (necessarily compact) connected intervals in the intersection of $p\inv(x)$ with the tubes in $\real^2\times[0,T]$. This number is precisely the sensor count (the number of times a sensor detects a vehicle coming into range).
\end{proof}

\begin{figure}[hbt]
\begin{center}
\includegraphics[angle=0,width=5in]{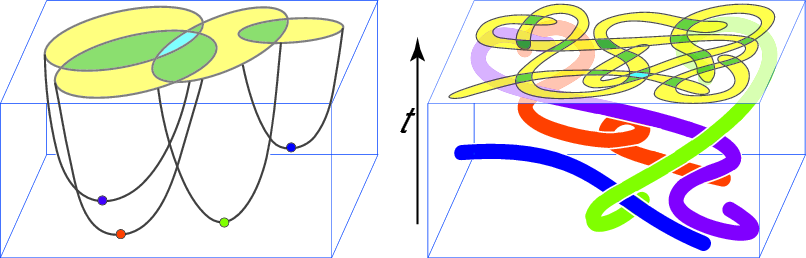}
\caption{Time-dependent phenomena (waves [left] or tracks [right]) can be enumerated by sensors without clocks.}
\label{fig:t-dep}
\end{center}
\end{figure}

\subsection{Enumerating wavefronts}
Consider a finite collection of points $\Obs_\alpha$ in $W\subset\real^n$. Each $\Obs_\alpha$ represents an event which occurs at some time and which triggers a wavefront that propagates for a finite extent. Assume that each sensor has the ability to record the presence
of a wavefront which passes through its vicinity. Each node has a simple counter memory which allows it to store the number of wavefronts that have passed. Under the continuum assumption, this yields a counting function $h:W\to\nats$ that returns wavefront detection counts {\em a posteriori}. The problem is to determine the number of source events $\Obs_\alpha$.

In this setting, there is no temporal data associated to the sensors. With a particular assumption on the sensing modality, this problem is solved as a corollary of Proposition \ref{prop:t-dep-euler}, using the same Fubini argument. Assume that the `wavefront' associated to each event $\Obs_\alpha$ induces a continuous definable map $F_\alpha$ from a compact ball $D^n$ to $W$ whose restriction to rays from the origin are geodesic rays in $W$ based at $\Obs_\alpha$. It is not enough to model sensors which count wavefronts by recording the number of `fronts' that have passed, as one must account for singularities. To that end, the cleanest assumption for the counting sensors is the following:

\begin{corollary}
\label{cor:wavefront}
In the context of this subsection, assume that each sensor at $w\in W$ increments its internal counter by $\chi(F_\alpha\inv(w))$ as the wavefront of $\Obs_\alpha$ passes over. Under this assumption, the number of triggering events is $\#\alpha = \int_W h\,d\chi$.
\end{corollary}
\begin{proof}
Apply the Fubini theorem to $h=\sum_\alpha (F_\alpha)_*\one_{D^n}$, where $F_\alpha$ is the mapping of the wavefront into $W$.
\end{proof}

The assumption is academic, but not outrageous, since we have suppressed the time variable. Since $n=\dim(W)=\dim(D^n)$, the inverse image $F_\alpha\inv(w)$ is generically discrete, and the assumption on the sensor modality boils down to a count of the number of passing wavefronts over the entire time interval. Of course, certain complications can arise in practice. For example, very coarse binary sensors may not be able to distinguish between one wavefront and several wavefronts passing over simultaneously: this can lead to positive-codimension defects in the counting function $h$. A similar loss of upper semi-continuity occurs when there is reflection of wavefronts along the boundary $\partial W$. For a compact domain $W\subset\real^n$ with smooth boundary $\partial W$, consider a wavefront-counting integrand $h=\sum_\alpha (F_\alpha)_*\one_{D^n}$ whose projection maps $F_\alpha$ may have fold singularities (reflections) along $\partial W$. Let $h^+:W\to\nats$ be the upper semi-continuous extension of $h$. Then, as the reader may show:
\begin{equation}
\label{eq:wavefront}
\#\alpha =
    \int_W h\,d\chi =
        \int_W h^+d\chi
        - \frac{1}{2}\int_{\partial W}h^+\,d\chi .
\end{equation}

\begin{figure}[hbt]
\begin{center}
\includegraphics[angle=0,width=5in]{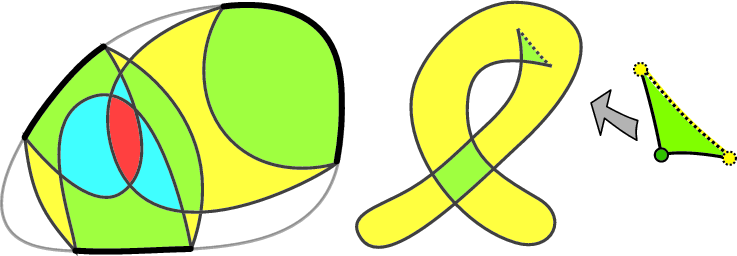}
\caption{Singularities in a system of wavefronts can be resolved by numerical approximation in the case of fold (reflection) singularities [left], but not in the case of cusps [right].}
\label{fig:singularities}
\end{center}
\end{figure}

This allows for numerical approximation of integrals involving fold singularities (in the context of signal reflection). We are unaware of how to meaningfully compute integrals when wavefronts have cusp singularities. There are several other challenges associated to measuring signals in the context of waves and obstacles: besides reflections and caustics, diffractions around corners are nontrivial and seemingly difficult to excise.

\subsection{Enumerating via beams}
Fix a Euclidean target space in $\real^n$ and consider a variant of the counting problem in which a sensor at each $x\in\real^n$ senses targets via a ``beam'' that is a round compact $k$-dimensional ball in $\real^n$ centered at $x$ (the term beam evoking the case $k=1$). Each target $\Obs_\alpha$ has spatial extent equal to a closed convex domain $V_\alpha$ in $\real^n$. Each sensor at $x\in\real^n$ performs a sweep of its $k$-ball beam over all possible bearings. At each such location-bearing pair, the sensor counts the number of target shapes $V_\alpha$ which intersect the beam. The sensor field is parameterized by the affine Grassmannian  $\AGr_k^n=\real^n\times\Gr^n_k$, where $\Gr^n_k$ is the Grassmannian of $k$-planes in $\real^n$. (For example, the projective space $\RP^{n-1}$ is $\Gr^n_1$.) Thus, the sensor field returns a counting function $h\colon\AGr_k^n\to\zed$.

\begin{theorem}
\label{thm:beam}
Under the above assumptions and the additional assumption that if $n$ if even then so is $k$, the number of targets is equal to
\begin{equation}
\label{eq:grass}
\#\alpha =
   \left(
    \begin{array}{c}
            \lfloor\frac{n}{2}\rfloor
            \\
            \lfloor\frac{k}{2}\rfloor
    \end{array}
    \right)
    \int_{\AGr_k^n}h\,d\chi   ,
\end{equation}
where $\lfloor\cdot\rfloor$ denotes the floor function.
\end{theorem}
\begin{proof}
Target supports $U_\alpha\subset\AGr_k^n$ are computed as follows. Fix an $\alpha$ and fix a {\em bearing} $k$-plane in the Grassmannian $\Gr_k^n$. The set of nodes in $\real^n$ at which this $k$-ball intersects $V_\alpha$ is star-convex with respect to (the centroid of) $V_\alpha$. Thus, the target support $U_\alpha$ is homeomorphic to $V_\alpha\times\Gr_k^n$ and thus to $\Gr_k^n$. The Euler characteristic of $\Gr_k^n$ is:
\[
    \chi(\Gr_k^n) = \left\{
    \begin{array}{cl}
    0 & : n {\mbox{ even}}, k {\mbox{ odd}}
    \\
   \left(
    \begin{array}{c}
            \lfloor\frac{n}{2}\rfloor
            \\
            \lfloor\frac{k}{2}\rfloor
    \end{array}
    \right)
    & : {\mbox{ else}}
    \end{array}
    \right.  .
\]
The result follows from Theorem \ref{thm:enumerate}.
\end{proof}

Of course, such sensors (``$k$-dimensional planes in $n$-space'') at continuum-level densities are fanciful at best. Worse still, in the most physically plausible setting --- the case of the plane ($n=2$) with linear beams ($k=1$) --- the theorem fails to be useful, since target supports have the homotopy type of $\Sphere^1$, a measure zero set in Euler calculus.

\section{Numerical approximation}
\label{sec:numerical}

What if (as is true in practice) sensor counting data is not given over a continuous space of sensors but rather a discrete set of points? In expressing the answer to the aggregation problem as a formal integral, one is led to the notion of numerical integration. The task of approximating an integral based on a discrete sampling has a long and fruitful history. One must be cautious, however; this integration theory is not at all like that of Riemann or Lebesgue.  An annulus in the plane is a set of Euler measure zero; a single point is a set of Euler measure one. Thus, $d\chi$ gives tremendous flexibility, as one can integrate target supports with vastly dissimilar geometry; however, noise, in the form of errors at individual sensors, can be fatal to the estimation of the integral. The development of numerical integration theory for $d\chi$ is in its infancy.

\subsection{Euler integration based on a triangulation}
\label{sec:triang}
Let $h\colon\real^2\to\nats$ be constructible and assume that $h$ is known only over the vertex set of a triangulation of $\real^2$ and that $\tilde{h}$ is the constructible function whose value on an open $k$-simplex equals the minimum of all boundary vertex values. The most straightforward approach to the approximation of the integral is (since $\tilde{h}\geq 0$ and upper semicontinuous) to use the following formula:
\begin{equation}
\label{eq:combnumeuler}
    \sum_{s=0}^\infty \#V\{\tilde{h}>s\}-\#E\{\tilde{h}>s\}+\#F\{\tilde{h}>s\} ,
\end{equation}
where $\#V$, $\#E$, and $\#F$ refer to the number of vertices, edges, and faces, respectively. This formula can, however, fail in several ways. The sampling can be too sparse relative to the features of $h$, as in Figure \ref{fig:badchambers}. This is not too surprising, since a too-sparse sampling impedes the approximation of Riemann integrals likewise.

\begin{figure}[hbt]
\begin{center}
\includegraphics[angle=0,width=5in]{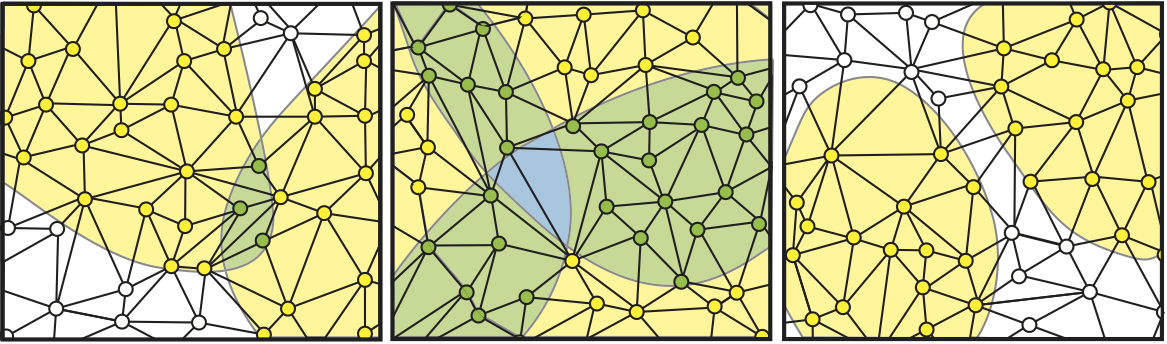}
\caption{These types of triangulations leads to incorrect Euler integrals due to mis-sampling of connected components of upper or lower excursion sets of the integrand: the relevant features are too small for the triangulation to accurately sample the Euler characteristics.}
\label{fig:badchambers}
\end{center}
\end{figure}

Unfortunately, such discretization errors can persist, even with triangulations of arbitrarily high density. A generic codimension two singularity in $h\in \CF^+(\real^2)$ is locally equivalent (up to a constant) to the sum of Heaviside step functions $h(x,y):=H_0(x)+H_0(y)\in\CF^+(\real^2)$. A random triangulation will be with positive probability locally equivalent to that in Figure \ref{fig:discreterror}. The level sets of $h$ with respect to the triangulation produce a spurious path component, adding $1$, wrongly, to the Euler integral. This example is scale-invariant and not affected by increased density of sampling. In a random placement of $N$ discs in a bounded planar domain, the expected number of such intersection points is quadratic in $N$, leading to potentially large errors in numerical approximations to $\int\cdot d\chi$, regardless of sampling density.

\begin{figure}[hbt]
\begin{center}
\includegraphics[angle=0,width=4.5in]{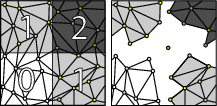}
\caption{At a typical discontinuity in an integrand $h\in\CF(\real^2)$, a triangulation [left] can lead to an error in the topology of the level sets [right], causing an error in the integral computation. This is a scale-independent phenomenon.}
\label{fig:discreterror}
\end{center}
\end{figure}

\subsection{Exploiting numerical errors in Fubini}
\label{sec:errors}
Sometimes, numerical errors are serendipitous. Consider $h$ a sum of characteristic functions over compact annuli $A$ in $\real^2$ that are {\em convex} in the sense that each annulus is a compact convex disc minus a smaller open convex sub-disc. As noted $\int h\,d\chi=0$. However, assume (1) that $h$ is sampled over a rectilinear lattice; (2) that the geometric ``feature size'' of $h$ is sufficiently large relative to the sampling density (meaning that the aspect ratios, size, and distances between boundary curves are suitably bounded). Then, to put it in a less-than-technical form,
\begin{equation}
    \#{\mbox{annuli}} = \frac{1}{2}\int_{\tiny{\mbox{the other}}}\int_{\tiny{\mbox{one way}}} \tilde{h}\,d\chi\,d\chi ,
\end{equation}
where the integration is performed numerically (using Equation (\ref{eq:combnumeuler}) of \S\ref{sec:triang}). The reasoning is as follows. By additivity and general position, it suffices to demonstrate in the case of a single annulus. Performing the first integral means counting the number of connected intersections of lines in the rectilinear lattice with the annulus. The result is a function which increases from zero to two, then back to zero. Assuming transversality, the (numerical!) integral of the resulting function on $\real$ is equal to two, {\em not zero}, since codimension-1 values of the integrand are ignored by the numerical sampling. The curious reader should, armed with Fubini, contrast the case of a field of sensors over $\real\times\real$ with the discretized setting.

\begin{figure}[hbt]
\begin{center}
\includegraphics[angle=0,width=5in]{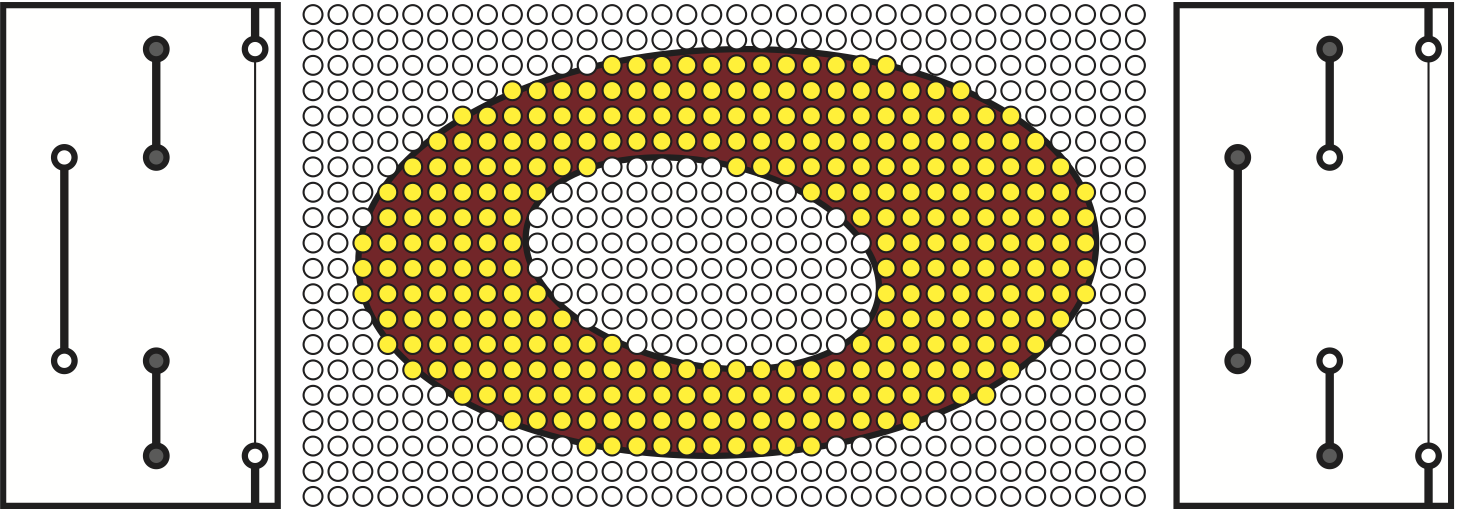}
\caption{In a dense lattice [center], integrating exactly along (horizontal) directions [left] yields an integrand with Euler integral zero. However, performing numerical integration introduces errors [right] which allows one to count fat convex annuli in general position.}
\label{fig:annuluscheeky}
\end{center}
\end{figure}

\section{Planar {\em ad hoc} networks}
\label{sec:adhoc}

A highly effective numerical method for performing integration with respect to Euler characteristic over a planar network uses the homological properties of the Euler characteristic and is an excellent argument for the utility of that approach to $\chi$. We consider an integrand $h\in\CF(\real^2)$ sampled over a network $\Graph$ in which values of $h$ are recorded at the vertices of $\Graph$ and edges of $\Graph$ correlate (roughly) to distance in $\real^2$. However, no coordinate data are assumed --- the embedding of $\Graph$ into $\real^2$ is unknown. This is a realistic model of an {\em ad hoc} coordinate-free network, such as might occur when simple sensors with no GPS set up a wireless communications network. This lack of coordinates or distances makes the use of a triangulation problematic, though not impossible.
Equation (\ref{eq:eulerexcursion}) suggests that the estimation of the Euler characteristics of the upper excursion sets is an effective approach. However, if the sampling occurs over a network with communication links, then it is potentially difficult to approximate those Euler characteristics, since (inevitable) undersampling leads to holes that ruin an Euler characteristic approximation. Duality is the key to mitigating this phenomenon.

\begin{theorem}
\label{thm:planardual}
For $h\colon\real^2\to\nats$ constructible and upper semi-continuous,
\begin{equation}
\label{eq:dual2d}
    \int_{\real^2}h\,d\chi
    =
    \sum_{s=0}^\infty
    \left(\beta_0\{h>s\} - \beta_0\{h\leq s\} + 1\right)
,
\end{equation}
where $\beta_0=\dim\ H_0$, the number of connected components of the set.
\end{theorem}
\begin{proof}
Let $A$ be a compact nonempty subset of $\real^2$. From (1) the homological definition of the Euler characteristic and (2) Alexander duality, we note:
\[
    \chi(A) = \dim\ H_0(A) - \dim\ H_1(A) = \dim\ H_0(A) + \dim\ H_0(\real^2-A)-1 .
\]
Since $h$ is upper semi-continuous, the set $A=\{h>s\}$ is compact. Noting that $\real^2-A=\{h\leq s\}$, one has:
\[
    \int h\,d\chi
    = \sum_{s=0}^\infty \chi\{h>s\}
    = \sum_{s=0}^\infty \dim\ H_0(\{h>s\})-\left(\dim\ H_0(\{h\leq s\})-1\right) .
\]
\end{proof}

This result gives both a criterion and an algorithm for correctly computing Euler integrals based on nothing more that an {\em ad hoc} network of sampled values.
\begin{figure}[hbt]
\begin{center}
\includegraphics[angle=0,width=4.5in]{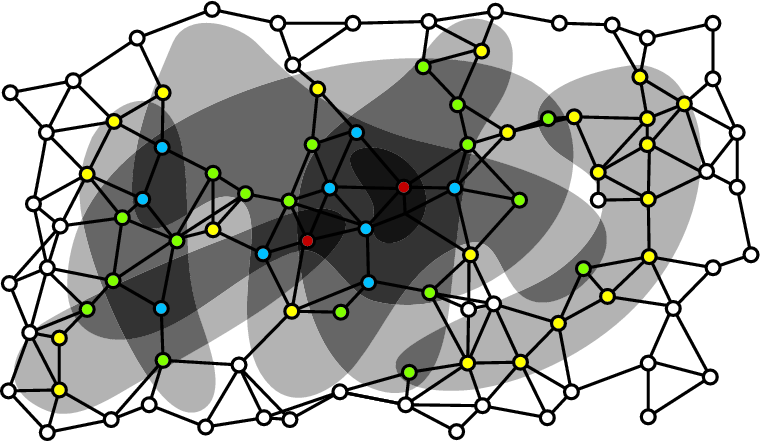}
\caption{A sparse sampling over an {\em ad hoc} network retains enough connectivity data to evaluate the integral exactly.}
\label{fig:dualex}
\end{center}
\end{figure}

\begin{corollary}
\label{cor:sampling}
The degree of sampling required to ensure exact approximation of $\int h\,d\chi$ over a planar network $\Graph$ is that $\Graph$ correctly samples the connectivity of all the upper and lower excursion sets of $h$.
\end{corollary}

In other words, if, given $h$ on the vertices of $\Graph$, one extends to an upper semicontinuous $\tilde{h}\in\CF(\Graph)$ in the usual manner, then the criterion is that the upper and lower excursion sets of $\tilde{h}$ on $\Graph$ have the same number of connected components as those of $h$ on $\real^2$.

\section{Software implementation}
\label{sec:eucharis}

The formula in Theorem \ref{thm:planardual} has been implemented in {\tt Java} as a general {\em Eu}ler {\em ch}aracteristic {\em i}ntegration {\em s}oftware, {\em Eucharis} \cite{Eucharis}. The determination of the number of connected components of the upper and lower excursion sets is a simple clustering problem, computable in logspace with respect to the number of network nodes. The software implementation has the following features:
\begin{enumerate}
\item Lattice or {\em ad hoc} networks of arbitrary size and communication distance can be generated.
\item Targets with supports of predetermined shape (circular, polygonal, etc.) can be placed at will; targets with support a neighborhood of a drawn path are also admissible.
\item  The value of the counting function $h$ is represented by colors on the nodes.
\item Euler integrals are estimated via Equation (\ref{eq:dual2d}) and compared to true values.
\item Various advanced transforms are available, including addition of noise, smoothing via convolution with Gaussians , and other integral transforms (using ideas from \S\ref{sec:Rval}-\ref{sec:defker}).
\end{enumerate}

Screenshots are included as Figures \ref{fig:eucharis1}.

\begin{figure}[hbt]
\begin{center}
\includegraphics[angle=0,width=5in]{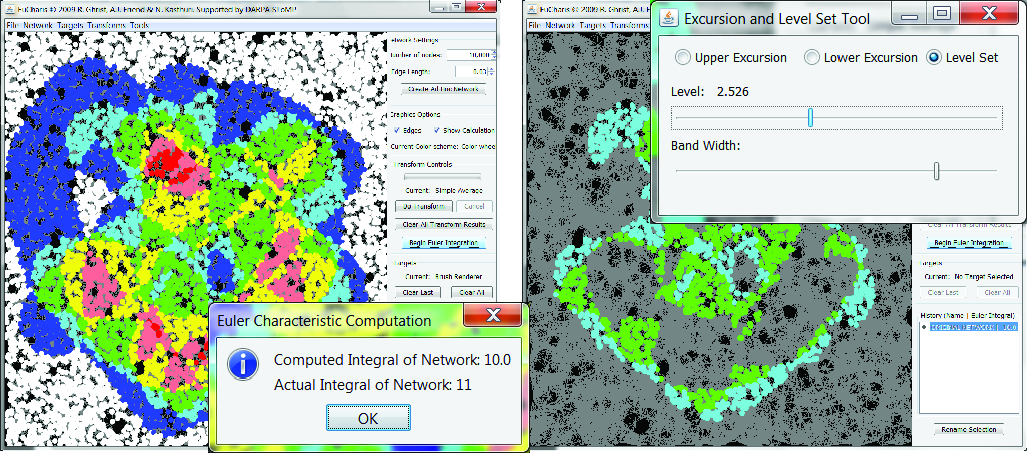}
\caption{Screenshots of {\em Eucharis}, including an Euler integral approximation [left] and a level-set explorer [right]. The true integral ($11$) is approximated ($10$) numerically, but imperfectly so, due to the existence of small, poorly-sampled chambers.}
\label{fig:eucharis1}
\end{center}
\end{figure}

Of course, the criterion of Corollary \ref{cor:sampling} is not always verifiable: to know the connected components of the excursion sets of the true integrand is, under realistic circumstances, a luxury. Approximating integrals of Riemann integrals based of discrete sampling suffers from similar problems, as mass can concentrated in a small region. Unlike that situation, we cannot easily bound derivatives and work in that manner. There is a wealth of open problems relating to how one properly does numerical estimation in Euler calculus. It is remarkable that, though the Euler calculus has existed in its fullest form for more than thirty years, only a few works (\eg, \cite{Klain}) have addressed the problem of estimating Euler characteristics of sets based on discrete data. The broader problem of numerical computation of the Euler integral based on discrete sampling seems to be in its infancy.

\specialsection*{\bf Integral Transforms \& Signal Processing}
\sweetline
\vspace{0.1in}

The previous chapter presented simple applications of basic Euler integrals to data aggregation problems, using little more than the definition and well-definedness of the integral, along with the corresponding Fubini Theorem. Euler {\em Calculus} is more comprehensive than the integral alone: Euler integration admits a variety of operations which mimic analytic constructs. The operations surveyed in this chapter are the beginnings of a rich calculus blending analytic, combinatorial, and topological perspectives into a package of particular relevance to signal processing, imaging, and inverse problems.

\section{Convolution and duality}
\label{sec:conv}

On a finite-dimensional real vector space $V$, a \style{convolution} operation with respect to Euler characteristic is straightforward. Given $f,g\in \CF(V)$, one defines
\begin{equation}
\label{eq:conv}
    (f*g)(x) = \int_V f(t)g(x-t)\,d\chi(t) .
\end{equation}

\begin{lemma}
\label{lem:convprod}
\begin{equation}
    \int f*g\,d\chi=\int f\,d\chi\,\int g\,d\chi
\end{equation}
\end{lemma}
\begin{proof}
Fubini.
\end{proof}
This convolution operator is of interest to problems of computational geometry, given the close relationship to the \style{Minkowski sum}: for $A$ and $B$ convex, $\one_A*\one_B=\one_{A+B}$, where $A+B$ is the set of all vectors expressible as a sum of a vector in $A$ and a vector in $B$ \cite{Viro,Schapira:op,BGRR,Groemer}.
%
%
The work of Guibas and collaborators \cite{GRS,Guibas,Ramkumar} contains a wealth of results on the computational complexity of using Euler convolution instead of the usual Minkowski sums in computational geometry, with applications in \cite{Ramkumar} to robot motion planning and obstacle avoidance.


Where there is convolution, a deconvolution operator lurks, if as nothing more than desideratum. In this case, the appropriate avenue to deconvolution is via a formal (Verdier) duality operator on sheaves \cite{Schapira:op}. In the context of $\CF$, this vast generalization of Poincar\'e duality takes on a simple and concrete form. Define the \style{dual} of $h\in CF(X)$ to be:
\begin{equation}
    \Dual h(x)
    =
    \lim_{\epsilon\to 0^+}\int_Xh\one_{B(x,\epsilon)}d\chi ,
\end{equation}
where $B(x,\epsilon)$ denotes an {\em open} ball of radius $\epsilon$ about $x$. This limit is well-defined thanks to Theorem \ref{thm:Hardt} applied to $h\colon X\to\zed$. Duality provides a de-convolution operation ---  a way to undo a Minkowski sum.

\begin{lemma}[\cite{Schapira:tom}]
\label{lem:deconv}
For any convex closed $A\subset V$ with non-empty interior, the convolution inverse of $\one_A$ is $\Dual\one_{-A}$, where $-A$ denotes the reflection of $A$ through the origin. Specifically, $\one_A*\Dual\one_{-A}=\delta_0$, where $\delta$ is the indicator function of the origin.
\end{lemma}

The related \style{link} transform $\Link$ is defined as in $\Dual$ but by integrating over the boundary sphere:
\begin{equation}
    \Link h(x)
    =
    \lim_{\epsilon\to 0^+}\int_X h\one_{\del B(x,\epsilon)}d\chi.
\end{equation}
The link transform acts, in a sense, like a derivative in the Euler calculus. For $h\in\CF(\real^{2n})$, $\Link h$, like a derivative, vanishes everywhere except at the discontinuities of $h$. On an $n$-manifold, $\Link$ acts as multiplication by $1-(-1)^n$ on open regions where $h$ is constant. By now, the reader is used to the particulars of this dimensional parity. The relationship to $\Dual$ is made precise in the following:
\begin{lemma}
$\Link=\Id-\Dual$.
\end{lemma}

\section{Radon transforms and inversion}
\label{sec:radon}

One of the most flexible integral transforms is the \style{Radon transform} of Schapira \cite{Schapira:tom,Brocker,CE}. This Fredholm-type transform integrates over a kernel via pushforward and pullback operations on $\CF$. Recall from \S\ref{sec:integral} the pushforward $F_*\colon\CF(X)\to\CF(Y)$ associated to a definable $F\colon X\to Y$ via fibrewise Euler integration. The \style{pullback} $F^*\colon\CF(Y)\to\CF(X)$ is the obvious composition: given $g\in\CF(Y)$ constructible, $(F^*g) := g\circ F$. Functoriality of these operations is expressed in the \style{projection formula}: for $g\in\CF(Y)$ and $h\in\CF(X)$,
\begin{equation}
\label{eq:proj}
           F_*(h(F^*g)) = F_*(h)g ,
\end{equation}

Given a locally closed definable set $S\subset W\times X$, let $P_W$ and $P_X$ denote the projection maps of $W\times X$ to their factors. The \style{Radon transform} is the map $\Radon_S\colon\CF(W)\to \CF(X)$ given by
\begin{equation}
    \Radon_S h = (P_X)_*((P_W^*h)\one_S) .
\end{equation}
The push/pull language is apt. One pulls back $h\in\CF(W)$ via projection to the support $S$ (along vertical fibers) then pushes this down (via integration along horizontal fibers) to $\Radon_S h\in\CF(X)$. This is reminiscent of a Fredholm transform with kernel $\one_S$.

\begin{example}
Duality on $\CF(X)$ is the Radon transform associated to the relation $S\subset X\times X$ where $S$ is a sufficiently small open tubular neighborhood of the diagonal $\Delta=\{(x,x)\,\colon\,x\in X\}$.
\end{example}

In the context of sensor networks, the Radon transform is entirely natural. As in \S\ref{sec:target} let $W$ denote the  target space and $X$ denote the sensor space. The sensing relation is $\Support = \{(w,x)\, :$~the sensor at $x$ senses a target at $w \}$. This lies in the product space $W\times X$ as a relation whose vertical fibers $\Support_w=P_X(P_W\inv(w)\cap\Support)$ are target supports and whose horizontal fibers $\Support_x=P_W(P_X\inv(x)\cap\Support)$ are sensor supports. Consider a finite set of targets $T\subset W$ as defining an atomic function $\one_T\in \CF(W)$. Observe that the counting function $h\in\CF(X)$ which the sensor field on $X$ returns is precisely the Radon transform $\Radon_{\Support}\one_T$. In this language, Theorem \ref{thm:enumerate} is implied by the following:

\begin{lemma}
\label{lem:radon}
Assume that $\Support\subset W\times X$ has vertical fibers $P_W\inv(w)\cap \Support$ with constant Euler characteristic $N$. Then, $\Radon_\Support\colon\CF(W)\to \CF(X)$ scales Euler integration by a factor of $N$:
\[
    \int_X\,d\chi\,\circ\,\Radon_\Support = N \int_W\,d\chi .
\]
\end{lemma}

\begin{proof}
Consider a compact contractible stratum $V\subset W$. Then $(P_W)^*\one_W=\one_{\Support'}$ for some $\Support'\subset\Support$. Knowing that $(P_X)_*\one_{\Support'}=\Radon_\Support\one_{V}$ and that the vertical fibers of $\Support'$ have $\chi=N$ yields
\[
    \int_X\Radon_\Support\one_{\Support'}d\chi
    = \int_{W\times X}\one_{\Support'}d\chi
    = \int_W\int_{F\inv(w)}\one_{\Support'}d\chi(s)\,d\chi(w)
    = N\int_W\one_Vd\chi = N ,
\]
thanks to the Fubini theorem. Linearity of the integral extends the result from $\one_V$ to all of $\CF(W)$.
\end{proof}

\subsection{Schapira's inversion formula}
\label{sec:inversion}

A similar regularity in the Euler characteristics of fibers allows a general inversion formula for the Radon transform \cite{Schapira:tom}.

\begin{theorem}[Schapira]
\label{thm:invradon}
Assume that $\Support\subset W\times X$ and $\Support'\subset X\times W$ have fibers $\Support_w$ and $\Support'_w$ in $X$ satisfying (1) $\chi(\Support_w\cap\Support'_w)=\mu$ for all $w\in W$; and (2) $\chi(\Support_w\cap\Support'_{w'})=\lambda$ for all $w'\neq w \in W$. Then for all $h\in \CF(W)$,
\begin{equation}
\label{eq:schapira}
    (\Radon_{\Support'}\circ\Radon_\Support)h
    =
    (\mu-\lambda)h
    + \lambda\left(\int_W h\,d\chi\right)\one_W
    .
\end{equation}
\end{theorem}
\begin{proof}
We prove a generalization as from \cite{BGL}. It is a simple matter to generalize from Radon transforms to Fredholm transforms with arbitrary constructible kernels having an Euler regularity in the fibers. Given a kernel $K\in\CF(W\times X)$, one defines the weighted Radon transform as $\Radon_K h = (P_X)_*((P_W^*h)K)$. Inversion requires an inverse kernel $K' \in \CF(X \times W)$ such that there exist constants $\mu$ and $\lambda$ with
$\int_X  K(w,x) K'(x,w') d\chi=(\mu-\lambda)\delta_{w-w'} + \lambda$ for all $w,w' \in W$. Then, for all $h \in \CF(W)$, we claim that:
\begin{equation}
\label{eq:weightedRadon}
         (\Radon_{K'}\circ\Radon_K)h = (\mu-\lambda) h + \lambda\left(\int_W h\, d\chi\right)\one_{W}.
\end{equation}
To show this, note that for any $w' \in W$,
\begin{eqnarray*}
    (\Radon_{K'}\circ\Radon_{K}h)(w')
    & = & \int_{X} \left[\int_{W} h(w)K(w,x)\,d\chi \right] K'(x,w')\,d\chi \\
    & = & \int_{W} h(w)\left[\int_{X}  K(w,x) K'(x,w') d\chi\right] d\chi  \\
    & = & \int_{W} \left[ (\mu-\lambda) h(w)\delta_{w-w'} + \lambda h(w)\right] d\chi    \\
    & = & (\mu-\lambda)h(w') + \lambda\int_W h\,d\chi ,
\end{eqnarray*}
where the Fubini theorem is used in the second line.
\end{proof}

If the conditions of Theorem \ref{thm:invradon} are met and $\lambda\neq\mu$, then the inverse Radon transform $\Radon_{S'}h=\Radon_{S'}\Radon_S\one_T$ is equal to a multiple of $\one_T$ plus a multiple of $\one_W$. Thus, one can localize and identify a collection of targets --- and determine the exact shape of $T$ --- by performing the inverse transform.

A more general proof still (by D. Lipsky \cite{BGL}) uses the following \style{cocycle condition}. Let $X_1$, $X_2$, and $X_3$ be spaces. Supposing the constructible kernels $K_i \in \CF\left(\prod_{j\neq i}X_j\right)$ and the projection maps $P_i\colon\prod_iX_i\to\prod_{j\neq i}X_j$ satisfy the cocycle condition
\begin{equation}
    K_{3} =(P_3)_*(P_1^*K_1 \cdot P_2^*K_2).
\end{equation}
Then, consider the diagram
\begin{equation}
\xymatrix{
 & X_1\times X_3 \ar[dl]\ar[dr]& \\
X_1 & & X_3\\
 &X_1\times X_2\times X_3 \ar[uu]^{P_2}\ar[dr]_{P_1}\ar[dl]^{P_3} & \\
X_1\times X_2\ar[uu]\ar[dr] & & X_2\times X_3 \ar[uu]\ar[dl]\\
 & X_2 &
}
\end{equation}
It follows from commutativity and Equation (\ref{eq:proj}) that $\Radon_{K_2} \circ \Radon_{K_1} = \Radon_{K_3}$.

\subsection{Examples of Radon inversion}
\label{sec:radonex}

\begin{example}[Hyperplanes]
\label{ex:radonbrain}
Schapira's paper outlines a potential application to imaging. Assume that $W=\real^3$ and one scans a compact subset $T\subset W$ by slicing $\real^3$ along all flat hyperplanes, recording simply the Euler characteristics of the slices of $T$. Since a compact subset of a plane has Euler characteristic the number of connected components minus the number of holes (which, in turn, equals the number of bounded connected components of the complement), it may be feasible to compute an accurate Euler characteristic, even in the context of noisy readings (using convolution to smooth the data).

This yields a constructible function on the sensor space $X=\AGr^3_2$ (the affine Grassmannian manifold of 2-planes in $\real^3$) equal to the Radon transform of $\one_T$. Using the same sensor relation to define the inverse transform is effective. Since $\Sense_w\cong\RP^2$ and $\Sense_w\cap\Sense_{w'}\cong\RP^1$, one has $\mu=\chi(\RP^2)=1$, $\lambda=\chi(\RP^1)=0$, and the inverse Radon transform, by (\ref{eq:schapira}), yields $\one_T$ exactly: one can recover the shape of $T$ based solely on connectivity data of black and white regions of slices.
\end{example}

\begin{example}[Sensor beams]
\label{ex:radonbeam}
Let the targets $T$ be a finite disjoint collection of points in $W=D^n$, the open unit disc in $\real^n$. Assume that the boundary $\partial W$ is lined with sensors, each of which sweeps a ray over $W$ and counts, as a function of bearing, the number of targets intersected by the beam. The sensor space $X$ is homeomorphic to $T_*\Sphere^{n-1}$, the tangent bundle of $\partial W$. (To see this, note that the bearing of a ray at a point $p\in\partial W$ lies in the open hemisphere of the unit tangent bundle to $W$ at $p$. This open hemisphere projects to the open unit disc in $T_p\partial W$.)
\begin{figure}[hbt]
\begin{center}
\includegraphics[angle=0,width=4in]{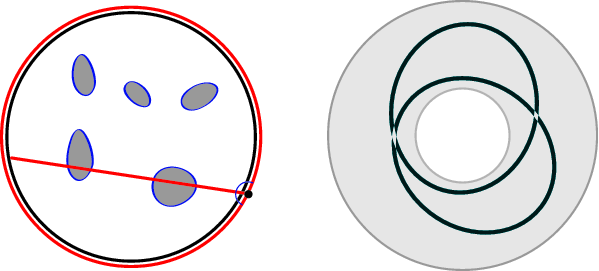}
\caption{Convex targets occupying a disc $W=D^2$ are sensed by a beam emitted from the edge [left]; any pair of distinct vertical fibers correspond to intersecting circles in $X$ [right]. }
\label{fig:beams}
\end{center}
\end{figure}

Any point in $W$ is seen by any sensor in $\partial W$ along a unique bearing angle. Thus, the sensor relation $\Sense$ has vertical fibers (target supports) which are sections of $T_*\Sphere^{n-1}$ and hence spheres of Euler measure $\mu=\chi(\Sphere^n)=1+(-1)^n$.
Any two distinct vertical fibers project to $X$ and intersect along the subset of rays from $\partial W$ that pass through both points in $W$: this is a discrete set of cardinality (hence $\chi$) $\lambda=2$. For $n$ even, one has $\lambda=2$, $\mu=0$, and Equation (\ref{eq:schapira}) implies that the inverse Radon transform gives $-2\one_T + 2(\# T)\one_W$. Thus, on even-dimensional spaces, targets may be localized with boundary beam sensors. One may likewise assume that the target set $T$ is not a collection of points but rather a disjoint collection of compact sets in $W$; target shapes will then be recovered, assuming the ability of beams to measure Euler characteristic of slices exactly. This is perhaps reasonable for convex-shaped targets.
\end{example}

\begin{example}[distance-based sensing]
\label{ex:invradoncomplements}
Non-self-dual examples of Radon inversion can be generated easily with complementary supports. For example, let $X=W=\real^n$ with $\Support_w$ a closed ball about $w$ and $\Support'_w$ the closure of the complement in $\real^n$. Physically, this means that the sensor relation detects proximity-within-range, and the inverse sensor relation counts targets out-of-range. The inversion formula applies, because of the singular nature of the ``eclipse'' that occurs when targets coalesce. Specifically: (1) $\chi(\Support_w\cap\Support'_{w'\neq w})=1$; and (2) $\chi(\Support_w\cap\Support'_{w})=\chi(\Sphere^{n-1})\neq 1$.

There are technicalities in applying the inversion formula in settings where the supports are non-compact, since we have defined $\CF$ in terms of compactly supported functions. This does not interfere with inversion in this example: all non-compact integrands are still tame. In addition, this example may be easily modified so that the supports (and the corresponding complements) change smoothly from point-to-point within the domain, so long as $\Support_w$ is, say, always compact and contractible and varies ``slowly'' enough so that $\chi(\Support_w\cap\Support'_{w'\neq w})=1$.
\end{example}

\begin{example}[discrete distance]
\label{ex:invradonbessel}
Using discrete distances can lead to inversions. Let $W=\real^{2n}=X$ and define  $\Support_w$ as the unit sphere about $w$: sensors count targets at a fixed distance. The inverse kernel $\Support'$ has fibers equal to concentric spheres about the basepoint of radius $r=1,3,5,\ldots$. Distinct fibers of $\Support$ and $\Support'$ always intersect in a set homeomorphic either to $\Sphere^{2n-2}$ or, if the points line up at integer distances, $\Sphere^0$: either way, $\chi(\Support_w\cap\Support'_{w'\neq w})=2$. In the instance when the basepoints coincide, the intersection is precisely the unit sphere $\Sphere^{2n-1}$ with $\chi=0$: full invertibility follows.
\begin{figure}[hbt]
\begin{center}
\includegraphics[angle=0,width=5in]{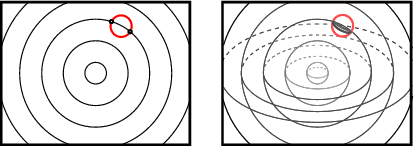}
\caption{Using supports for a Radon transform based on discrete distances works in even dimensions [left] but requires a signed weighting scheme in odd dimensions [right].}
\label{fig:radonex}
\end{center}
\end{figure}
\end{example}

\begin{example}[weighted kernels]
\label{ex:invradonweighted}
Consider the case of Example \ref{ex:invradonbessel}, modified so that $W=X=\real^{2n+1}$. On odd-dimensional spaces, the spheres of intersection have dimension $2n-1$ or $0$, depending on position. Thus $\lambda$ is not a constant, preventing invertibility. However, if we weight the inverse kernel $\Support'$ so that concentric spheres of radius $r=1,3,5,\ldots$ have weight $(-1)^{r+1}$, then $\chi(\Support_w\cap\Support'_{w'\neq w})=0$, independent of position. Full inversion is therefore possible using the generalization of Theorem \ref{thm:invradon} given in the proof.
\end{example}

\section{The Microlocal Fourier transform}
\label{sec:MF}

One of the deep tools of sheaf theory is a Fourier (or Fourier-Sato) transform \cite{KS} that is perhaps best described as {\em microlocal}.  In the context of $\CF(\real^n)$, this transform reveals itself to be extremely concrete and important in applications. Fix $\real^n$ with the standard Euclidean structure. The \style{microlocal Fourier transform} of $h\in\CF(\real^n)$ at $x$ takes as its argument a unit vector $\xi\in\Sphere^{n-1}$ and returns:
\begin{equation}
    (\FourierSato\,h)_x(\xi) = \lim_{\epsilon\to 0^+}\int_{B_\epsilon(x)}\one_{\xi\cdot (y-x)\geq 0}h\,d\chi(y) ,
\end{equation}
where, as before $B_\epsilon$ denotes the {\em open} ball of radius $\epsilon$. Like the classical Fourier transform, $\FourierSato$ takes a frequency vector $\xi$ and integrates over isospectral sets defined by dot product. The microlocal character of the transform stems from its dependence on the location $x$ as well as the direction $\xi$; in \S\ref{sec:fourier}, we will outline a more global Fourier transform.
\begin{figure}[hbt]
\begin{center}
\includegraphics[angle=0,width=4.5in]{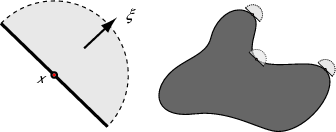}
\caption{The microlocal Fourier transform at $x$ has support on a half-open half-disc in the direction $\xi$ [left]. This directional component captures critical-point features of the integrand with respect to the Morse function pointing in the direction $\xi$ [right].}
\label{fig:MF}
\end{center}
\end{figure}

The microlocal Fourier transform is strongly connected to the (differential) geometry of the integrand. For $Y$ a compact tame set, $\FourierSato\one_Y(\xi)$ is the constructible function on $Y$ that records the \style{Euler-Poincar\'e index} of the (generically Morse) function pointing in the direction $\xi$ (\cf the results from \S\ref{sec:morse}). This function, when averaged over $\xi$ with respect to Haar measure on $\Sphere^{n-1}$ yields the curvature measure $d\kappa_Y$ on $Y$ implicit in the Gauss-Bonnet Theorem (as shown by Br\"ocker and Kuppe \cite{BK}): for any open Borel set $U\subset\real^n$,
\[
    \int_U d\kappa_Y
    =
    \frac{1}{\vol\ \Sphere^{n-1}}
    \int_{\Sphere^{n-1}}\int_U \FourierSato\,\one_Y
    \, d\chi \, d\xi
    .
\]
For $Y$ a smooth surface in $\real^3$, $d\kappa$ is (up to $2\pi$) Gauss curvature times the area form; on the boundary of such a surface, $d\kappa$ is geodesic curvature times the length form; and on a polyhedral domain, $d\kappa$ is supported on the vertex set and returns the exterior angles. The Gauss-Bonnet Theorem, properly interpreted in this language, says that
\begin{equation}
    \frac{1}{\vol\ \Sphere^{n-1}}
    \int_{\Sphere^{n-1}}\int_Y \FourierSato\,\one_Y
    \, d\chi \, d\xi
    =
    \int_Y\,d\chi = \chi(Y)
\end{equation}
It is apparent that the Gauss-Bonnet theorem extends via linearity of $\FourierSato$ to all of $\CF(Y)$, and one may sensibly speak of the curvature form $d\kappa_h$ of $h\in\CF(Y)$:
\begin{equation}
    \int_U d\kappa_h
    =
    \frac{1}{\vol\ \Sphere^{n-1}}
    \int_{\Sphere^{n-1}}\int_U \FourierSato\, h
    \, d\chi \, d\xi
    .
\end{equation}
This is yet another concrete and elegant example in the style of categorification that the sheaf-theoretic Euler calculus permits: the Gauss-Bonnet Theorem applies not just to sets but also to (constructible) data over sets.

\section{Hybrid Euler-Lebesgue integral transforms}
\label{sec:hybrid}

Blending Euler characteristic with Lebesgue integration has a long history in integral geometry and its applications \cite{Adler,Dohmen,GS,Wors1}. There are a number of interesting integral transforms based on a mixture of $d\chi$ and standard Lebesgue measure. We introduce (from \cite{GR}) two Euler integral transforms on real vector spaces for use in signal processing problems.

\subsection{Fourier}
\label{sec:fourier}

The following integral transform is best thought of as a global version of the microlocal Fourier transform $\FourierSato$ on a finite-dimensional real vector space $V$ with inner product. The Fourier transform takes as its argument a covector $\xi\in V^\vee$. For $h\in\CF(V)$ define the \style{Fourier transform} of $h$ in the direction $\xi\in V^\vee$ as
\begin{equation}
    \Fourier h(\xi) = \int_0^\infty\int_{\xi\inv(r)}h\,d\chi\,dr .
\end{equation}
The integration domain $\xi\inv(r)$ is the \style{isospectral set} of the transform; in this case, as in the classical Fourier transform, a hyperplane normal to $\xi$. One may reasonably rescale $\xi$ to a unit covector: this has the effect of rescaling the normal $dr$ measure.

\begin{example}
\label{ex:fourierwidth}
For $A$ a compact convex subset of $\real^n$ and $\norm{\xi}=1$, $(\Fourier\one_A)(\xi)$ equals the projected width of $A$ along the $\xi$-axis.
\end{example}

\begin{figure}[hbt]
\begin{center}
\includegraphics[angle=0,width=5in]{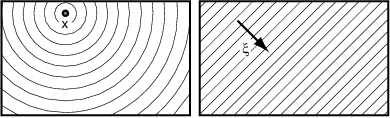}
\caption{Isospectral sets for Bessel [left] and Fourier [right] transforms. In these transforms, one integrates with respect to $d\chi$ along isospectral sets, then across the isospectral sets with respect to Lebesgue measure.}
\label{fig:isospectrals}
\end{center}
\end{figure}

\subsection{Bessel}
\label{sec:bessel}
For the Eulerian generalization of a classical Bessel (or Hankel) transform, let $V$ denote a finite-dimensional real vector space with norm $\norm{\cdot}$, and let $D_r(x)$ denote the compact disc of points $\{y \, : \, \norm{y-x}\leq r\}$. Recall that $\CF$ denotes compactly-supported definable integer-valued functions. For $h\in \CF(V)$ define the \style{Bessel transform} of $h$ via
\begin{equation}
    \Bessel h(x) = \int_0^\infty\int_{\partial D_r(x)}h\,d\chi\,dr .
\end{equation}

This transform Euler-integrates $h$ over the concentric spheres at $x$ of radius $r$, and Lebesgue-integrates these spherical Euler integrals with respect to $r$. For the Euclidean norm, these isospectral sets are round spheres. Given our convention that $\CF(V)$ consists of compactly supported functions, $\Bessel\colon\CF(V)\to\Def(V)$ is well-defined using Theorem \ref{thm:Hardt} (Hardt Theorem).

The Bessel transform can be seen as a Fourier transform of the log-blowup. This perspective leads to results like the following.

\begin{proposition}
\label{prop:fourierbessel}
The Bessel transform along an asymptotic ray is the Fourier transform along the ray's direction: for $h\in\CF(V)$ and $x\neq 0\in V$,
\begin{equation}
    \lim_{\lambda\to\infty}(\Bessel h)(\lambda x) = (\Fourier h)\left(\frac{x^\vee}{\norm{x^\vee}}\right) .
\end{equation}
where $x^\vee$ is the dual covector.
\end{proposition}
\begin{proof}
The isospectral sets restricted to the (compact) support of $h$ converge in the limit and the scalings are identical.
\end{proof}

While the Fourier transform obviously measures a {\em width} associated to a constructible function, the geometric interpretation of the Bessel transform is more involved and related to image processing.
\begin{example}
Consider $h\in\CF(\real^2)$ equal to $h=\sum_\alpha\one_{D_\alpha}$ for $D_\alpha$ a closed disc of unknown location and radius. For each $\alpha$, the Bessel transform of $\one_{D_\alpha}$ takes on a minimum of $0$ precisely at the center of the disc (since $\chi(\Sphere^1)=0$) and has maximum equal to the diameter of $D_\alpha$ (equal to $\Fourier(\one_{D_\alpha})$) asymptotically via Proposition \ref{prop:fourierbessel}.
\end{example}


Section \ref{sec:index} will use index theory to give computational formulae and interpretations for both $\Fourier$ and $\Bessel$, and will apply these observations to target localization and classification problems.

\section{Wavelet transforms}
\label{sec:wavelet}

Wavelet transforms have found numerous applications in recent years for representing multiscale signals \cite{Usevitch_2001,Bradley_1994,Wang_1996}. Though first appearing in the early 20th century \cite{Haar_1909}, the idea lay dormant until the 1970s, when it began to find application in signal processing, culminating in a substantial literature \cite{Daubechies_1992,Graps_1995,Hubbard_1998}. Underlying all wavelet theory is the idea that information can be extracted by examining the inner product between a signal of interest and a collection of localized reference signals.  Typical classes of localized reference signals are the Haar basis \cite{Haar_1909}, the Mexican hat \cite{Brinks_2008}, and the Daubechies wavelet \cite{Daubechies_1992}, though there are many others.

This section gives some preliminary observations about possible wavelet transforms in Euler calculus.  The obvious choice for an Euler inner product on $\CF(X)$ is:
\begin{equation}
\label{eulerinnerproduct}
    (f,g)_\chi=\int_X f\,g \, d\chi.
\end{equation}
This product is bilinear but lacks definiteness as for $f=\one_{(0,1)}$ one has $(f,f)_\chi=-1$. It is
unreasonable to expect the typical Fourier series decomposition of functions to hold in Euler calculus.
Nevertheless, let $S$ be a set of mutually ($\chi$-)orthogonal constructible functions.  What can be learned about $f\in \CF(X)$ by its Euler products with the elements of $S$?

As a first step, we will define two families of wavelets that together define an Euler calculus version of the Haar wavelet transform.
We work over $\real^n$ via induction on $n$. Define the first family
\begin{equation}
    H^{(0)}_{s,t}(x)=
    \begin{cases}
        1&\text{if }2^sx=t\\
        0&\text{otherwise}\\
    \end{cases}
\end{equation}
consisting of atomic functions on dyadic points of the real line, where $s$ and $t$ are integers.  The integers $s$ and $t$ are called {\it scale} and {\it translation} respectively.  We note that there is some harmless duplication of functions with this indexing, for instance, all of the functions with $t=0$ are identical.  The second family consists of step functions of the following form:
\begin{equation}
H^{(1)}_{s,t}(x)=
\begin{cases}
    1&\text{if} \, 0<2^sx-t<1/2\\
    -1&\text{if} \, 1/2<2^sx-t<1\\
    0&\text{otherwise}\\
\end{cases}
\end{equation}
for $s,t\in\zed$.

It should be immediately clear that distinct elements of $\{H^{(0)}_{s,t}\}$ are orthogonal, and that they are orthogonal to each element of $\{H^{(1)}_{\sigma,\tau}\}$ for which $\sigma \le s$. Similarly, the set $\{H^{(1)}_{s,t}\}$ is orthogonal but not orthonormal since $(H^{(1)}_{s,t},H^{(1)}_{s,t})_\chi=-2$.

We extend this definition to functions on $\real^n$ as follows. Let $p\in\{0,1\}^n$ be a binary $n$-tuple, and $s,t\in\zed^n$ be integer $n$-tuples.  We define the Haar wavelet functions on $\bfx=(x_1,\ldots,x_n)$ multiplicatively as
\begin{equation}
    H^{(p)}_{s,t}(x)=H^{(p_1)}_{s_1,t_1}(x_1) ... H^{(p_n)}_{s_n,t_n}(x_n),
\end{equation}
where we again have some harmless duplication of functions, due to the indexing of $H^{(0)}$.

We define the {\it Euler-Haar wavelet transform} of a function $f\in \CF(\real^n)$ to be
\begin{equation}
    (\Wavelet\,f)(p,s,t)=(f,H^{(p)}_{s,t})_\chi,
\end{equation}
for $p\in\{0,1\}^n$ and $s,t\in\zed^n$.

\begin{theorem}
\label{thm:EHWTinjectivity}
The Euler-Haar wavelet transform is injective on $\CF(\real^n)$.
\end{theorem}
\begin{proof}
The result follows trivially from showing injectivity in the case $n=1$. Suppose that $f\in \CF(\real)$.  Discern two cases
\begin{enumerate}
\item If some nonzero level set of $f$ contains a dyadic point $t/2^s$, then $\int f\,H^{(0)}_{s,t} d\chi \not= 0$.
\item Otherwise, $f\not= 0$ for only a finite set of points, since $f$ is constructible.  Therefore, there are integers $s,t$ so that the support of $H^{(1)}_{s,t}$ contains exactly one point on which $f\not= 0$.  Hence $\int f H^{(1)}_{s,t} d\chi \not= 0$.
\end{enumerate}
These cases are non-overlapping, so both types of Haar wavelet functions, the $H^(0)_{s,t}$ and the $H^(1)_{s,t}$, are necessary for injectivity to hold.
\end{proof}

We note from the proof that there is a function orthogonal to each $H^{(p)}_{s,t}$, so removing any element of the $\{H^{(p)}_{s,t}\}$ destroys the injectivity of $\Wavelet$ on the real line.

\begin{figure}[h]
  \begin{center}
       \includegraphics[width=2.85in]{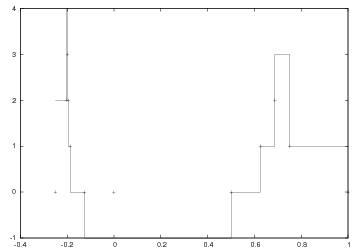}
  \end{center}
  \caption{This is the result of using the Fourier summation formula with $\Wavelet$ coefficients. The original function was the indicator function on the compact interval $[-0.2,0.7]$.}
  \label{fig:invertfail}
\end{figure}

We reiterate that $\Wavelet$ cannot be used to directly reconstruct functions by the usual Fourier series formula
\begin{equation*}
    \sum_{p,s,t} (f,H^{(p)}_{s,t})_\chi\, H^{(p)}_{s,t},
\end{equation*}
by lack of mutual orthogonality.  A rather spectacular example of this failure is shown in Figure \ref{fig:invertfail} which is supposed to be the reconstruction of an indicator function on an interval.

\begin{example}
As an example of how $\Wavelet$ operates on $\CF(\real^2)$, consider the indicator function shown in Figure \ref{fig:circleehwt}[left], with the $\Wavelet$ transform (with $p=(1,1)$, all scales) shown at right, with all relevant scales shown. The image is broken up into rows and columns divided by heavy lines.  Each column corresponds to a particular choice of horizontal scale (starting at zero at the left and increasing to the right), and each row corresponds to a particular vertical scale (starting at zero at the top and increasing as one proceeds downward).  Within each subframe are the coefficients corresponding to various translates of the $H^{(1)}$.
\begin{figure}[htb]
  \begin{center}
       \includegraphics[width=5.0in]{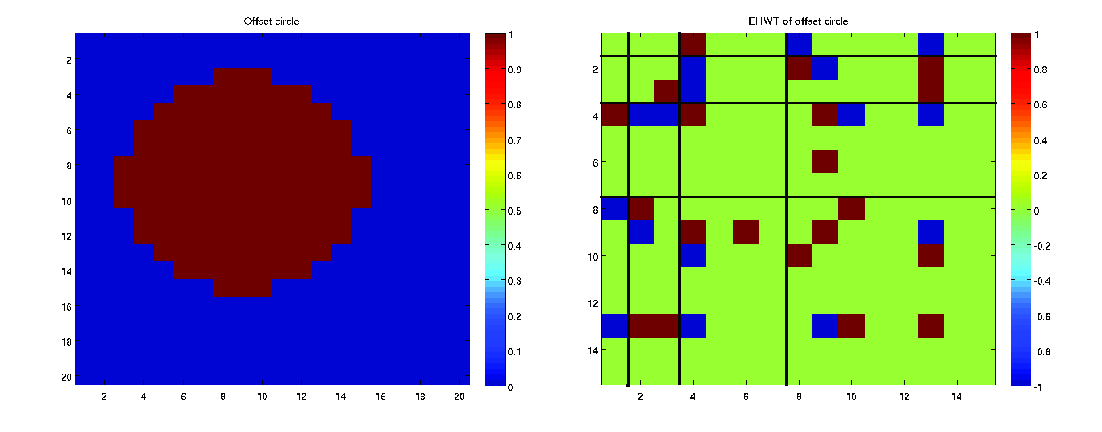}
  \end{center}
  \caption{An indicator function on a disc (left) and its $\Wavelet$ transform (right) with $p=(1,1)$. }
  \label{fig:circleehwt}
\end{figure}

From the definition of the $H^{(1)}$ functions, it is clear that the transform is sensitive to jumps in the original function.  In particular, in the lower-right corner of the transformed function in Figure \ref{fig:circleehwt} (corresponding to the smallest scale) is a rough copy of the gradient of the original function. However, since the original function is not aligned with any dyadic points in $x$ or $y$, there is not perfect symmetry in the transform.  It is also immediately apparent that the transform is rather sparse, though this is more a reflection of of the sparsity of the original function.
\end{example}

Another example of a more complicated constructible function and its $\Wavelet$ transform is shown in Figure \ref{fig:logoehwt}.  Together, these two examples indicate that the transform's coefficients depend on the function in a fairly complicated and sensitive manner and is currently not well understood. Indeed, $\Wavelet$ is {\it not} translation invariant, not even in the weak sense of the Fourier transform in which translation results in a multiplication by a phase. Although this sensitivity to translation might be perceived as a difficulty, it is shared (to a lesser extent) with more traditional wavelet transforms.  Both of these kinds of wavelet transforms are {\it dyadic} translation invariant, so that $\tau \Wavelet=\Wavelet \tau$ for a translation $\tau$ of a power-of-$2$ amount in a cardinal direction.  Its increased sensitivity to translation makes $\Wavelet$ potentially useful for detecting when signals fail to be aligned.  We suspect (but do not here argue) that $\Wavelet$ may prove useful in prefiltering mechanisms for correlation processes.  These find extensive application in synchronization contexts, the most prominent being in localization such as GPS signals and radar.
\begin{figure}[hbt]
  \begin{center}
       \includegraphics[width=5.0in]{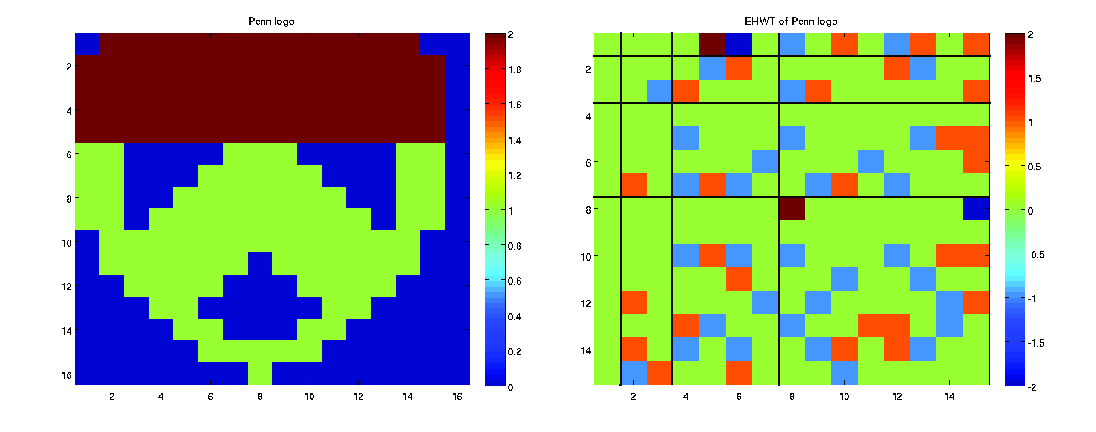}
  \end{center}
  \caption{A two dimensional image (left) and its $\Wavelet$ transform (right) with $p=(1,1)$.}
  \label{fig:logoehwt}
\end{figure}

\specialsection*{\bf Toward a $\real$-valued Euler Calculus}
\sweetline
\vspace{0.1in}

\section{Real-valued integrands}
\label{sec:Rval}

The Euler calculus, being integer-valued, has a delimited purview: it would be beneficial to have a calculus that manages other types of data than integral. Extending the Euler integral from $h\colon X\to\zed$ to (tame) functions with range in a discrete subgroup of a ring $R$ is a simple exercise for the reader. For purposes of constructing an honest numerical analysis for the constructible setting of Euler integration, it would be beneficial to have an extension to $\real$-valued continuous functions. Any such extension will require the appropriate notion of tameness for integrands. Fortunately, this is easy in the o-minimal setting. Given $X$ and $Y$ tame spaces in an o-minimal structure, let $\Def(X,Y)$ denote the definable functions $f\colon X\to Y$ --- those whose graphs in $X\times Y$ are tame. Recall from \S\ref{sec:tame} that such maps are only piecewise smooth/continuous in general; thus, denote by $\Def^k(X,Y)$ those definable maps which are $C^k$ smooth in the usual sense. This section focuses on extensions of Euler integration to $\Def(X):=\Def(X,\real)$. We will, as per the remainder of this article, carry an implicit assumption of compactly supported integrands; this assumption can be removed with care and additional work if needed.

\subsection{The Rota-Chen definition}
\label{sec:RC}

One such extension was proposed by Rota \cite{Rota} and used and enriched by Chen \cite{Chen}. For $h\in\Def(\real)$, define the integral to be:
\begin{equation}
    \int_\real h\,d\chi := \int_\real J h\,d\chi
    \quad : \quad
    (J h)(x) := h(x)-\frac{1}{2}\left(\lim_{y\to x^-}h(y)+\lim_{y\to x^+}h(y)\right) .
\end{equation}
The operator $J\colon\Def(\real)\to\CF(\real)$ records the (finite) jumps of $h$, so that the integral of $Jh$ becomes a finite sum. For $h\in\Def(\real^n)$, the integral is defined so as to satisfy the Fubini theorem:
\[
    \int_{\real^n} h\,d\chi := \iint\cdots\int h\,d\chi\cdots d\chi\,d\chi ,
\]
where each integral is over a coordinate axis of $\real^n$. This yields a well-defined integral operator which agrees with the usual $\int d\chi$ on $\CF$. Unfortunately, this extension of $\int d\chi$ to $\Def(\real^n)$ has the definable continuous functions $\Def^0$ in its kernel. This is singularly unhelpful as a tool for doing numerical analysis in Euler calculus in which we want to, say, take a piecewise-linear approximation to a sampled integrand, or perhaps smooth out a noisy integrand via convolution with a Gaussian.

\subsection{A Riemann-sum definition}
\label{sec:riemann}

A more simple-minded definition yields an entirely different extension of Euler integrals over $\Def(X)$. Following \cite{BG:PNAS}, extend integration to $\Def(X)$ via upper- or lower- step function approximations. Given $h\in \Def(X)$, define:
\begin{equation}
\label{eq:realintegral}
    \int_X h\,\dchifloor
    =
    \lim_{n\to\infty}
    \frac{1}{n}\int_X \lfloor nh\rfloor d\chi
    \quad : \quad
    \int_X h\,\dchiceil
    =
    \lim_{n\to\infty}
    \frac{1}{n}\int_X \lceil nh\rceil d\chi ,
\end{equation}
where $\lfloor\cdot\rfloor$ is the floor function and $\lceil\cdot\rceil$ is the ceiling function (returning the nearest integer below/above respectively). The notations $\dchifloor$ and $\dchiceil$ are chosen to denote which step-function approximation (lower/upper) is used in the limit. These limits exist and are well-defined, though not equal.

\begin{figure}[hbt]
\begin{center}
\includegraphics[angle=0,width=5in]{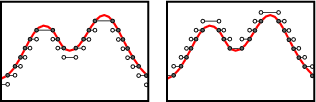}
\caption{Lower and upper step function approximations lead to different limits due to, in this case, endpoints.}
\label{fig:stepfunction}
\end{center}
\end{figure}

\begin{lemma}
\label{lem:r-simplex}
Given an affine function $h\in\Def(\sigma)$ on an open $k$-simplex
$\sigma$,
\begin{equation}
    \int_\sigma h\,\dchifloor = (-1)^k\inf_\sigma h
    \quad : \quad
    \int_\sigma h\,\dchiceil = (-1)^k\sup_\sigma h    .
\end{equation}
\end{lemma}
\begin{proof}
For $h$ affine on $\sigma$, $\chi\{\lfloor nh\rfloor > s\}=(-1)^k$ for all $s<\lfloor n\inf_\sigma h\rfloor$, and $0$ otherwise. One
computes
\[
    \lim_{n\to\infty}
    \frac{1}{n}\int_{\sigma}\lfloor nh\rfloor d\chi
    =
    \lim_{n\to\infty}
    \frac{1}{n}\sum_{s=0}^\infty \chi\{\lfloor nh\rfloor > s\}
    =
    (-1)^k\inf_\sigma h
    .
\]
The analogous computation holds with $\chi\{\lceil nh\rceil > s\}=(-1)^k$ for all $s<\lceil n\sup_\sigma h\rceil$, and $0$ otherwise.
\end{proof}

This integration theory is robust to changes in coordinates.

\begin{lemma}
\label{lem:cov}
Integration on $\Def(X)$ with respect to $\dchifloor$ and $\dchiceil$ is invariant under the right action of definable bijections of $X$.
\end{lemma}
\begin{proof}
The claim is true for Euler integration on $\CF(X)$; thus, it holds for $\int_X\lfloor nh\rfloor\,d\chi$ and $\int_X\lceil
nh\rceil\,d\chi$.
\end{proof}

\begin{lemma}
\label{lem:welldef}
The limits in Equation (\ref{eq:realintegral}) are well-defined.
\end{lemma}
\begin{proof}
An extension of the Triangulation Theorem to $\Def(X)$ \cite{vdD} states that to any $h\in\Def(X)$, there is a definable triangulation of $X$ on which $h$ is definably affine, meaning that there is a Euclidean simplicial complex $K$ and a definable homeomorphism $\phi\colon K\to X$ for which $h\circ\phi$ is affine on each open simplex of $K$. The result now follows from Lemmas \ref{lem:r-simplex} and \ref{lem:cov}.
\end{proof}

\subsection{Nonlinearity}
\label{sec:nonlin}

One is tempted to apply all the sheaf-theoretic perspectives of \S\ref{sec:dictionary} to $\int_X\colon\Def(X)\to\real$. However, for this formulation of the integral on $\Def(X)$, functoriality, and indeed, linearity, fails.

\begin{lemma}
\label{lem:nonlinear}
$\int_X\colon\Def(X)\to\real$ (via $\dchifloor$ or $\dchiceil$) is not a homomorphism for any $X$ with $\dim\ X>0$.
\end{lemma}
\begin{proof}
By explicit computation,
\[
    1 =
    \int_{[0,1]} 1\,\dchifloor
    \neq
    \int_{[0,1]} x\,\dchifloor + \int_{[0,1]} (1-x)\,\dchifloor
    = 1+1 = 2.
\]
\end{proof}

Nonlinearity of these integral operators is due to the fact that the floor and ceiling functions used are nonlinear: $\lfloor f+g \rfloor$ agrees with $\lfloor f \rfloor + \lfloor g \rfloor$ up to a set of Lebesgue measure zero, but not  Euler measure zero. In particular, the problems arise when increasing and decreasing functions compete. Though the change of variables formula (Lemma \ref{lem:cov}) holds, the more general Fubini theorem does not:
\begin{corollary}
\label{cor:fubinifail}
The Fubini theorem fails for $\int_X\colon\Def(X)\to\real$ (via $\dchifloor$ or $\dchiceil$) in general.
\end{corollary}
\begin{proof}
Let $F\colon X=Y\sqcup Y\to Y$ be the projection map with fibers $\{p\}\sqcup\{p\}$. Any $h\in\Def(X)$ is decomposable as $h=f\sqcup g$ for $f,g\in\Def(Y)$. The Fubini theorem applied to $F$ is equivalent to the statement
\[
    \int_Y f + \int_Y g = \int_X h = \int_Y F_*h = \int_Y f+g
\]
(where the integration is with respect to $\dchifloor$ or $\dchiceil$ as desired). Lemma \ref{lem:nonlinear} completes the proof.
\end{proof}

The Fubini theorem {\em does} hold when the map respects fibers.

\begin{theorem}
\label{thm:r-fubini}
For $h\in\Def(X)$, let $F\colon X\to Y$ be definable and $h$-preserving ($h$ is constant on fibers of $F$). Then $\int_Y F_*h\dchifloor = \int_X h\dchifloor$, and $\int_Y F_*h\dchiceil = \int_X h\dchiceil$.
\end{theorem}
\begin{proof}
By Theorem \ref{thm:Hardt}, $Y$ has a definable partition into tame sets $Y_\alpha$ such that $F\inv(Y_\alpha)$ is definably homeomorphic to $U_\alpha\times Y_\alpha$ for $U_\alpha$ tame, and that $F:U\times Y_\alpha\to Y_\alpha$ acts via projection. Since $h$ is constant on fibers of $F$, one computes
\[
    \int_{Y_\alpha}F_*h\dchifloor
    =
    \int_{Y_\alpha}h\,\chi(U_\alpha)\dchifloor
    =
    \int_{U_\alpha\times Y_\alpha}h\dchifloor .
\]
The same holds for $\int\dchiceil$.
\end{proof}

\begin{corollary}
\label{cor:r-fubini}
For $h\in\Def(X)$, $\int_Xh = \int_{\real}h_*h$. In other words,
\begin{equation}
    \int_Xh\,\dchifloor
    =
    \int_\real s\,\chi\{h=s\}\dchifloor ,
\end{equation}
and likewise for $\dchiceil$.
\end{corollary}

\begin{figure}[hbt]
\begin{center}
\includegraphics[angle=0,width=4in]{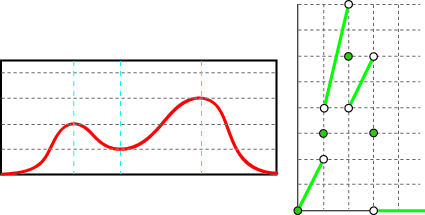}
\caption{The application of Corollary \ref{cor:r-fubini} to a smooth univariate function [left] gives a piecewise-linear integrand [right] with the same Euler integral.}
\label{fig:r-fubini}
\end{center}
\end{figure}

\subsection{Computation}
\label{sec:comput}

There are several definable analogues of computational formulae for integration over $\CF$. The following results parallel those of \S\ref{sec:integral}.

\begin{proposition}
\label{prop:r-valued}
For $h\in \Def(X)$,
\begin{eqnarray}
\label{eq:r-valued}
    \int_X h\,\dchifloor
    &=&
    \int_{s=0}^\infty \chi\{h\geq s\} - \chi\{h<-s\}\,ds
    \\
    \int_X h\,\dchiceil
    &=&
    \int_{s=0}^\infty \chi\{h> s\} - \chi\{h\leq -s\}\,ds
    .
\end{eqnarray}
\end{proposition}
\begin{proof}
For $h\geq 0$ affine on an open $k$-simplex $\sigma$,
\[
    \int_{\sigma} h\,\dchifloor
    = (-1)^k\inf_{\sigma} h
    = \int_0^\infty \chi(\sigma\cap\{h\geq s\})ds ,
\]
and for $h\leq 0$, the equation holds with $-\chi(\sigma\cap\{h < -s\})$. The result for $\int\dchiceil$ follows from an identical computation.
\end{proof}

It is tempting to guess that $\int_Xh\,\dchifloor=\int_0^\infty s\ \chi\{h=s\}\,ds$: the reader will easily verify that this is false. However, Equation (\ref{eq:level}) does have a definable analogue as follows.
\begin{proposition}
\label{prop:r-valued2}
For $h\in \Def(X)$,
\begin{eqnarray}
\label{eq:r-valued2}
    \int_X h\,\dchifloor
    &=&
    \lim_{\epsilon\to 0^+}\frac{1}{\epsilon}\int_{\real}
    s\,\chi\{s\leq h<s+\epsilon\}\,ds
\\
    \int_X h\,\dchiceil
    &=&
    \lim_{\epsilon\to 0^+}\frac{1}{\epsilon}\int_{\real}
    s\,\chi\{s< h\leq s+\epsilon\}\,ds .
\end{eqnarray}
\end{proposition}
\begin{proof}
For $h$ affine on an open $k$-simplex $\sigma$, and $0<\epsilon$ sufficiently small, $\int_{\real} s\,\chi\{s\leq
h<s+\epsilon\}\,ds = \epsilon\,(-1)^k\left(-\frac{\epsilon}{2}+\inf_{\sigma} h\right)$ and $\int_{\real} s\,\chi\{s< h \leq s+\epsilon\}\,ds =
\epsilon\,(-1)^k\left(-\frac{\epsilon}{2}+\sup_{\sigma} h\right)$.
\end{proof}

\subsection{Interpretation}
\label{sec:interpret}

The measures $\dchifloor$ and $\dchiceil$, though not equal, are clearly related. A measured perusal of the formulae derived in this section suggest a type of {\em duality} between them (sup versus inf; $\geq$ versus $<$). One general result reinforcing that duality is the following: these measures are related by conjugation with $-1$.
\begin{lemma}
\label{lem:floortoceiling}
\begin{equation}
\label{eq:floortoceiling}
    \int h\dchiceil = -\int -h\dchifloor .
\end{equation}
\end{lemma}
\begin{proof}
Triangulate. Note that $\sup_\sigma h=-\inf_\sigma -h$.
\end{proof}


As to the question of what these integrals measure, the following direct computation gives a clue: both integrals are generalizations of total variation.

\begin{corollary}
\label{cor:totvar}
If $M$ is a 1-dimensional manifold and $h\in\Def^0(M)$, then
\begin{equation}
\label{eq:totvar}
    \int_M h\,\dchifloor
    =
    -\int_M h\,\dchiceil
    =
    \frac{1}{2}\totvar(h) .
\end{equation}
\end{corollary}
\begin{proof}
Choose an affine triangulation of $h$ which triangulates $M$ with the maxima $\{p_i\}$ and minima $\{q_j\}$ as 0-simplices and the intervals between them as 1-simplices. To each minimum $q_j$ is associated two open 1-simplicies, since $M$ is a 1-manifold. Thus, via Lemma \ref{lem:r-simplex}:
\[
    \int_M h\,\dchifloor = \sum_ih(p_i)+\sum_jh(q_j)-2\sum_jh(q_j) =\frac{1}{2}\totvar(h) .
\]
This equals $-\int_M h\,\dchiceil$ by either an analogous computation or by Lemma \ref{lem:floortoceiling}.
\end{proof}

Equation (\ref{eq:totvar}) {\em concentrates} the measures $\dchifloor$ and $\dchiceil$ on the critical points of the (univariate) integrand. This is the key to understanding what happens in the higher-dimensional generalization of total variation.

\section{Morse theory}
\label{sec:morse}

Morse theory is a means of relating the global features of (in the classical setting) a Riemannian manifold $M$ with the local features of critical points of smooth $\real$-valued functions on $M$. Recall that $h:M\to\real$ is \style{Morse} if all critical points of $h$ are nondegenerate, in the sense of having a nondegenerate Hessian matrix of second partial derivatives. Denote by $\crit(h)$ the set of critical points of $h$. For each $p\in\crit(h)$, the \style{Morse index} of $p$, $\mu(p)$, is defined as the number of negative eigenvalues of the Hessian at $p$, or, equivalently, the dimension of the unstable manifold $W^u(p)$ of the gradient vector field $-\nabla h$ at $p$.

We will show, as a consequence of a more general approach, the following Morse-theoretic interpretation of $\int\cdot\dchifloor$:
\begin{theorem}
\label{thm:morse}
If $h$ is a Morse function on a closed $n$-manifold $M$, then:
\begin{align}
\label{eq:morse}
    \int_M h\,\dchifloor
    =
    \sum_{p\in\crit(h)}(-1)^{n-\mu(p)}h(p)
    ;
\\
\label{eq:morse2}
    \int_M h\,\dchiceil
    =
    \sum_{p\in\crit(h)}(-1)^{\mu(p)}h(p)
    .
\end{align}
\end{theorem}
\begin{proof}
The Euler characteristic of upper/lower excursion sets is piecewise-constant, changing only at critical values. Assume distinct critical values of $h$ (using localization for the more general case). For $p\in\crit(h)$, $s=h(p)$, and $\epsilon\ll 1$, classical Morse theory \cite{Milnor} says that $\{h\geq s+\epsilon\}$ differs from $\{h\geq s-\epsilon\}$ by the addition of a product of discs $D^{\mu(p)}\times D^{n-\mu(p)}$ glued along $D^{\mu(p)}\times \del D^{n-\mu(p)}$. The change in Euler characteristic resulting from this handle addition is $(-1)^{n-\mu(p)}$. Equation (\ref{eq:r-valued}) and duality complete the proof.
\end{proof}

A more general result arises from a more general Morse theory.  Stratified Morse theory \cite{GM} is a powerful extension to singular  spaces, including definable sets with respect to an o-minimal structure \cite{BK,Schurmann}. Let $X$ be a tame subset of $\real^n$ and $h\colon X\to\real$ a smooth map which extends to $\widetilde{h}\colon\real^n\to\real$ smooth. The \style{local Morse data} of $h$ at a point $x\in X$ is not a number (like the Morse index $\mu$) but rather (the homotopy type of) a pair of compact spaces:
\[
    \LMD(h,x) =
    \lim_{\epsilon'\ll\epsilon\to 0^+}
    \left(
        \overline{B_\epsilon(x)}
        \cap
        \{\abs{h-h(x)}\leq\epsilon'\}
    \, , \,
        \overline{B_\epsilon(x)}
        \cap
        \{h=h(x)-\epsilon'\}
    \right) ,
\]
where $\overline{B_\epsilon(x)}$ is the closed ball of radius $\epsilon$ about $x\in X$. The limit (in homotopy type) is well-defined (indeed, the homeomorphism type stabilizes) thanks to Theorem \ref{thm:Hardt}. The \style{Euler-Poincar\'e index} of $h$ at $x$ is, equivalently,
\begin{eqnarray*}
\label{eq:index}
    (\Index_*h)(x) &=& \chi(\LMD(h,x)) \\
        &=&
        1-
        \lim_{\epsilon'\ll\epsilon\to 0^+} \chi\left(\overline{B_\epsilon(x)}\cap\{h=h(x)-\epsilon'\}\right) \\
        &=&
        \lim_{\epsilon'\ll\epsilon\to 0^+}
        \chi\left(
        \overline{B_\epsilon(x)}
        \cap
        \{h>h(x)-\epsilon'\}
        \right) .
\end{eqnarray*}
There is a dual \style{co-index} given by $\Index^*h=\Index_*(-h)$. Note that $\Index_*,\Index^*\colon\Def(X)\to \CF(\overline{X})$, and the restriction of these operators to $\CF(X)$ is the identity (every point of a constructible function is a critical point). The two types of integration on $\Def(X)$ correspond to the Euler-Poincar\'e (co)index.

\begin{theorem}
\label{thm:stratified}
For $h\in\Def^0(X)$,
\begin{equation}
\label{eq:morseindex}
        \int_X h\,\dchifloor = \int_{\overline{X}}h\,\Index^*h\,d\chi
         \quad  ;  \quad
        \int_X h\,\dchiceil   = \int_{\overline{X}}h\,\Index_*h\,d\chi
.
\end{equation}
\end{theorem}
\begin{proof}
On an open $k$-simplex $\sigma\subset X\subset\real^n$ in an affine triangulation of $h$, the co-index $\Index^*h$ equals $(-1)^{\dim(\sigma)}$ times the characteristic function of the closed face of $\sigma$ determined by $\inf_\sigma h$. Since $h$ is
continuous, $\int_{\overline{\sigma}}h\Index^*h\,d\chi = (-1)^{\dim(\sigma)}\inf_\sigma h$. Lemma \ref{lem:r-simplex} and additivity complete the proof; duality gives a proof for $\Index_*$ and $\dchiceil$.
\end{proof}

This yields a concise proof of Theorem \ref{thm:morse}: for $p\in\crit(h)$ a nondegenerate critical point on an $n$-manifold, $\Index^*h(p)=(-1)^{n-\mu(p)}$ and $\Index_*h(p)=(-1)^{\mu(p)}$. Furthermore, one sees clearly that the relationship between $\dchifloor$ and $\dchiceil$ is regulated by Poincar\'e duality --- the canonical {\em flip} $\Index^*h=\Index_*(-h)$ is in play. For example, on continuous definable integrands over an $n$-dimensional manifold $M$,
\begin{equation}
\label{eq:morse3}
    \int_M h\,\dchiceil
    =
    (-1)^n\int_M h\,\dchifloor
    .
\end{equation}

The generalization from continuous to general definable integrands requires weighting $\Index^*h$ by $h$ directly. To compute $\int_X h\dchifloor$, one integrates the weighted co-index
\begin{equation}
        \lim_{\epsilon'\ll\epsilon\to 0^+}
        h(x+\epsilon')\chi\left(
        \overline{B_\epsilon(x)}
        \cap
        \{h<h(x)+\epsilon'\}
        \right)
\end{equation}
with respect to $d\chi$. Details are left to the curious reader.

It is worth noting that although the integral operators $\int \dchifloor$ and $\int \dchiceil$ are not linear, not functorial, and have seemingly little relationship to the sheaf-theoretic tools of \S\ref{sec:presheaves}-\ref{sec:dictionary}, it is nevertheless true that at heart these operators stem from stratified Morse theory, a subject born from and intrinsic to the techniques of constructible sheaves \cite{GM,Schurmann}.

\section{Incomplete data and harmonic interpolation}
\label{sec:incomplete}

Our first application of the definable integration theory concerns management of uncertain or incomplete data in estimation of an Euler integral over $\CF(\real^2)$. It is common in sensor networks to possess regions of undersampling (or {\em holes}) in the underlying domain, through incomplete coverage or node failures. In this case, one wants to estimate the number of targets relative to the missing information. This translates to the following relative problem: if one knows $h\colon X\to\nats$ only on some subset $A\subset X$, how well can one estimate $\int_X h\,d\chi$ from the restriction $h\vert_A$?

\subsection{Bounds}
\label{sec:bounds}
We give bounds for the planar case, following \cite{BG:RSS}.

\begin{theorem}
\label{thm:holes}
Assume $h\colon\real^2\to\nats$ is the sum of indicator functions over a collection of compact contractible sets in $\real^2$, none of which
is contained entirely within $D$, a fixed open contractible disc. Then
\begin{equation}
\label{eq:hole}
    \int_{\real^2} \hat{h}\, d\chi
    \leq
    \int_{\real^2} h\,d\chi
    \leq
    \int_{\real^2} \bigcheck{h}\, d\chi ,
\end{equation}
where
\begin{eqnarray*}
    \hat{h}(y) &=& \left\{
    \begin{array}{cl}
        \max_{\partial D} h & : y\in\overline{D} \\
        h & : {\mbox{else}}
    \end{array}
    \right.
    \\
    \bigcheck{h}(y) &=& \left\{
    \begin{array}{cl}
        \min_{\partial D} h & : y\in D \\
        h & : {\mbox{else}}
    \end{array}
    \right.
\end{eqnarray*}
\end{theorem}

\begin{proof}
Via additivity of $\chi$ over domains, Eqn. (\ref{eq:hole}) follows from the corresponding inequalities over the compact domain $\overline{D}$. Explicitly, if $\overline{h}=h$ on $\real^2-\overline{D}$, then
\[
\int_{\real^2}\overline{h}\,d\chi
    =
    \int_{\real^2-D}h\,d\chi
    -
    \int_{\partial D}h\,d\chi
    +
    \int_{\overline{D}}\overline{h}\,d\chi  .
\]
Denote by $\V=\{V_\beta\}$ the collection of nonempty connected components of intersections of all target supports $U_\alpha$ with
$\overline{D}$. Since we work in $\real^2$, each $V_\beta$ is a compact contractible set which intersects $\partial D$. By Theorem
\ref{thm:enumerate}, $\int_{\overline{D}}h\,d\chi$ equals the number of components $\abs{\V}$. There are at least $\max_{\partial D}h$
such pieces; hence
\[
\int_{\overline{D}}\hat{h}\,d\chi  \leq \int_{\overline{D}}h\,d\chi
.
\]

Consider $\min_{\partial D}h$ and remove from the collection $\V$ this number of elements, including all such $V_\beta$ equal to
$\overline{D}$ (which is possible since we remove at most $\min_{\partial D}h$ such elements). Each remaining $V_\beta\in\V$ is not equal to $\overline{D}$ and thus intersects $\partial D$ in a set with strictly positive Euler characteristic. Thus,
\[
    \int_{\overline{D}} h\,d\chi
    \leq
    \int_{\overline{D}}\bigcheck{h}\,d\chi.
\]
\end{proof}

\begin{figure}[hbt]
\begin{center}
\includegraphics[angle=0,width=4in]{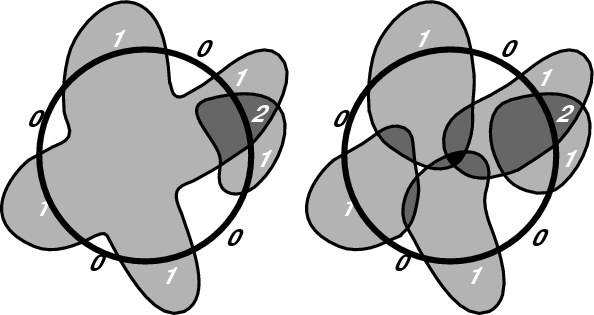}
\caption{An example for which the upper and lower bounds of Equation (\ref{eq:hole}) are sharp.}
\label{fig:hole}
\end{center}
\end{figure}

\begin{example}
\label{ex:hole}
Consider the example illustrated in Figure \ref{fig:hole}. The upper and lower estimates for the number of targets are $2$ and $4$ respectively. For this example, the estimates are sharp in that one can have collections of target supports over compact contractible sets which agree with $h$ outside of $\overline{D}$ and realize the bounds.
\end{example}

\begin{remark}
\label{rem:sharpness}
The lower bound $\int\hat{h}\,d\chi$ may fail to be sharp in several ways. For example, a target support can intersect $D$ in multiple components, causing $\hat{h}$ to not have a decomposition as a sum of characteristic functions over contractible sets (but rather with annuli). One can even find examples for which each target support intersects $\overline{D}$ in a contractible set but for which $\int\hat{h}\,d\chi$ is negative. The fact that the lower bound $\int\hat{h}$ can be so defective follows from the difficulty associated with annuli in the plane --- these are `large sets of measure zero' in $d\chi$.
\end{remark}

%
\begin{figure}[hbt]
\begin{center}
\includegraphics[angle=0,width=4in]{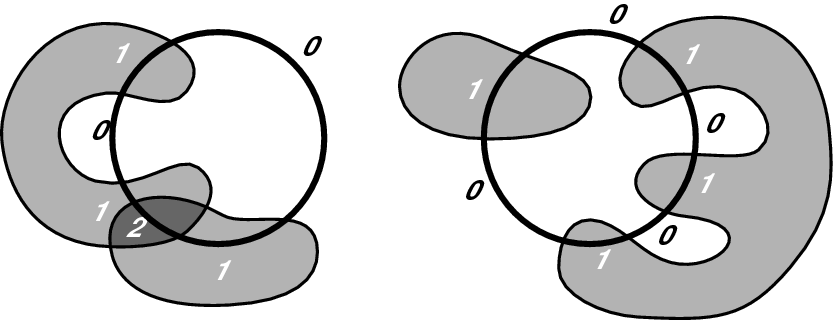}
\caption{Two examples in which the lower bound $\int\hat{h}\,d\chi$ fails: [left] the lower bound is $1$, but $\hat{h}$ cannot have a contractible support; [right] $\int\hat{h}\,d\chi$ is negative.}
\label{fig:holefail}
\end{center}
\end{figure}

The bounds of Theorem \ref{thm:holes} allow one to conclude when a hole is inessential and no ambiguity about the integral exists.

\begin{corollary}
\label{cor:tight}
Under the hypotheses of Theorem \ref{thm:holes}, the upper and lower bounds are equal when there is a unique connected local maximum of $h$ on $\partial D$.
\end{corollary}
\begin{proof}
In the case where $h$ is constant on $\partial D$, $\bigcheck{h}=\hat{h}$ and the result is trivial. Otherwise, both $\bigcheck{h}$ and $\hat{h}$ have connected (and thus contractible) upper excursion sets; thus,
\begin{eqnarray*}
\label{eq:diff}
    \int_{\real^2} \bigcheck{h}\, d\chi
    -
    \int_{\real^2} \hat{h}\, d\chi
    &=&
    \sum_s \chi\{\bigcheck{h}\geq s\} - \chi\{\hat{h}\geq s\}
    \\
    &=&
    \sum_s 1-1 = 0 .
\end{eqnarray*}
\end{proof}

\subsection{Harmonic extension and ``expected'' target counts}
\label{sec:harmonic}

We continue the results of the previous subsection, considering the case of a planar domain with a contractible hole on which the integrand is unknown. As shown, upper and lower bounds are realized by extending the integrand across the hole via minimal and maximal values on the boundary of the hole. Inspired by the result that the PL-extension of a discretely sampled integrand yields correct integrals with respect to Euler characteristic, we consider extensions over holes via {\em continuous} functions.

The following result says that there is a principled interpolant between the upper and lower extensions. Roughly speaking, an extension to a harmonic function (discrete or continuous, solved over the hole with Dirichlet boundary conditions) provides an approximate integrand whose integral lies between the bounds given by upper and lower convex extensions. Analytically harmonic functions are unneeded: any form of weighted averaging will lead to an extension which respects the bounds. A specific criterion follows.

\begin{theorem}
\label{thm:harmonic}
Given $h\colon\real^2-D\to\nats$ satisfying the assumptions of Theorem \ref{thm:holes}, let $\overline{h}$ be any definable extension of $h$ which has no strict local maxima or minima on $D$. Then
\begin{equation}
    \int_{\real^2} \hat{h}\, d\chi
    \leq
    \int_{\real^2} \overline{h}\,\dchifloor
    \leq
    \int_{\real^2} \bigcheck{h}\, d\chi ,
\end{equation}
\end{theorem}
\begin{proof}
Consider an open neighborhood of $\overline{D}$ in $\real^2$ and modify $\overline{h}$ so that it preserves critical values, is Morse, and falls off to zero quickly outside of $D$. This perturbed function, denoted $\widetilde{h}$, has isolated maxima on $\partial D$, isolated saddles in the interior of $D$ (since there are no local extrema in $D$ by hypothesis) and no other critical points outside of $D$. Since $\widetilde{h}$ is a small perturbation of $\overline{h}$, the integral of $\widetilde{h}$ with respect to $\dchifloor$ is equal to $\int_{\overline{D}}\overline{h}\,\dchifloor$. Via the Morse-theoretic formula of Equation (\ref{eq:morse}),
\[
    \int \widetilde{h}\,\dchifloor
    =
    \sum_{p\in\Crit(\widetilde{h})} (-1)^{2-\mu(p)}\widetilde{h}(p) .
\]
The integral thus equals the sum of $h$ over the maxima on $\partial D$ minus the sum of $\overline{h}$ over the saddle points in the interior of $D$, since saddles have Morse index $\mu=1$.

Denote by $\{p_i\}_1^M$ the maxima of $\widetilde{h}$, ordered by their (increasing) $\widetilde{h}$ values. Denote by $\{q_i\}_1^N$ the saddles of $\widetilde{h}$, ordered by their (increasing) $\widetilde{h}$ values. By the Poincar\'e index theorem,
\[
    1 = \chi(\overline{D}) =
    \#{\mbox{\rm maxima}}(\widetilde{h})
    -
    \#{\mbox{\rm saddles}}(\widetilde{h}) ,
\]
hence, $N=M-1$. Note that, since there are no local minima, $\widetilde{h}(q_i)<\widetilde{h}(p_i)$ for all $i=1\ldots M-1$. Thus,
\begin{eqnarray*}
    \int_{\overline{D}}\overline{h}\,d\chi
    &=&
    \int_{\overline{D}}\widetilde{h}\,d\chi
    \\
    &=&
    \widetilde{h}(p_M) + \sum_{i=1}^{M-1} \widetilde{h}(p_i)-\widetilde{h}(q_i)
    \\
    &\geq&
    \widetilde{h}(p_M)
=    \max_{\partial D}h
=    \int_{\overline{D}}\hat{h}\,d\chi .
\end{eqnarray*}
For the other bound,
\begin{eqnarray*}
    \int_{\overline{D}}\overline{h}\,d\chi
    &=&
    \widetilde{h}(p_M) + \sum_{i=1}^{M-1} \widetilde{h}(p_i)-\widetilde{h}(q_i)
    \\
    &\leq&
    \sum_{i=1}^{M} \widetilde{h}(p_i)
=    \int_{\overline{D}}\bigcheck{h}\,d\chi .
\end{eqnarray*}
\end{proof}

A harmonic or harmonic-like function $\widetilde{h}$ will often lead to an integral with non-integer value. One is tempted to interpret such an integral as an {\em expected} target count, though no formal notion of expected values for the Euler calculus as yet exists.

\begin{example}
\label{ex:grid}
Consider a hole $D$ and a function $h$ which is known only on $\partial D$ and which has two maxima with value $1$ and two minima with value $0$. Without knowing more about the possible size and shape of the target supports which make up $h$, it is not clear whether this is more likely to come from one target support (which crosses the hole) or from two separate target supports. Computing a harmonic extension of this $h$ over the interior of $D$ yields a function $\widetilde{h}$ with one saddle-type critical point in $D$. The value of the saddle is $c$ and satisfies $0<c<1$, depending on the geometry of $h$ on $\partial D$. This yields $\int\bar{h}\,d\chi=2-c$, reflecting the uncertainty of either one or two targets. In the perfectly symmetric case of Figure \ref{fig:harmonic}[left], $c=\frac{1}{2}$ and the integral is $\frac{3}{2}$, reflecting the balanced geometric ambiguity in target counts. In Figure \ref{fig:harmonic}[right], the harmonic extension has $c<\frac{1}{2}$, meaning that it is more likely that there are two target supports.
\end{example}

\begin{figure}[hbt]
\begin{center}
\includegraphics[angle=0,width=4.25in]{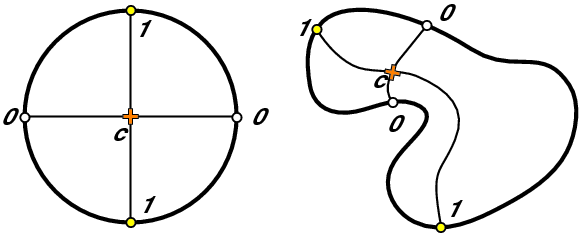}
\caption{An integrand with a hole has two minima at height $0$ and two maxima at height $+1$. Filling in by a harmonic function $\widetilde{h}$ has an interior saddle at height $0<c<1$, depending on the geometry of $h$ on $\partial D$: [left] $c=\frac{1}{2}$; [right] $c<\frac{1}{2}$.}
\label{fig:harmonic}
\end{center}
\end{figure}

In the network setting, holes often arise due to node failure or lack of sufficient node density. In these scenarios, one may reasonably employ any weighted local averaging scheme across dead nodes to recover a function which will respect the bounds of Theorem \ref{thm:holes}. Different weighting schemes may be more appropriate for different systems. For example, node readings can be assigned a {\em confidence measure}, which, when used as a weighting for the averaging over the dead zone, returns a value of the integral which reflects the fidelity of the data.

\begin{example}
We have assumed that targets are interrogated by a fixed network of stationary counting sensors: it is desirable to generalize to mobile sensors, especially in the context of robotics. Consider the following scenario: a collection of fixed target supports $\{U_\alpha\}$ lie in the plane. One or more mobile robots $R_i$ can maneuver in the plane along chosen paths $x_i(t)$, returning sensed counting functions $h_i(t)=\#\{\alpha: x_i(t)\in U_\alpha\}$. How should the paths $x_i$ be chosen so as to effectively determine the correct target count? If target supports are extremely convoluted, no guarantees are possible: therefore, assume that some additional structure is known (\eg, an injectivity radius) giving a lower bound on how ``thin'' the target supports may be. Assume that the robots initially explore the planar domain along a rectilinear graph $\Gamma$ that tiles the domain into rectangles. If desired, one can make these rectangles have either width or height small enough in order to guarantee that all the $U_\alpha$ intersect $\Gamma$. Consider the sensor function $h\colon\Gamma\to\nats$. The integral $\int_\Gamma h\,d\chi$ is likely to give the wrong answer, even (especially!) for a dense $\Gamma$. Two means of getting a decent approximation are (1) use the network approximation formula of Equation (\ref{eq:dual2d}); or (2) perform a harmonic extension over the holes of $\Gamma$.

Neither approach is guaranteed to give a good {\em a priori} approximation to the target count. How can one tell if $\Gamma$ should be filled in more? The simplest criterion follows from Corollary \ref{cor:tight}. Consider a basic cycle $\Gamma'\subset\Gamma$ in the tiling induced by $\Gamma$. If there is a single connected local maximum on $\Gamma'$, then (assuming that no small $U_\alpha$ lies entirely within the hole) the harmonic extension over $\Gamma'$ gives an accurate contribution to the integral.

If, on the other hand, there are multiple maximal sets on $\Gamma'$, then one must refine $\Gamma$ into smaller cycles for which the criterion holds. The obvious approach is to guide the mobile sensors so as to try and connect disjoint maxima and/or disjoint minima. Figure \ref{fig:grid} gives the sense of the technique. The crucial observation is that Corollary \ref{cor:tight} provides a stopping criterion.
\end{example}

\begin{figure}[hbt]
\begin{center}
\includegraphics[angle=0,width=3.5in]{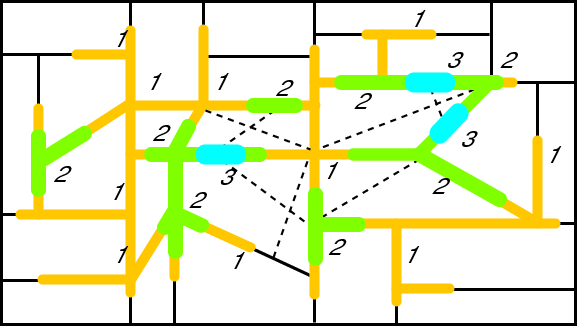}
\caption{Mobile agents determine target counts over a network. Holes with multiple maxima require further refinement (dashed lines).}
\label{fig:grid}
\end{center}
\end{figure}

\section{Index formulae for integral transforms}
\label{sec:index}

The definable Euler integration theory has important implications in the efficient computation of certain integral transforms from \S\ref{sec:hybrid}. Recall the Euler-Bessel transform $\Bessel\colon\CF(V)\to\Def(V)$ on a finite-dimensional normed real vector space $V$,  given by integrating along concentric spherical isospectral sets with respect to $d\chi$ and along the radius in Lebesgue measure. This transform, though acting on $\CF(V)$, is interpretable in terms of definable Euler integrals, and, hence, Morse theory. The following is from \cite{GR}.

\begin{lemma}
\label{lem:cone}
For $A\subset V$ a compact tame codimension-0 ball, star-convex with respect to $x\in A$,
\begin{equation}
    \Bessel\one_A(x) = \int_{\del A}d_x\,\dchifloor ,
\end{equation}
where $d_x$ is the distance-to-$x$ function $d_x:V\to[0,\infty)$.
\end{lemma}
\begin{proof}
By definition,
\[
            \Bessel h (x) = \int_0^\infty \chi(A\cap \del B_r(x)) dr .
\]
For $A$ a tame ball, star-convex with respect to $x$, $A\cap \del B_r(x)$ is homeomorphic to $\del A\cap\{d_x\geq r\}$. By Equation (\ref{eq:eulerexcursion}),
\[
            \Bessel h (x) = \int_0^\infty \chi(\del A\cap\{d_x\geq r\})dr = \int_{\del A} d_x\,\dchifloor .
\]
\end{proof}

This theorem is a manifestation of Stokes' Theorem: the integral of the distance over $\del A$ equals the integral of the `derivative' of distance over $A$. For non-star-convex domains, it is necessary to break up the boundary into positively and negatively oriented pieces. These orientations implicate $\dchifloor$ and $\dchiceil$ respectively.

\begin{theorem}
\label{thm:besseldchi}
For $A\subset V$ a codimension-0 submanifold with corners in $V$ and $x\in V$, decompose $\del A$ into $\del A=\del_x^+A\cup\del_x^-A$, where $\del_x^\pm A$ are the (closure of) subsets of $\del A$ on which the outward-pointing halfspaces  contain (for $\del_x^-$) or, respectively, do not contain (for $\del_x^+$) $x$. Then,
\begin{eqnarray}
\label{eq:besselindex}
    \Bessel\one_A(x)
            &=& \int_{\del_x^+A} d_x\, \dchifloor - \int_{\del_x^-A} d_x\, \dchiceil \\
            &=& \int_{\Crit_x\cap\del_x^+A} d_x\, \Index^*\, d\chi - \int_{\Crit_x\cap\del_x^-A} d_x\, \Index_*\, d\chi .
\end{eqnarray}
where $\Crit_x$ denotes the critical points of $d_x\colon\del A\to[0,\infty)$.
\end{theorem}
\begin{proof}
Assume, for simplicity, that $A$ is the closure of the difference of $C^+_x$, the cone at $x$ over $A^+_x$, and $C^-_x$, the cone over $A^-_x$, with the case of multiple cones following by induction. These cones, being star-convex balls with respect to $x$, admit  Lemma \ref{lem:cone}. The crucial observation is that, by additivity of $\chi$,
\[
        \chi(\del B_r(x) \cap A) = \chi(\del C^+_x\cap\{d_x\geq r\}) - \chi(\del C^-_x\cap\{d_x>r\}) .
\]
Integrating both sides with respect to $dr$ and applying Equation (\ref{eq:eulerexcursion}) gives
\[
        \Bessel\one_A(x) = \int_{\del C^+_x}d_x\,\dchifloor - \int_{\del C^-_x}d_x\,\dchiceil .
\]
By Theorem \ref{thm:stratified}, this reduces to an integral over the critical sets of the (stratified Morse) function $d_x$. The only critical point of $d_x$ on $C^+_x-\del A$ or $C^-_x-\del A$ is $x$ itself, on which the integrand $d_x$ takes the value $0$ and does not contribute to the integral. Therefore the integrals over the cone boundaries may be restricted to $\del^+A$ and $\del^-A$ respectively. The index-theoretic result follows from Theorem \ref{thm:stratified}.
\end{proof}

In even dimensions, the $\dchifloor$-vs-$\dchiceil$ dichotomy dissolves:

\begin{corollary}
\label{cor:evendim}
For $\dim\ V$ even and $A\subset V$ a codimension-0 submanifold with corners,
\begin{equation}
\label{eq:evendim}
    \Bessel\one_A(x) = \int_{\del A} d_x\, \dchifloor  = \int_{\Crit_x} d_x\,\Index_*\,d\chi .
\end{equation}
\end{corollary}
\begin{proof}
In this setting, $\del A$ is an odd-dimensional definable manifold. Equation (\ref{eq:morse3}) implies that $\int \dchiceil=-\int \dchifloor$. Equation (\ref{eq:besselindex}) completes the proof.
\end{proof}


Given the index theorem for the Euler-Bessel transform, that for the Euler-Fourier is a trivial modification that generalizes Example \ref{ex:fourierwidth}.

\begin{theorem}
\label{thm:fourierdchi}
For $A\subset V$ a codimension-0 submanifold with corners in $V$ and $\xi\in V^\vee-\{0\}$ a nonzero dual vector, decompose $\del A$ into $\del A=\del_x^+A\cup\del_x^-A$, where $\del_x^\pm A$ are the (closure of) subsets of $\del A$ on which $\xi$ points out of ($\del^+$) or into ($\del^-$) $A$. Then,
\begin{eqnarray}
\label{eq:fourierindex}
    \Fourier\one_A(\xi)
            &=& \int_{\del_{\xi}^+A} \xi\, \dchifloor - \int_{\del_{\xi}^-A} \xi\, \dchiceil \\
            &=& \int_{\Crit_\xi\cap\del_{\xi}^+A} \xi\, \Index^*\, d\chi - \int_{\Crit_\xi\cap\del_\xi^-A} \xi\, \Index_*\, d\chi .
\end{eqnarray}
where $\Crit_\xi$ denotes the critical points of $\xi\colon\del A\to[0,\infty)$. For $\dim\ V$ even, this becomes:
\begin{equation}
    \Fourier\one_A(\xi) = \int_{\del A} \xi\, \dchifloor  = \int_{\Crit_\xi} \xi\,\Index_*\,d\chi .
\end{equation}
\end{theorem}

The proof follows that of Theorem \ref{thm:besseldchi} and is an exercise. Figure \ref{fig:index} gives a simple example of the Bessel and Fourier index theorems in $\real^2$.

\begin{figure}[htb]
\begin{center}
\includegraphics[width=5.0in]{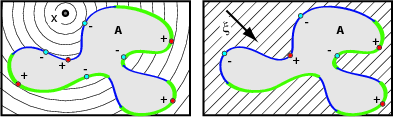}
\caption{The index formula for $\Bessel$ [left] and $\Fourier$ [right] applied to $\one_A$ localizes the transform to (topological or smooth) tangencies of the isospectral sets with $\del A$.}
\label{fig:index}
\end{center}
\end{figure}

\begin{example}
By linearity of $\Bessel$ and $\Fourier$ over $\CF(V)$, one derives index formulae for integrands in $\CF(V)$ expressible as a linear combination of $\one_{A_i}$ for $A_i$ the closure of definable bounded open sets. For a set $A$ which is not of dimension $\dim\ V$, it is still possible to apply the index formula by means of a limiting process on compact tubular neighborhoods of $A$. For example, let $A\subset\real^2$ be a compact straight line segment: see Figure \ref{fig:interval}. Let $A_0$ and $A_1$ denote the endpoints of $A$. The reader may compute directly that
\[
    \Bessel\one_A(x) =  d(x,A_0) + d(x,A_1) - 2d(x,A),
\]
where $d(x,A)$ denotes the distance from $x$ to the segment $A$. When one of the endpoints minimizes this distance, one gets $\Bessel\one_A(x)=\max_{\del A}d_x - \min_{\del A}d_x$, exactly as Equation (\ref{eq:besselindex}) would suggest. This correspondence seems to fail when there is a tangency between the interior of $A$ and the isospectral circles, as in Figure \ref{fig:interval}[right]: why the factor of $2$? The index interpretation is clear, however, upon taking a limit of neighborhoods, in which case a pair of negative-index tangencies are revealed.

\begin{figure}[bht]
\begin{center}
\includegraphics[width=5in]{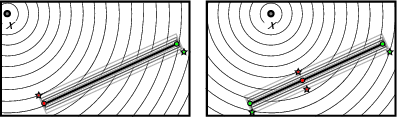}
\caption{The Euler-Bessel Transform of a line segment in $\real^2$ has an index formula determined by $d_x$ at the endpoints and at an interior tangency. This follows from Theorem \ref{thm:besseldchi} by limits of compact tubular neighborhoods.}
\label{fig:interval}
\end{center}
\end{figure}
\end{example}

These index formulae are useful in explaining the geometric content of the Fourier and Bessel transforms.

\begin{corollary}
\label{cor:monotone}
For $A$ a compact round ball about $p\in\real^{2n}$, the Bessel transform $\Bessel\one_A$ is a nondecreasing function of the distance to $p$, having unique zero at $p$.
\end{corollary}
\begin{proof}
Convexity of balls and Corollary \ref{cor:evendim} implies that
\[
        \Bessel\one_A(x) = \int_{\del A} d_x\,\dchifloor = \max_{\del A}d_x - \min_{\del A}d_x ,
\]
which equals $\diam A$ for $x\not\in A$ and is monotone in distance-to-$p$ within $A$.
\end{proof}

Note that Corollary \ref{cor:monotone} is vacuous in odd dimensions; the Bessel transform of a ball in $\real^{2n+1}$ is constant, and $\Bessel$ obscures all information. However, for even dimensions, Corollary \ref{cor:monotone} provides a basis for target localization and shape detection. For targets with convex supports (regions detected by counting sensors), the local minima of the Euler-Bessel transform can reveal target locations: see Figure \ref{fig:besselloc}[left] for an example. Note that in this example, not all local minima are target centers: interference creates ghost minima. However, given $h\in\CF(V)$, the integral $\int_Vh\,d\chi$ determines the number of targets. This provides a guide as to how many of the deepest local minima to interrogate. The deepest minima correspond to perfectly round targets: this gives a basis for performing shape discrimination based solely on enumerative data, as in Figure \ref{fig:besselloc}[right].

\begin{figure}[hbt]
\begin{center}
\includegraphics[width=5.0in]{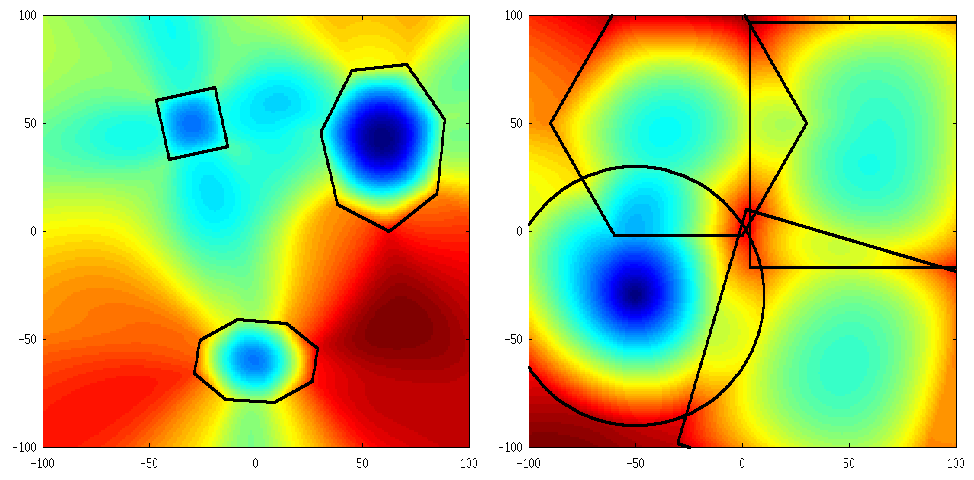}
\caption{The Euler-Bessel transform of a collection of convex targets [left] has local minima at the target centers. The deepest minima in the Bessel-transform correspond to round targets, and may be useful for shape detection [right].}
\label{fig:besselloc}
\end{center}
\end{figure}

There are significant limitations to superposition by linearity for this application. When targets are nearby or overlapping, their individual transforms will have overlapping sidelobes, which results in uncertainty when the transform is being used for localization.

\section{Integral transforms with definable kernels}
\label{sec:defker}

The previous subsection examined integral transforms over $\CF$ by means of index theory. There are numerous open questions concerning integral transforms over $\Def(X)$.

\subsection{Continuity}
\label{sec:cont}

Though the integral operators with respect to $\dchifloor$ and $\dchiceil$ are not linear on $\Def(X)$, they nevertheless retain some nice properties. The following remarks are adapted from \cite{BG:PNAS}. All properties below stated for $\int\,\dchifloor$ hold for $\int\,\dchiceil$ via duality.

\begin{lemma}
\label{lem:homog}
The integral $\int\,\dchifloor\colon\Def(X)\to\real$ is positively homogeneous.
\end{lemma}
\begin{proof}
For $f\in \Def(X)$ and $\lambda\in\real^+$, the change of variables variables $s\mapsto \lambda s$ in Eqn. [\ref{eq:r-valued}] gives $\int\lambda f\,\dchifloor = \lambda\int f\,\dchifloor$.
\end{proof}

\begin{example}
Integration is not continuous on $\Def(X)$ with respect to the $C^0$ topology. An arbitrarily large change in $\int h\dchifloor$ may be effected by small changes to $h$. An example appears in \cite{Wright}: consider $h(x,y)=1-2\abs{x-\frac{1}{2}}\in\Def([0,1]^2)$ as in Figure \ref{fig:tent}[left]. Since there is one maximal value at $1$ along the compact interval $\{x=\frac{1}{2}\}$, the integral $\int h\dchifloor = 1$.
However, sampling over a triangulation that does not sample the topology of the maximal set correctly yields a number of maxima and saddles. It may be arranged (as in Figure \ref{fig:tent}[right]) so that successively fine triangulations cause the net variation to blow up to infinity, even as the graphs converge pointwise.
\begin{figure}[hbt]
\begin{center}
\includegraphics[width=5.0in]{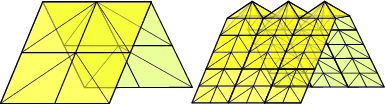}
\caption{The integrand $h$ has PL approximations that converge pointwise, but diverge in integration with respect to $\dchifloor$.}
\label{fig:tent}
\end{center}
\end{figure}
In some situations the ``complexity'' of the definable functions can be controlled in a way sufficient to ensure continuity \cite{BG:PNAS}. The above example fails if sufficiently smoothed. It remains an interesting open question to explore topologies on $\Def(X)$ and the relationship to continuity of Euler integration.
\end{example}

\subsection{Duality}
\label{sec:duality}
The duality operator $\Dual\colon\CF(X)\to\CF(X)$ of \S\ref{sec:conv} extends seamlessly to integrals on $\Def(X)$ in the obvious manner:
\begin{equation}
\label{eq:defdual}
    \Dual h(x)
    =
    \lim_{\epsilon\to 0^+}\int_Xh\one_{B_\epsilon(x)}\dchifloor
    =
    \lim_{\epsilon\to 0^+}\int_Xh\one_{B_\epsilon(x)}\dchiceil ,
\end{equation}
where $B_\epsilon$ is an open metric ball of radius $\epsilon$.

\begin{lemma}
\label{lem:dualindependent}
Duality $\Dual\colon\Def(X)\to\Def(X)$ is well-defined and independent of whether the integration in (\ref{eq:defdual}) is with respect to $\dchifloor$ or $\dchiceil$.
\end{lemma}
\begin{proof}
The limit is well-defined thanks to the Hardt Theorem. To show that it is independent of the upper- or lower-semicontinuous approximation, take $\epsilon>0$ sufficiently small. Note that by triangulation, $h$ can be assumed to be piecewise-affine on open simplices. Pick a point $x$ in the support of $h$ and let $\{\sigma_i\}$ be the set of open simplices whose closures contain $x$. Then for each $i$, the limit $h_i(x):=\lim_{y\to x}h(y)$ for $y\in\sigma_i$ exists. One computes
\begin{equation}
\label{eq:dualsimplices}
    \Dual h(x)
    =
        \sum_i (-1)^{\dim\sigma_i}h_i(x) ,
\end{equation}
independent of the measure $\dchifloor$ or $\dchiceil$.
\end{proof}

For a continuous definable function $h$ on a manifold $M$, $\Dual h = (-1)^{\dim\ M}h$, as one can verify. This is commensurate with the result of Schapira \cite{Schapira:op} that $\Dual$ is an involution on $\CF(X)$.

\begin{theorem}
Duality is involutive on $\Def(X)$: $\Dual\circ\Dual h = h$.
\end{theorem}
\begin{proof}
Given $h$, fix a triangulation on which $h$ is affine on open simplices. Note that the dual of $h$ at $x$ is completely determined by the trivialization of $h$ at $x$. Let $L_x{h}$ be the constructible function on $B_\epsilon(x)$ which takes on the value $h_i(x)$ on strata $\sigma_i\cap B_\epsilon(x)$. (Though this is not necessarily an integer-valued function, its range is discrete and therefore is constructible.) As $L_xh$ is close to $h$ in $B_\epsilon(x)$ (this follows from the continuity of $h$ on each of the strata), $\Dual h$ is close to $\Dual L_xh$ in $B_\epsilon(x)$: indeed, the total Betti number of intersections of strata with any ball $B_\epsilon(y)$ is bounded, and Euler integral of a function small in absolute value is small as well. Hence the definable function $\Dual^2 h$ is close to the constructible function $\Dual L_xh$ with $\epsilon$ small. As $\Dual^2L_x{h}(x)=L_x{h}(x)=h(x)$, the result follows.
\end{proof}

Duality may be used to define a link operator on definable functions as
\begin{equation}
\label{eq:link}
    \Lambda h(x)
    =
    \lim_{\epsilon\to 0^+}\int_Xh\one_{\partial B_\epsilon(x)}d\chi.
\end{equation}
The link of a continuous function on an $n$-manifold $M$ is multiplication by $1+(-1)^n$, as a simple computation shows. In general, $\Lambda = \id - \Dual$, where $\id$ is the identity operator.

\subsection{Linearity}
\label{sec:linearity}

The nonlinearily of the integration operator prevents most straightforward applications of Schapira's inversion formula. Fix a kernel $K\in\Def(X\times Y)$ and consider the general integral transform $\Transform_K\colon\Def(X)\to\Def(Y)$ of the form $(\Transform_K h)(y)=\int_Xh(x)K(x,y)\dchifloor(x)$. In general, this operator is non-linear, via Lemma \ref{lem:nonlinear}. However, some vestige of (positive) linearity survives within $\CF^+$, the {\em positive} linear combinations of indicator functions over tame top-dimensional subsets of $X$.

\begin{lemma}
\label{lem:CF^+}
The integral transform $\Transform_K$ is positive-linear over $\CF^+(X)$.
\end{lemma}
\begin{proof}
Any $h\in \CF^+(X)$ is of the form $h=\sum_k a_k\one_{U_k}$ for $a_k\in\nats$ and $U_k\in\Def(X)$. For $h=\one_A$, $\Transform_K h = \int_A K \dchifloor$. Additivity of the integral in $\dchifloor$ combined with Lemma \ref{lem:homog} completes the proof.
\end{proof}

This implies in particular that when one convolves a function $h\in \CF^+(\real^n)$ with a smoothing kernel (\eg, a Gaussian) as a means of filtering noise or taking an average of neighboring data points, that convolution may be analyzed one step at a time (decomposing $h$).
The restriction to positive-linearity is critical, since $\int -h\dchifloor \neq -\int h\dchifloor$. However, integral transforms which combine $\dchifloor$ and $\dchiceil$ compensate for this behavior. Define the measure $[d\chi]$ to be the average of $\dchifloor$
and $\dchiceil$:
\begin{equation}
    \int_X h[d\chi] = \frac{1}{2}\left(\int_X h\dchifloor + \int_X h\dchiceil\right) .
\end{equation}

\begin{theorem}
\label{thm:lineartransform}
Any integral transform of the form
\begin{equation}
    (\Transform_K h)(y)
    =  \int_X h(x)K(x,y) [d\chi](x)
\end{equation}
for $K\in\Def(X\times Y)$ is a linear operator $\CF(X)\to\Def(Y)$.
\end{theorem}
\begin{proof}
From Lemma \ref{lem:CF^+}, $\Transform$ is positive-linear over $\CF^+(X)$. Full linearity follows from the observation that $\int_{X} -h [d\chi]  = -\int_{X} h [d\chi]$, which follows from Lemma \ref{lem:r-simplex} by triangulating $h$.
\end{proof}

\begin{example}
Consider the transform with $X=\real^n$, $Y=\Sphere^{n-1}\subset X$, and kernel $K(x,\xi)=\langle x,\xi\rangle$. This transform with respect to $[d\chi]$ applied to $\one_A$ for $A$ compact and convex returns the `centroid' of $A$ along the $\xi$-axis: the average of the maximal and minimal values of $\xi$ on $\partial A$. Note how the dependence on critical values of the integrand on $\partial A$ reflects the Morse-theoretic interpretation of the integral in this case. The interested reader may wish to derive a more general index-theoretic result over $\CF(X)$ using the ideas of \S\ref{sec:morse}.
\end{example}

Integration with respect to $[d\chi]$ seems suitable only for integral transforms over $\CF$, since, on a continuous integrand, the integral with respect to $[d\chi]$ either returns zero (cf. the Rota-Chen definition of \S\ref{sec:RC}) or else the integral with respect to $\dchifloor$, depending on the parity of the $\dim\ X$.

\subsection{Convolution}
\label{sec:convolution}

Recall the convolution operator from \S\ref{sec:conv}: $(f*g)(x) = \int_V f(t)g(x-t)\,d\chi$. Convolution is well-defined on $\Def(V)$ by integrating with respect to $\dchifloor$ or $\dchiceil$. However, the product formula (Lemma \ref{lem:convprod}) for $\int f*g$ fails in general, since one relies on the Fubini theorem to prove it in $\CF(V)$. By Lemma \ref{lem:CF^+}, convolution with a definable $g\in\Def(V)$ {\em is} positive-linear over $\CF^+(V)$. We indicate (following \cite{BG:PNAS}) how to use this to smooth noise in an integrand $h\in\CF^+(V)$.

Integration over $\CF^+(X)$ is poorly behaved with respect to noise (Figure \ref{fig:eucharis4}[left]), owing to the fact that points have full measure in $d\chi$. Assume a sampling of $f\in\CF^+(X)$ over a network, with an error of $\pm 1$ on random nodes: specifically, $h=f+e$, where $e\colon\Nodes\to\{-1,0,1\}$ is an error function that is nonzero on a sparse subset of nodes $\Nodes'\subset\Nodes$. For typical choices of $\Nodes'$, $\abs{\int h\, d\chi - \int f\, d\chi}$ will be a normal of variance $O(\abs{\Nodes'}^\frac{1}{2})$.

\begin{figure}[hbt]
\begin{center}
\includegraphics[angle=0,width=5.0in]{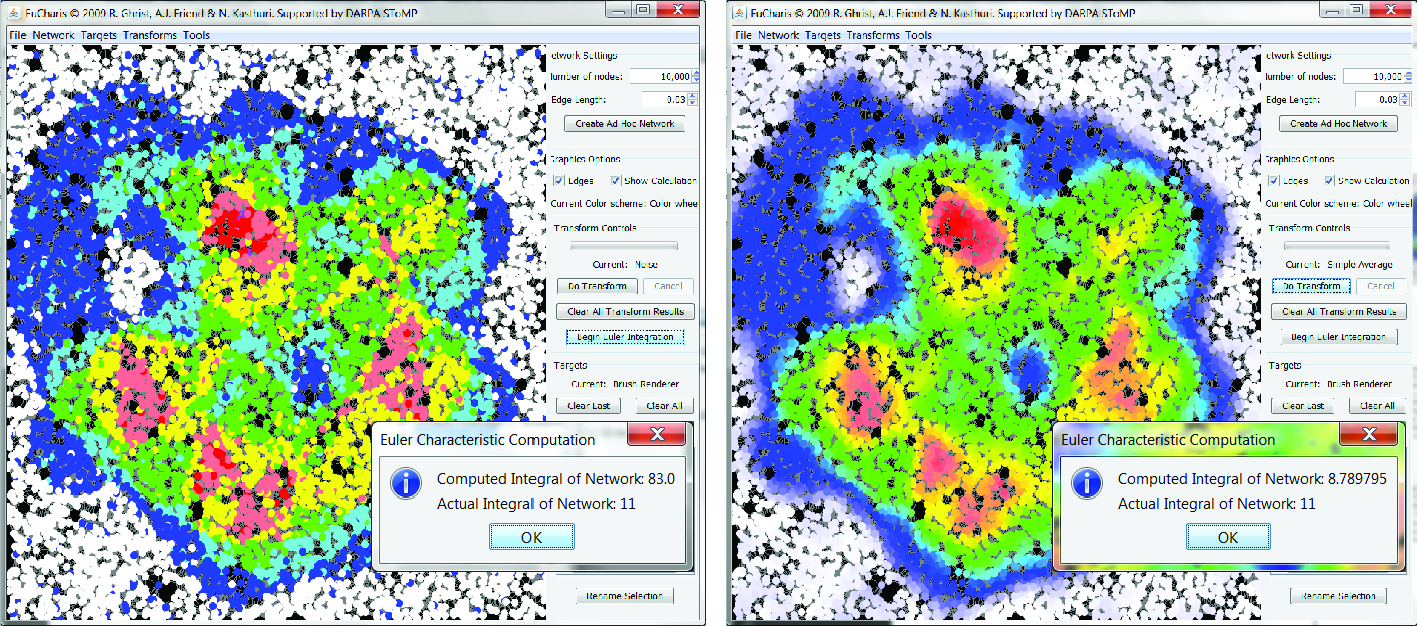}
\caption{Screenshots of {\em Eucharis}, demonstrating the deleterious impact of noise [left] and the resulting mitigation via convolution with a bump function via neighbor-averaging [right]. The true integral ($11$) is poorly approximated ($83$) when the integrand has a $10\%$ noise added, but is more reasonably approximated ($8.8$) after convolution.}
\label{fig:eucharis4}
\end{center}
\end{figure}

One strategy for mitigating noise is to dissipate via convolution with a smoothing kernel. For discrete data, this is best accomplished via a weighted average of neighboring node values, with weights inverse to distance (hop metric for the network). Such an averaging, for a sufficiently dense network of randomly paced nodes, approximates a convolution of the constructible data with a bump function. This poses the problem of when such a convolution returns the true integral. Assume a convolution $h*\Kernel$ of $h\in\CF^+(\real^n)$ with an appropriate kernel $\Kernel$, unimodal and of appropriately small support relative to $h$. To quantify the optimal size of  $\Kernel$, we use a characteristic length (related to the \style{weak feature size} of \cite{Chazal}) which encodes the fragility of an integrand $h\in \CF^+(\real^n)$ with respect to $\dchifloor$. Define the \style{constructible feature size} of $h\in \CF^+(\real^n)$ at $x\in\real^n$ to be
$\CFS_x(h)$, the supremum of all $R$ such that, for any closure $\Chamber$ of any connected component of upper or lower excursion sets of $h$, the convex hull of all outward-oriented normals to $\del\Chamber\cap B_R(x)$ contains $0$: see Figure \ref{fig:CFS}. Minimizing over $x$ yields $\CFS(h)=\inf_x\CFS_x(h)$. Constructible feature size regulates the impact of a smoothing kernel, whether coming from data diffused at the hardware level, or purposefully smoothed to mitigate noise.

\begin{figure}[hbt]
\begin{center}
\includegraphics[angle=0,width=4.5in]{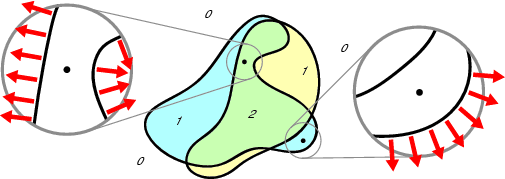}
\caption{Constructible feature size $\CFS$ of $h\in\CF(\real^n)$ detects fragility in excursion sets' topology.}
\label{fig:CFS}
\end{center}
\end{figure}

\begin{theorem}[\cite{BG:PNAS}]
\label{thm:conv=}
For $h\in \CF^+(\real^n)$ and $\Kernel\in\Def^1(\real^n)$ a radially-symmetric kernel with support of radius $R<\CFS(h)/2$,
\begin{equation}
\label{eq:conv=}
    \int_{\real^n} h*\Kernel \dchifloor
    =
    \int_{\real^n} h\,d\chi   .
\end{equation}
\end{theorem}
\begin{proof}
Note that $h*\Kernel\in\Def^1(\real^n)$. By Lemma \ref{lem:CF^+}, $h*\Kernel$ is the sum over $\alpha$ of $\one_{U_\alpha}*\Kernel$. Each such convolution is a smoothing of $\one_{U_\alpha}$ on an exterior tubular neighborhood of radius $R$ of $U_\alpha$. The integral of each $\one_{U_\alpha}*\Kernel$ is thus unchanged. Lack of linearity for $\int\dchifloor$ forbids concluding that the integral of the full $h*\Kernel$ is unchanged. It does suffice, however, to show that no changes in critical points arise in summing these $R$-neighborhood smoothed supports.

Consider $x$ in the intersection of the exterior $R$-neighborhoods of (after re-indexing) a collection $\{U_\beta\}$. The gradient $\nabla_x(h*\Kernel)$ is a weighted vector sum over $\beta$ of $\nabla_x(\one_{U_\beta}*\Kernel)$. By radial symmetry of $\Kernel$, each such vector is parallel (and oppositely oriented) to the normal from $x$ to $U_\beta$. By definition of $\CFS$, this sum cannot be zero, and thus no new critical points arise in the convolution $h*\Kernel$. Invoking Theorem \ref{thm:stratified} and verifying that the indices of previously existing critical sets are unchanged by smoothing, one has that the integrals agree.
\end{proof}

Thus, integrating the (convolved) intensity data in $\dchifloor$ returns the $d\chi$ integral of the (true) constructible data, when feature size is sufficiently large. Stochastic geometry techniques reveal that typical feature size is small \cite{BG:PNAS}. More recent results \cite{BoBo} give strong results on convolution formulae for Gaussian smoothing of noisy integrands.

\
\specialsection*{\bf Closing Thoughts}
\sweetline
\vspace{0.1in}

\section{Open questions}
\label{sec:open}

\subsection{Sampling and estimation}
The success of Alexander duality in computing Euler integrals over {\em ad hoc} planar networks as in \S\ref{sec:adhoc} stands in marked contrast to situation in dimensions three and higher. In dimension three, {\em ad hoc} networks are still a sensible way to deploy sensors to sample a function, motivating the need for numerically efficient and stable algorithms for computing Euler integrals. Alexander duality permits one to pair $H_0$ and $H_2$, thus avoiding the need to determine voids in the sampling by means of clustering; however, $H_1$ is self-dual, and this is a nontrivial difficulty. {\em Is there a fast method for estimating Euler integrals over a network-sampled integrand on $\real^k$ for $k>2$?}

\subsection{Numerical analysis}
This survey leaves open the broader question of rigorous computation and estimation: {\em is there a sensible numerical analysis for the Euler calculus?} It seems that classical tools and techniques for approximating integrals fail to hold in the Euler calculus. The typical {\em sine qua non} --- convergence of integration with arbitrarily dense discrete sampling --- can fail, as shown in \S\ref{sec:numerical}. Any sceptic would feel justified in giving up. Nevertheless, we believe that the potential for applications demands the careful investigation of numerical issues in Euler calculus. It is our belief (offered without proof!) that a well-defined and applicable numerical analysis for Euler calculus exists and is worthwhile.

\subsection{Efficient integral transforms}
The potential of Euler-based integral transforms seems very significant in image and radar signal processing. Besides the difficulties of determining inverse transforms, etc., the sheer numerical complexity of computation is an impediment: {\em are there fast algorithms for computing and inverting Euler integral transforms?} As described in \S\ref{sec:index}, one way to dramatically improve the burden of computing certain Euler integral transforms is to specifically exploit the underlying index theory, focusing the transforms onto critical sets. As a small step in this direction, it is possible to construct explicit formulae for the integral transforms for geometric figures in, \eg, the Euler-Bessel and Fourier transform. In computer graphics contexts, where a (usually semialgebraic) description of sets exist, it may be efficient to exploit critical point computations on surfaces of positive codimension.  This would enable dramatically less computation than computing Euler integrals directly, and may also incur less numerical error.

\subsection{DSP and Euler calculus}
One of the major enablers in the signal processing revolution was the discovery of the Fast Fourier Transform (FFT).  By recursively interleaving the domain, a discrete-time signal's Fourier transform can be computed in $O(N \log N)$ operations rather than $O(N^2)$.  The crucial insight is that the transform can be quickly computed on an interval of length $N_1 N_2$ by first computing it on the much shorter intervals of length $N_1$ and $N_2$. {\em What is the Eulerian FFT?} It is presently unclear if such a reduction in computation can be made for the Euler integral on regularly-sampled domains: the underlying mechanism for combining the interleaved computations in the usual FFT is orthogonality, and the Euler product does not support the usual synthesis formula.  As a result, the relationship of a function's Euler integral to subsampled copies is likely to be complicated. That there is any possibility of an Eulerian FFT at all comes from the scissors relation for definable sets: definable level sets of a function on a regular grid have a decomposition (via a definable homeomorphism) that may enable computational simplifications.

\subsection{Euler wavelets}
As shown in \S\ref{sec:wavelet}, the Euler-Haar wavelet transform is injective, and can be used to synthesize a function from its transform algorithmically.  However, the primary use of the analysis-synthesis properties of transforms is to enable filtering.  {\em How do particular classes of Euler wavelet-based filters affect a given signal?} The inversion algorithm we have exhibited is delicate: not every set of wavelet coefficients corresponds to a transform of a function, which could cause nontermination of the algorithm.  Worse, since we do not have an explicit characterization of what the family of appropriate wavelet coefficients are, naive filtering may be destructive. In the definable category of functions, especially, there are likely many other wavelet families to be studied. As in the Lebesgue case, Haar wavelets represent the crudest (and not terribly effective) family. It is to be hoped that other Euler wavelet families that have substantially better properties may be found.

\subsection{Data fusion}
The Euler integral and its associated operations arise from a subtle connection to the sheaf of constructible functions. These functions take values in the integers, which provides much of the motivation for considering integer-valued data in this context.  Substantially more complicated algebraic data can be stored into sheaves; appropriately generalized integration theories {\em may} follow. It remains an open problem: {\em how to fuse disparate data types over a space (or network) into a consistent global datum?} These data types may be as simple as $\zed$ versus $\real$ (some sensors count; others record intensities); more complex data types (local logical statements about targets) would be of phenomenal impact. In the latter case, one suspects a relationship with geometric logic and topoi.

\subsection{Target classification}
The Euler calculus may provide a promising way to address the problem of identifying characteristics of a particular family of target supports. {\em How can one perform target identification with only enumerative data?} For instance, the number of wheels on a vehicle may be an identifying characteristic that can be counted by the Euler integral if applied to field of pressure sensors or remote imaging data.  More refined applications can be envisioned that use the integral transforms we have described, most notably the Bessel and Fourier transforms \cite{GR}.  In particular, families of invariants might be extracted from transforms of functions (rather than the Euler integral alone) that would aid in robust identification even if errors are present.

\subsection{Lifting topological invariants}
The Euler integral is a lift of the classical Euler characteristic $\chi$ from an invariant of spaces to an invariant of {\em data} (constructible or definable) over spaces. The success of this lifting prompts an examination of other classical topological invariants. One prime candidate is the Lusternik-Schnirelmann category of a space \cite{CLOT}, which has been lifted recently to an invariant of $\real$-valued distributions \cite{BG:uni}. Others are more geometric: there is a valuation $\mu_k$ on tame subsets of $\real^n$, one for each $k=0,\ldots,n$ that interpolate between Euler ($\mu_0$) and Lebesgue ($\mu_n$) measure. These \style{intrinsic volumes}\footnote{Also known as Hadwiger measures, mixed volumes, quemasseintegrals, Lipshitz-Killing curvatures, etc.} have been lifted not only to $\CF(\real^n)$ but also to $\Def(\real^n)$ in dual pairs $\lfloor\mu_k\rfloor$, $\lceil\mu_k\rceil$ \cite{Wright}. {\em Are there other algebraic invariants that can be lifted from spaces to data over spaces? Is it possible to lift algebraic-topological invariants of maps between spaces (\eg, Lefshetz index) to morphisms between data structures over spaces?}

\subsection{Applied sheaf theory}
Euler calculus is the vanguard of an emerging family of techniques in applied mathematics based on constructible sheaves. Any locally-defined algebraic data has a sheaf (or cosheaf) interpretation, and constructible sheaves are especially suited to applications due to their ease of manipulation and computation.  Sheaf cohomology permits local-to-global inference and is incisive, even in simple examples. We suggest that many scientifically-motivated problems that exhibit locality of information can be addressed via the cohomology of an appropriate sheaf. For instance, in communications networks, nodes and links typically only retain information about messages from their immediate neighbors: information in the network is local. Work in progress \cite{Rob:logic,Rob:flow,GH} indicates a sheaf-theoretic interpretation of network information, with sheaf cohomology measuring features of messages passed through the network at large. Examples with higher-dimensional base spaces would permit the use of sheaf cohomology beyond $H^0$ and $H^1$. {\em What are other examples in which higher-dimensional constructible sheaf cohomology classes solve data aggregation problems?}

\vspace{0.1in}

\sweetline

\vspace{0.1in}

{\small
\section*{Acknowledgements}

The applications described in this survey would not have been possible without the guidance and hard work of two individuals.  Yuliy Baryshnikov was the first to appreciate applications of Euler calculus to sensing and is directly or jointly responsible for most of the  applications appearing in this survey. Benjamin Mann oversaw the DARPA program {\em SToMP: Sensor Topology \& Minimalist Planning} under which the sensing applications of Euler calculus were developed.

\bibliographystyle{amsalpha}

}
\end{document}